\documentclass[10pt,reqno,final]{amsart} 
\usepackage{epsfig,amssymb,amsmath,mathtools}
\usepackage{amssymb,version,graphicx,fancybox,mathrsfs}
\usepackage{url,hyperref}
\hypersetup{colorlinks=true,citecolor=blue,linkcolor=red,anchorcolor=blue}
\usepackage[notcite,notref]{showkeys}
\usepackage{subfigure}
\usepackage{color}
\usepackage{stmaryrd}
\usepackage{multirow}
\usepackage{booktabs,siunitx,tablefootnote}
\usepackage{multicol,makecell}
\catcode`\@=11 \theoremstyle{plain}
\@addtoreset{equation}{section}   

\@addtoreset{figure}{section}
\renewcommand\thefigure{\thesection.\@arabic\c@figure}
\@addtoreset{table}{section}
\renewcommand\thetable{\thesection.\@arabic\c@table}
\newtheorem{thm}{\bf Theorem}

\newenvironment{theorem}{\begin{thm}} {\end{thm}}
\newtheorem{cor}{\bf Corollary}
\newtheorem{prop}{Proposition}[section]

\newenvironment{corollary}{\begin{cor}} {\end{cor}}
\newtheorem{lmm}{\bf Lemma}

\newenvironment{lemma}{\begin{lmm}}{\end{lmm}}
\theoremstyle{remark}
\newtheorem{rem}{\bf Remark}[section]
\theoremstyle{example}

\definecolor{ligreen}{rgb}{0.0, 0.3, 0.0}

\definecolor{darkblue}{rgb}{0.0, 0.0, 0.55}

\definecolor{anti-flashwhite}{rgb}{0.55, 0.57, 0.68}

\def \rd {{\rm d}}
\def \ri {{\rm i}}
\def \re {{\rm e}}
\newcommand{\bs}[1]{\boldsymbol{#1}}
\allowdisplaybreaks

\makeatletter
\newsavebox\myboxA
\newsavebox\myboxB
\newlength\mylenA

\newcommand*\xoverline[2][0.75]{%
    \sbox{\myboxA}{$\m@th#2$}%
    \setbox\myboxB\null
    \ht\myboxB=\ht\myboxA%
    \dp\myboxB=\dp\myboxA%
    \wd\myboxB=#1\wd\myboxA
    \sbox\myboxB{$\m@th\overline{\copy\myboxB}$}
    \setlength\mylenA{\the\wd\myboxA}
    \addtolength\mylenA{-\the\wd\myboxB}%
    \ifdim\wd\myboxB<\wd\myboxA%
       \rlap{\hskip 0.5\mylenA\usebox\myboxB}{\usebox\myboxA}%
    \else
        \hskip -0.5\mylenA\rlap{\usebox\myboxA}{\hskip 0.5\mylenA\usebox\myboxB}%
    \fi}
\makeatother

\begin{document}
\bibliographystyle{plain}

\title[Eigenvalue Analysis and Applications of LDPG methods]
{Eigenvalue Analysis and Applications of the Legendre Dual-Petrov-Galerkin Methods  for  Initial Value Problems}
\author[
	D.S. Kong, J. Shen, L.-L. Wang  $\&$\,  S.H. Xiang
	]{
		\;\; Desong Kong${}^{\dag}$,\;  Jie Shen${}^{\ddag}$,\; Li-Lian Wang${}^{\S}$ \; and\; Shuhuang Xiang${}^{\dag}$
		}
	\thanks{${}^{\dag}$School of Mathematics and Statistics,
		Central South University, Changsha 410083, China (desongkong@csu.edu.cn, xiangsh@csu.edu.cn).
    The work of S. Xiang and D. Kong is supported by the National Natural  Science Foundation of China (No. 12271528, No. 12001280). The first author is partially supported by the Fundamental Research Funds for the Central Universities of Central South University (No. 2020zzts031). \\
		\indent ${}^{\ddag}$Department of Mathematics, Purdue University, West Lafayette, IN 47907, USA (shen7@purdue.edu). The
		work of J.S. is supported in part by NSF DMS-1720442 and NFSC 11971407. \\
			\indent ${}^{\S}$Corresponding author. Division of Mathematical Sciences, School of Physical
		and Mathematical Sciences, Nanyang Technological University,
		637371, Singapore (lilian@ntu.edu.sg). The research of this author is partially supported by Singapore MOE AcRF Tier 1 Grant: MOE2021-T1-RG15/21.\\
		\indent The first author would like to acknowledge the support of the China Scholarship Council (CSC, No.\! 202106370101) for visiting NTU and working on this topic.
		}


\begin{abstract}
  In this paper, we show that the eigenvalues and eigenvectors of the spectral discretisation matrices resulted from  the
  Legendre dual-Petrov-Galerkin (LDPG) method for the $m$th-order initial value problem (IVP):
   $u^{(m)}(t)=\sigma u(t),\, t\in (-1,1)$ with constant $\sigma\not=0$ and usual initial conditions at $t=-1,$ are associated with the generalised Bessel polynomials (GBPs). The essential idea of the analysis is to properly construct the basis functions for the solution and its dual spaces so that
   the matrix of the $m$th derivative is an identity matrix, and  the mass matrix is then identical or approximately equals to the Jacobi matrix of the three-term recurrence of GBPs with specific integer parameters. This allows us to characterise the eigenvalue distributions and identify the eigenvectors. As a by-product,
   we are able to answer some open questions related to the very limited known results on the collocation method at Legendre points (studied in 1980s)
    for the first-order IVP, by reformulating it into a Petrov-Galerkin formulation. Moreover, we present two stable algorithms for computing zeros of the GBPs,  and develop a general space-time spectral method for  evolutionary  PDEs using either the matrix diagonalisation, which is restricted to a small number of unknowns in time due to the ill-conditioning but is fully parallel, or the QZ decomposition which is numerically stable for a large number of unknowns in time but involves sequential computations. We provide ample numerical results to demonstrate the high accuracy and robustness  of the space-time spectral methods for some interesting examples of linear and nonlinear wave problems.
\end{abstract}
\maketitle

\vspace*{-20pt}

\section{Introduction}\label{sec:int}
Spectral methods are a class of numerical methods which use global orthogonal polynomials/functions as basis functions, and have gained much popularity  due to its high accuracy for problems with   solutions that are smooth in suitable functional spaces \cite{Canuto1987Book,Canuto2006Book,Shen2011Book}.
However,  spectral methods are mostly used in  spatial discretisations, while  time discretisations are usually done with traditional approaches such as implicit-explicit schemes or  explicit Runge-Kutta methods which are of fixed-order convergence rates, thus creating a mismatch between  spectral accuracy in space and usually lower-order accuracy in time.
In practice, space-time spectral methods are attractive for problems with dynamics that require high resolution in both space and time, e.g., oscillatory wave propagations.

Several attempts have been made in developing  spectral methods in time. Tal-Ezer \cite{TalEzer1986SINUM,TalEzer1989SINUM} first studied Chebyshev spectral methods  in time for linear hyperbolic and parabolic problems.
Tang and Ma \cite{Tang2002Adv}
presented a Legendre-tau spectral method in time for parabolic  PDEs with periodic boundary conditions in space;
and later in \cite{Tang2007}  for  hyperbolic problems in a similar setting.
Shen and Wang \cite{Shen2007ANM} proposed a Fourierization space-time Legendre spectral method for parabolic equations.
Recently, Lui \cite{Lui2016} investigated  a space-time Legendre spectral collocation method for the heat equation.
Shen and Sheng \cite{Shen2019} developed a space-time dual-Petrov-Galerkin method for fractional (in time) subdiffusion equations and adopted a QZ decomposition technique to overcome the extreme ill-conditioning  when the eigen-decomposition is used.
However, a theoretical foundation for such methods is still lacking.
It is known that the eigenvalue distribution of the spectral differentiation matrices plays a central role in the stability and round-off error of the underlying spectral method.  While the eigenvalue distribution  of spectral approximation to second-order BVPs were well established \cite{Gottlieb1983SINUM,Weideman1988SINUM,Boulmezaoud2007JSC,Zhang2014},
very little is known  about  the eigenvalue distribution  of spectral approximation to IVPs;  particularly no results appear  available for higher-order IVPs.

Unlike the second-order operators for BVPs,
spectral approximations  to the  derivative operators of IVPs can be problematic.  In particular, the derivative matrices are usually non-normal, which typically require more stringent conditions for the stability.  We refer to  Gottlieb and Orszag  \cite{Gottlieb1977Book} for  insightful stability analysis of spectral approximations to  first-order hyperbolic problems, and
briefly  review  the existing  findings
(mostly  on asymptotic  eigenvalue analysis of first-order spectral discretisation matrices  in 1980s).
Dubiner \cite{Dubiner1987}  conducted an  asymptotic  analysis of  the spectral-tau method for a prototype problem:
{\em find $\lambda\in {\mathbb C}$ and $u_N\in {\mathbb P}_N={\rm span}\{ x^k: 0\le k\le N\}$ such that
\begin{equation}\label{diseigen}
u_N'(x)-\lambda u_N(x)= p_N(x),\;\;\; x\in (-1,1);\quad u_N(1)=0,
\end{equation}
where $p_N(x)$ is a polynomial of degree $N.$} In particular,  for the Jacobi-tau approximation:
$$
\int_{-1}^1 \big(u_N'-\lambda u_N\big)(x) v(x) (1-x)^\alpha (1+x)^{\beta} {\rm d}x =0,\quad \forall\,  v\in {\mathbb P}_{N-1},
$$
we have $p_N(x)=\tau P_N^{(\alpha,\beta)}(x)$
(with any constant $\tau\not=0$), in view of the orthogonality of Jacobi polynomials and $u_N'-\lambda u_N\in {\mathbb P}_N.$
Dubiner \cite{Dubiner1987} showed that  for $\beta>-1,$
\begin{equation*}\label{Dubinas}
\lambda_i^{N}=
 O(N^2),\;\;  {\rm if} \;\;\alpha\in (-1, 1];\quad
 \lambda_i^{N}=O(N), \;\; {\rm if} \;\;\alpha=0.
\end{equation*}
According to \cite{Csordas2005PAMS}, all eigenvalues lie in the left half-plane.   Following \cite{Dubiner1987},   Wang and Waleffe \cite{WangWaleffe2014a} further proved  the existence of unstable eigenvalues for $\alpha>1$ and the eigenvalues of the Legendre case are zeros of the
modified Bessel  function $k_N(z)$, drawn from the property   $k_N(z)= z^{-1}{\re^{-z}} \sum_{k=0}^N {P_N^{(k)}(1)}{z^{-k}}$
and its relation with the characteristic polynomial of \eqref{diseigen}.

Tal-Ezer \cite{MR867974}, and Trefethen and Trummer \cite{Trefethen1987SINUM} discovered  that eigenvalues of the spectral differentiation matrix of
the collocation methods for the first-order IVP at the  \emph{Legendre  points} (see \eqref{Trencol}) behave like  $\lambda_i^{N}=O(N),$ but
$\lambda_i^{N}=O(N^2)$ at Chebyshev and other  Jacobi points.
Accordingly, the allowable time-step size in the explicit time discretisation of the hyperbolic problem   is    $O(N^{-1}),$ rather than $O(N^{-2})$  for the Chebyshev and other Jacobi collocation methods.
However,
numerical evidences in  \cite{Trefethen1987SINUM}  indicated that the  Legendre collocation approximation
with $O(N^{-1})$ time-step constraint appeared only in theory, but  subject to an $O(N^{-2})$ restriction in practice.  Moreover,  the computation
of the eigenvalues is precision-dependent, and is extremely sensitive to perturbations/round-off errors, e.g.,   the use of  EISPACK  could produce reliable eigenvalues for $N\le 28$ with double-precision calculations \cite{Trefethen1987SINUM}.  Here,  we  demonstrate  in Figure \ref{fig:Deig-Leg} the  eigenvalues computed by \texttt{eig($\bs D$)}  in {\tt Matlab},  where the reference eigenvalues are computed by the algorithm in Pasquini \cite{Pasquini2000}  for computing the zeros of the GBP (see Theorem \ref{thm:eigMp} and Section \ref{sect6:zero}).   The numerical study in \cite{Trefethen1987SINUM} gave rise to  open questions, e.g.,  on the identifications of  eigenvalues and eigenvectors, and  rigorous proof of the exponential growth of the condition number of  the eigenvector matrix.

\begin{figure}[t]
  \centering
  \includegraphics[width=.33\textwidth]{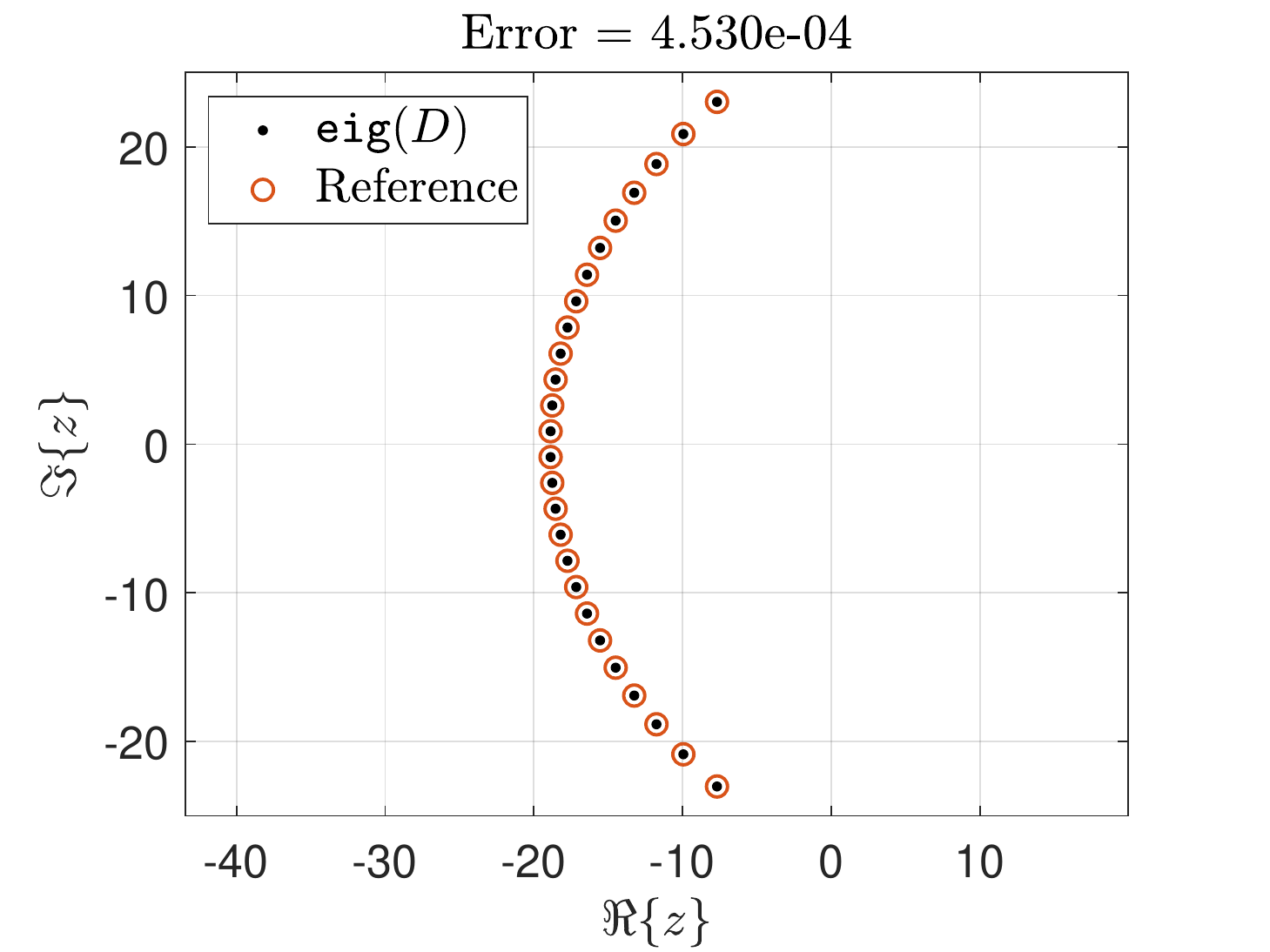} \quad
  \includegraphics[width=.33\textwidth]{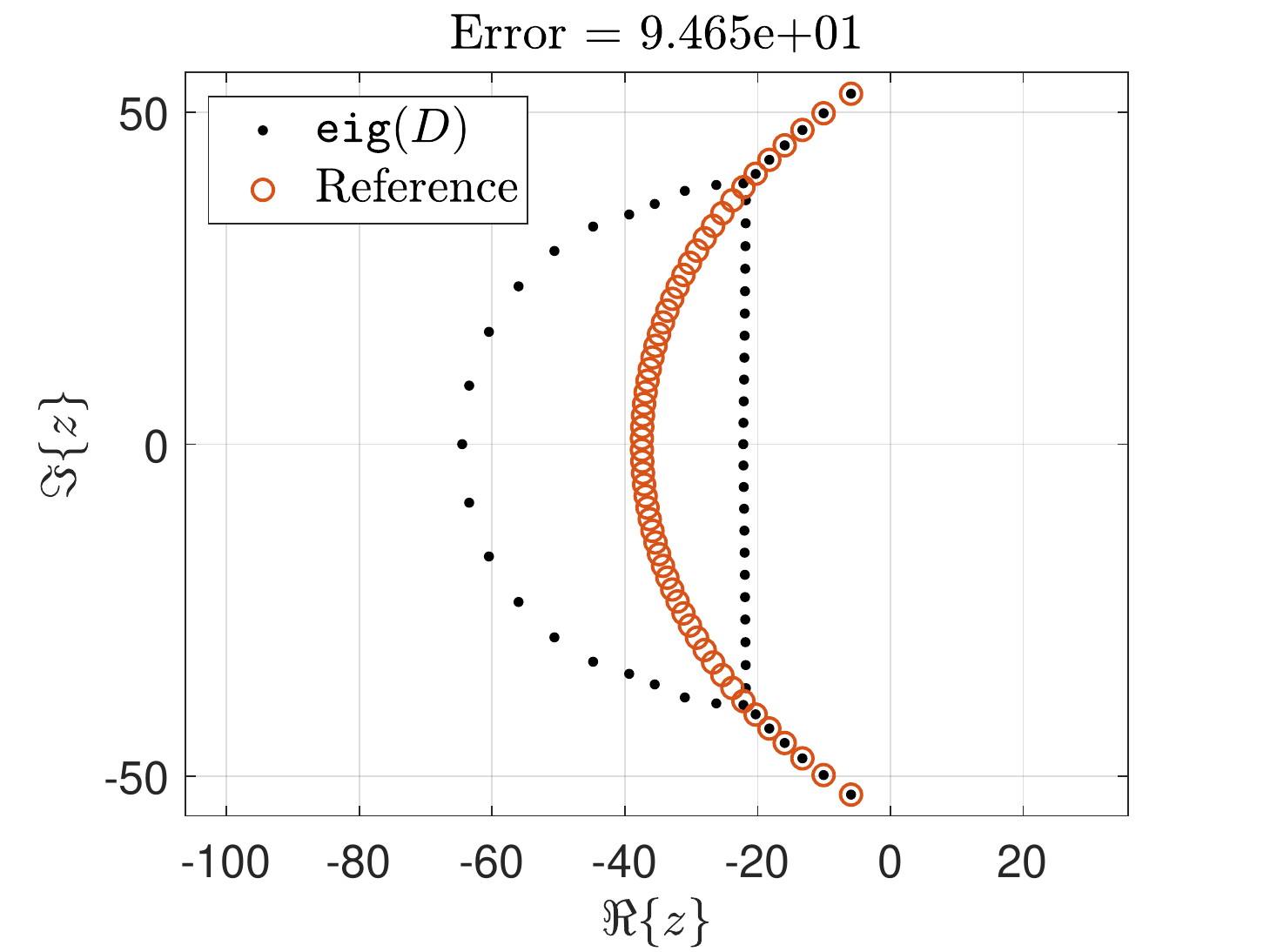}
  \caption{\small Eigenvalue distributions of  the first-order differentiation matrix  on Legendre points with $N = 28$ (left) and $N=56$ (right). The red circles denote the reference values, and the black dots denote the eigenvalues computed in  double precision.}
  \label{fig:Deig-Leg}
\end{figure}

The objective of this paper is to conduct rigorous eigenvalue analysis of the spectral matrices resulted from LDPG spectral methods for the first-order and general $m$th order IVPs.   The  argument of the analysis and main findings are summarized as follows.
\begin{itemize}
\item[(i)] We properly construct the  basis polynomials for the solution and dual approximation spaces so that the matrix of the highest derivative term is an identity matrix, but the mass matrix is a non-normal sparse matrix. This enables us to associate the mass matrix with the Jacobi matrix corresponding to the three-term recurrence relation of the Bessel or generalised Bessel polynomials, and then identify the eigenvalues and eigenvectors.

\smallskip
\item[(ii)] We reformulate the collocation method on Legendre points  (cf.\!\! \cite{MR867974,Trefethen1987SINUM}) as a Petrov-Galerkin formulation that allows us to identify eigen-pairs of the spectral differentiation matrix via the usual Bessel polynomials (BPs).
Then we can answer some open questions in \cite{Trefethen1987SINUM} and locate the eigenvalues
from  an approach different from \cite{WangWaleffe2014a}.
\smallskip

\item[(iii)] We introduce two techniques  to effectively deal with the (fully implicit) time-discretisation by the LDPG spectral method, which include (i) \emph{diagonalisation technique}; and (ii) \emph{QZ decomposition}.
Due to the severe  ill-conditioning of the eigenvector matrix, the diagonalisation technique, which involves the inverse of the eigenvector matrix,  is only numerically stable for small $N,$ so combine it with a multi-domain spectral method to  march in time sequentially. On the other hand, the QZ decomposition for such non-normal matrices is numerically stable for large $N,$ and we can solve the resulted linear systems simply by backward substitutions.
\end{itemize}

Finally, we remark that  there has been a continuing interest in solving evolutionary equations in a parallel-in-time (PinT) manner based on the diagonalisation technique initiated  by Maday and R{\o}nquist \cite{Maday2008}, see, for instance,  a recent work   \cite{Liu2021}  on a well-conditioned parallel-in-time algorithm  involving the diagonalisation  of
a nearly skew-symmetric matrix resulted with  a modified time-difference scheme.   Our LDPG method with diagonalisation can be directly applied in
a parallel-in-time manner, although in practice we also need to combine it with a multi-time-domain approach,
since we can only deal with a relatively short time  interval in parallel due to the ill-conditioning limitation mentioned above.

 The rest of this paper is organised as follows. In Section  \ref{sec:Bessel},
 we collect some related properties of the GBPs  and Legendre polynomials.
 In Section \ref{sec:IVP1st}, we conduct eigenvalue analysis for the LDPG spectral method for the first-order IVP, and reformulate the collocation method on Legendre points  \cite{MR867974,Trefethen1987SINUM} as a Petrov-Galerkin formulation to precisely characterise
 the eigenvalue distributions and the associated eigenvectors. We extend the analysis to the second-order IVPs in  Section \ref{sect4:2nd}, and higher-order IVPs in Section \ref{sect5:higher}.  We present two stable algorithms for computing zeros of the GBPs
 in  Section \ref{sect6:zero}. In Section \ref{Sect7:numerical}, we develop  a general framework for  space-time spectral methods, and apply it to solve  linear and nonlinear wave problems.
 We conclude the paper with some final remarks in Section 8.

\section{Properties of GBPs and Legendre polynomials}  
 \label{sec:Bessel}
\setcounter{equation}{0}
\setcounter{lmm}{0}
\setcounter{thm}{0}

The forthcoming eigenvalue analysis relies much  on  the GBPs, so we  collect below some relevant properties,
 which can be found in
\cite{Krall1949,Grosswald1978,Pasquini2000}.
We also review some properties of Legendre polynomials to be used later on.
Throughout this paper, let $\mathbb C$ (resp. $\mathbb R$) be the set all complex  (resp. real) numbers, and further  let
$\mathbb N=\{1,2,\cdots\}$ and  $\mathbb N_0=\{0,1,2,\cdots\}.$

\subsection{Generalised Bessel polynomials} For $\alpha,\beta\in {\mathbb C}, -\alpha\not \in {\mathbb N}_0$ and $\beta\not=0,$
the GBP,  denoted by $B_n^{(\alpha,\beta)}(z), z\in {\mathbb C}$, is a polynomial of degree $n$ that satisfies the second-order differential equation
\begin{equation}\label{gBP2ndODE2}
  z^2 y''(z) + (\alpha z+\beta)y'(z) = n(n+\alpha-1)y(z),\quad n\in {\mathbb N}_0,
\end{equation}
 which has the explicit  representation
\begin{equation}\label{gBPexp}
  B_n^{(\alpha, \beta)}(z) = \sum_{k=0}^{n} \binom{n}{k} \frac{\Gamma(n+k+\alpha-1)}{\Gamma(n+\alpha-1)} \Big(\frac{z}{\beta} \Big)^k.
\end{equation}
The GBPs can be generated by the three-term recurrence relation
\begin{equation}\label{eq:recur0}
\begin{dcases}
  B_{n+1}^{(\alpha,\beta)}(z) = \Big( a_n^{(\alpha)} \frac{z}{\beta} + b_n^{(\alpha)} \Big) B_n^{(\alpha,\beta)}(z) + c_n^{(\alpha)} B_{n-1}^{(\alpha,\beta)}(z),\quad n \geq 1, \\
  B_{0}^{(\alpha,\beta)}(z)=1,\quad B_{1}^{(\alpha,\beta)}(z)=1+\alpha\frac{z}\beta,
\end{dcases}
\end{equation}
where
\begin{equation}\label{coefabc}
\begin{split}
  & a_n^{(\alpha)} = \frac{(2n+\alpha)(2n+\alpha-1)}{n+\alpha-1}, \quad
  b_n^{(\alpha)} = \frac{(\alpha-2)(2n+\alpha-1)}{(n+\alpha-1)(2n+\alpha-2)}, \quad
  \\&c_n^{(\alpha)} = \frac{n(2n+\alpha)}{(n+\alpha-1)(2n+\alpha-2)}.
\end{split}
\end{equation}
Here the normalization $B_{n}^{(\alpha,\beta)}(0)=1$ is adopted.

The GBPs are orthogonal on the unit circle (see \cite[p. 30]{Grosswald1978}):
\begin{equation}\label{orthUnitC}
  \frac{1}{2\pi\ri}\int_{|z|=1} B_m^{(\alpha,\beta)}(z) B_n^{(\alpha,\beta)}(z) \rho^{(\alpha,\beta)}(z) \rd z
  = \gamma_n^{(\alpha, \beta)} \delta_{mn},
\end{equation}
where the weight function and the normalization constant are given by
 \begin{equation}\label{Weightrho}
 \begin{split}
 & \rho^{(\alpha,\beta)}(z):=\sum_{k=0}^\infty \frac{\Gamma(\alpha)}{\Gamma(k+\alpha-1)}\Big(\!-\frac \beta z\Big)^k={_1F_1}(1, \alpha-1; -\beta/z), \\
& \gamma_n^{(\alpha, \beta)} := \frac{(-1)^{n+1} \beta \Gamma(\alpha)n!}{(2n+\alpha-1) \Gamma(n+\alpha-1)}.
 \end{split}
 \end{equation}
 Here ${_1F_1}(\,)$ is the hypergeometric function.

Observe from \eqref{gBPexp}-\eqref{eq:recur0}  that the parameter $\beta$ plays a role as a scaling factor, and there holds
\begin{equation*}\label{Bnbeta}
 B_n^{(\alpha,-\beta)}(z)=B_n^{(\alpha,\beta)}(-z).
 \end{equation*}
  For simplicity, we  denote the GBP with $\beta=2$ by $B_n^{(\alpha)}(z):=B_n^{(\alpha,2)}(z).$
  When $\alpha=\beta=2$, the GBPs reduce to the classical Bessel polynomials (BPs), denoted by $B_n(z):=B_n^{(2,2)}(z).$
Accordingly, the recurrence relation \eqref{eq:recur0}-\eqref{coefabc}  becomes
\begin{equation}\label{3termBPdefn}
\begin{dcases}
B_{n+1}(z)= (2n+1)z B_n(z)+B_{n-1}(z),\quad n\ge 1,\\
B_0(z)=1,\quad B_1(z)=1+z.
\end{dcases}
\end{equation}
Moreover, the orthogonality \eqref{orthUnitC}-\eqref{Weightrho} simply reads
\begin{equation*}\label{orthA}
  \frac{1}{2\pi\ri}\int_{|z|=1} B_m(z) B_n(z) \re^{-2/z}  {\rm d}z
  = (-1)^{n+1} \frac{2}{2n+1} \delta_{mn}.
\end{equation*}
We summarize the most important  properties on zero distributions of GBPs  below.
\begin{theorem}\label{lem:Bnzero}  Let $\alpha\in {\mathbb R}$ and $n+\alpha-1>0.$ Then we have the following properties.
\begin{itemize}
\item[(i)] All zeros of $B_n^{(\alpha)}(z)$ are simple and conjugate of each other.
\smallskip
\item[(ii)] For  $\alpha\ge -1$ and $n\ge 2,$   all zeros of $B_n^{(\alpha)}(z)$ are in the open left half-plane.
\smallskip
\item[(iii)] For $n\ge 2,$ all zeros are located in the crescent-shaped region
  \begin{equation}\label{eq:annulBna0}
    {\mathcal R}_{n, \alpha}^- := \Big\{ z = \rho \re^{\ri\theta} \in \mathbb{C}: \frac{2}{2n+\alpha-\frac{2}{3}} < \rho\le  \frac{1-\cos\theta}{n+\alpha-1},\;\;
    \theta\in  (-\pi,-\Theta_{n,\alpha})\cup (\Theta_{n,\alpha},\pi] \Big\},
  \end{equation}
  where the equal sign can be attained only  $\theta=\pi,$ and
  \begin{equation}\label{thetanalpha}
  \Theta_{n,\alpha}:=\arccos \Big( \frac{-\alpha}{2n+\alpha-2} \Big).
  \end{equation}
 \end{itemize}
  \end{theorem}
\begin{proof}
According to   \cite[Theorem 1$'$]{Grosswald1978},
 all zeros of $B_n^{(\alpha)}(z)$ are simple. As $B_n^{(\alpha)}(z)$ is a polynomial of real coefficients, its zeros are conjugate of each other.  The second statement  is  presented in \cite[Theorem 4.3]{Bruin1981}. In fact, this result is sharp in the  sense that for
 $\alpha<-1$, the polynomial $B_n^{(\alpha)}(z)$ has at least one zero in the right half-plane.
Moreover, as stated in \cite[Theorem 5.1]{Bruin1981},  all zeros of  $B_n^{(\alpha)}(z)$ lie in the annulus
\begin{equation*}\label{Analpha}
 A(n,\alpha):=\Big\{ z\in {\mathbb C} :  \frac 2 {2n+\alpha-2/3}< |z|\le \frac{2}{n+\alpha-1}\Big\}.
\end{equation*}
More precisely, they are in the left half-annulus in view of (ii).  An improved upper enclosure is given by
\cite[Theorem 3.1]{Bruin1981}: all zeros of $B_n^{(\alpha)}(z)$ are confined in the cardioidal region
\begin{equation*}\label{CardioidR}
C(n, \alpha):=\Big\{z=r \mathrm{e}^{{\rm i} \theta} \in \mathbb{C}: 0<r<\frac{1-\cos \theta}{n+\alpha-1}\Big\} \cup\Big\{\frac{-2}{n+\alpha-1}\Big\}.
\end{equation*}
Moreover,  from \cite[Theorem \! 4.1]{Bruin1981}, we have that all zeros of $B_n^{(\alpha)}(z)$ with $n \ge 2$ are in the sector
  \begin{equation*}\label{Snalpha}
  S(n, \alpha) := \Big\{ z = \rho \re^{\ri \theta} \in \mathbb C :
    |\theta| > 
    \arccos \Big( \frac{-\alpha}{2n+\alpha-2} \Big),\;\; -\pi < \theta \le \pi \Big\}.
    \end{equation*}
Then we define
    $$
     {\mathcal R}_{n, \alpha}^- :=  A(n,\alpha)\cap C(n, \alpha)\cap S(n, \alpha),
    $$
    which leads to the statement (iii).
\end{proof}

\begin{rem}\label{nmBnalpha} \textit{It is also noteworthy of the following properties.
\begin{itemize}
\item[(a)] If $n\not =m,$ no zero of  $B_n^{(\alpha)}(z)$ can be a zero of $B_m^{(\alpha)}(z)$, see \cite[Theorem 1(b$'$)]{Grosswald1978}.
\item[(b)] For any fixed $\alpha$ and for odd $n,$ let $Z_{n,\alpha}$ be the unique (negative) real zero of  $B_n^{(\alpha)}(z).$
We find from \cite[Theorem \! 4.1]{Bruin1981} the asymptotic behaviour
\begin{equation}\label{realZnalpha}
Z_{n,\alpha}\approx -\frac 2{1.3254868n+ 1.00628995 \alpha-1.34983648+O((2n+\alpha-2)^{-1})},\quad n\gg 1.
\end{equation}
In particular, for the BP $B_n(z),$ its unique real zero for odd $n$ behaves like $Z_{n,2}=-\nu n^{-1}+O(n^{-2})$ with
$\nu\approx 1.50888\ldots.$
\end{itemize}
}
\end{rem}

\subsection{Legendre polynomials}
The Legendre  polynomials, denoted by $P_n(x)$, $x \in I:=(-1,1),$   are defined by the three-term recurrence relation (cf. \cite[p.\! 94]{Shen2011Book}):
\begin{equation*}
\begin{split}
  (n+1) P_{n+1}(x) = (2n+1) xP_n(x) - nP_{n-1}(x),\quad n \ge 1,
\end{split}
\end{equation*}
with $P_0(x) = 1$ and $P_1(x) = x$.
They satisfy the orthogonality
\begin{equation}\label{eq:PnOrtho}
  \int_{-1}^1 P_m(x) P_n(x) \rd x = \gamma_n \delta_{mn},\quad \gamma_n = \frac{2}{2n+1}.
\end{equation}
The derivatives of $P_n(x)$ satisfy \cite[p.\! 95]{Shen2011Book}:
\begin{equation}\label{eq:deriv12}
\begin{split}
  P_n'(x) &= \sum_{\substack{k=0 \\ k+n \mbox{ odd}}}^{n-1} (2k+1) P_{k}(x), \\
  P_{n}''(x) &= \sum_{\substack{k=0 \\ k+n \mbox{ even}}}^{n-2} (k+1/2)(n(n+1)-k(k+1)) P_{k+1}(x).
\end{split}
\end{equation}
A very useful property is
\begin{equation}\label{devfun}
(2n+1)P_n'(x)=P_{n+1}(x)-P_{n-1}(x),\quad n\ge 1.
\end{equation}
Moreover, we have the special values
\begin{equation}\label{eq:PnBV}
  P_n(\pm 1) = (\pm 1)^n,\quad P_n'(\pm 1) = \frac12  (\pm 1)^{n-1} n(n+1).
\end{equation}

\section{Eigenvalue analysis of  Petrov-Galerkin methods  for first-order IVPs}\label{sec:IVP1st}
\setcounter{equation}{0}
\setcounter{lmm}{0}
\setcounter{thm}{0}

In this section, we conduct  eigenvalue analysis of the LDPG spectral method for the first-order IVP, and then precisely characterise the eigenvalue distributions
 of the collocation differentiation matrix  at Legendre points in  \cite{MR867974,Trefethen1987SINUM}
by reformulating it as a Petrov-Galerkin form.
\subsection{Legendre dual Petrov-Galerkin method}\label{subsec:d1st}  To fix the idea, we consider the  model problem
\begin{equation}\label{eq:IVP1st}
   u'(t) =\sigma u(t),\;\;\; t \in I:=(-1,1); \quad u(-1)= u_0,
\end{equation}
for   given constants $\sigma\not=0$ and $u_0\not=0.$ 
It is   a  prototype problem for testing the stability of various numerical schemes, and    a good example to involve both the derivative and mass matrices.

We adopt the LDPG scheme (cf.\! \cite{Shen2003SINUM,Shen2007CMAME}) and  define the dual approximation spaces
 \begin{equation*}\label{dualspaces}
  {}_0{\mathbb P_N}:=\{ \phi \in \mathbb{P}_N : \phi(-1)=0 \},\quad {}^0{\mathbb P_N}:=\{ \psi \in \mathbb{P}_N : \psi(1)=0 \}.
 \end{equation*}
We seek $u_N=u_0+ v_N\in {\mathbb P_N}$ with $v_N\in {}_0{\mathbb P}_N$
such that
\begin{equation}\label{eq:dPGalM}
  (v_N', \psi) -\sigma (v_N, \psi)=\sigma (u_0, \psi), \quad  \forall\, \psi \in {}^0{\mathbb P}_N.
\end{equation}
Choose the basis functions for the trial function space ${}_0{\mathbb P_N}$ as
\begin{equation}\label{eq:dPetbas}
  \phi_k(t) = \frac{k+1}{\sqrt{2}}(P_{k}(t)+P_{k+1}(t)),\quad  0 \leq k \leq N-1,
\end{equation}
and  for the  test function space ${}^0{\mathbb P_N}$ as
\begin{equation}\label{eq:dPetbas2}
  \phi_j^*(t) = \frac{1}{\sqrt{2}(j+1)}(P_{j}(t)-P_{j+1}(t)),\quad  0 \leq j \leq N-1.
\end{equation}
Write and denote
\begin{equation*}\label{eq:dnotation}
  v_N(t) = \sum_{k=0}^{N-1} \tilde v_k \phi_k(t),\quad
  \tilde{ \bs v} = (\tilde v_0, \cdots, \tilde v_{N-1})^\top.
\end{equation*}
The corresponding  linear system of \eqref{eq:dPGalM} reads
\begin{equation}\label{eq:dodesystem}
  \big({\bs I}_{N}  - \sigma \bs M\big) \tilde{ \bs v} = \sqrt{2}\sigma u_0\bs e_1,
\end{equation}
where $\bs I_{\!N}$  is the ${N \times N}$ identity matrix and $\bs e_1$ is its first column. Note that with the above choice, we can  verify readily from  \eqref{eq:PnOrtho}-\eqref{eq:deriv12} that  $(\phi_k', \phi_j^*) = \delta_{jk},$ and
the mass matrix $\bs M \in \mathbb R^{N \times N}$ is a tri-diagonal (but non-symmetric) matrix with  nonzero entries given by
\begin{equation}\label{eq:Mat1st}
  \bs M_{jk} = (\phi_k, \phi_j^*) =
  \begin{dcases}
    \frac{j}{(j+1)(2j+1)}, & k=j-1,\;\;\; 1 \le j \le N-1, \\
    \frac{1}{2j+1} - \frac{1}{2j+3}, & k=j,\;\; 0 \le j \le N-1, \\
    -\frac{j+2}{(j+1)(2j+3)}, & k=j+1,\;\; 0 \le j \le N-2.
  \end{dcases}
\end{equation}

Notably,  the matrix $-\bs M$ is identical to the Jacobi matrix associated with the three-term recurrence relation
of $B_N^{(3)}(z)$, so we can  characterise the eigenvalues and eigenvectors of $\bs M$ as follows.
\begin{thm}\label{thm:eigMd}
Let $\{\lambda_j := \lambda_{N,j}\}_{j=1}^N$ be the eigenvalues of the tri-diagonal matrix $\bs M$ with $N \ge 2$ given in \eqref{eq:Mat1st}. Then we have the following properties of the eigenvalues.
\begin{itemize}
\item [(i)] The eigenvalues  $\big\{\lambda_{j}=-z_{j}^{(3)}\big\}_{j=1}^N$ with $\big\{z_{j}^{(3)}:=z_{N,j}^{(3)}\big\}$
 being zeros of the generalised Bessel polynomial $B_N^{(3)}(z)$ defined in
\eqref{eq:recur0} with $\alpha=3,\beta=2$, so the eigenvalues  are all simple and  conjugate of each other. For odd $N,$ its unique real
eigenvalue behaves like $-Z_{N,3}$ in \eqref{realZnalpha}.
\smallskip
\item [(ii)] All the eigenvalues lie in the open right half-plane, and are located in a crescent-shaped region:
\begin{equation}\label{annula_d}
  \lambda_j \in  \mathcal{R}_{N,3}^+ := \Big\{ z = \rho \re^{\ri \theta} \in \mathbb{C}: \frac{1}{N +  7/ 6} < \rho \le \frac{1+\cos\theta}{N+2},\;\;
    |\theta| <\pi-\Theta_{N,3} \Big\},
\end{equation}
for $1 \le j \le N,$ where $\Theta_{N,3}$ is given in  \eqref{thetanalpha} and the equal sign can be only attainable at $\theta=0$.
\end{itemize}
The corresponding eigenvectors are
\begin{equation}\label{djeigen}
  \bs v_j=\frac{\bs b_j}{|\bs b_j|}, \quad \bs b_j:= \big(B_0^{(3)}(-\lambda_j), B_1^{(3)}(-\lambda_j), \cdots, B_{N-1}^{(3)}(-\lambda_j)\big)^\top,\;\; 1 \le j \le N.
\end{equation}
\end{thm}
\begin{proof} From the three-term recurrence relation   \eqref{eq:recur0} with $\alpha=3, \beta = -2,$ we obtain
  \begin{equation}\label{3termBP20-2}
    \begin{dcases}
      -zB_0^{(3)}(z)  =   \frac23 B_0^{(3)}(z) - \frac23 B_1^{(3)}(z),\\
      -zB_j^{(3)}(z)   =   \frac{j}{(j+1)(2j+1)} B_{j-1}^{(3)}(z) + \frac{2}{(2j+1)(2j+3)} B_{j}^{(3)}(z)\\
       \qquad \qquad\quad  - \frac{j+2}{(j+1)(2j+3)} B_{j+1}^{(3)}(z), \quad 1 \le j \le N-1.
    \end{dcases}
  \end{equation}
 Then  we can rewrite \eqref{3termBP20-2} as  the matrix form and find from \eqref{eq:Mat1st} that
  \begin{equation}\label{notationBe}
    -\!z\bs b(z) = \bs M \bs b(z) - \frac{N+1}{N(2N+1)} B_{N}^{(3)}(z) \bs e_N,
  \end{equation}
  where
  $$\bs b(z):=\big(B_0^{(3)}(z), B_1^{(3)}(z), \cdots, B_{N-1}^{(3)}(z)\big)^\top,\quad \bs e_N:=(0, \cdots, 0, 1)^\top\in {\mathbb R}^N.$$
  Taking  $z=z_j^{(3)}=-\lambda_j$ in \eqref{notationBe},
  we derive immediately that $\lambda_j \bs v_j=\bs M \bs v_j$ for $1\le j\le N.$
Then the properties and distribution of the eigenvalues stated in (i)-(ii) are  direct consequences of  Theorem \ref{lem:Bnzero} with $\alpha=3$.
In view of Remark \ref{nmBnalpha}, the unique real eigenvalue of $\bs M$ for odd order $N$ has the asymptotic behaviour as in \eqref{realZnalpha}.
\end{proof}

We depict in  Figure \ref{fig:lamj51}  (left)  the distributions of the eigenvalues $\{\lambda_j\}_{j=1}^N$ of $\bs M$ with $N=51$, where we highlight
the crescent-shaped region $\mathcal R_{N, 3}^+$ in  \eqref{annula_d}, and also plot
the cardioid curve ${\mathcal C}_{N,3}: \rho=\frac{1+\cos\theta}{N+2}$ for $|\theta|\le \pi.$
It is evident from \eqref{annula_d} that
\begin{equation*}\label{evident33}
\frac{1}{N +  7/ 6} < |\lambda_j|=|\lambda_{N,j}|\le \frac 2 {N+2},\quad 1\le j\le N,
\end{equation*}
so   all  $\{|\lambda_{N,j}|\}_{j=1}^N$ behave like $O(N^{-1}).$ Indeed, we observe from  Figure \ref{fig:lamj51}  (middle) that  the ``radii''  of the ``semi-circles''
decay with respect to $N.$  In Figure \ref{fig:lamj51}  (right), we demonstrate the behaviour of the eigenvalues with the smallest and largest magnitudes from which we speculate
\begin{equation*}\label{evident330}
\min_j |\lambda_{j}|\simeq \frac {C_{\rm min}} {N},\quad \max_j |\lambda_{j}|\simeq \frac {C_{\rm max}} {N}, \quad
\end{equation*}
and $C_{\rm min}\to 1, C_{\rm max}\to 1.5$ for $N\gg 1.$

\begin{figure}[htp]
  \begin{center}
  \includegraphics[width=.34\textwidth]{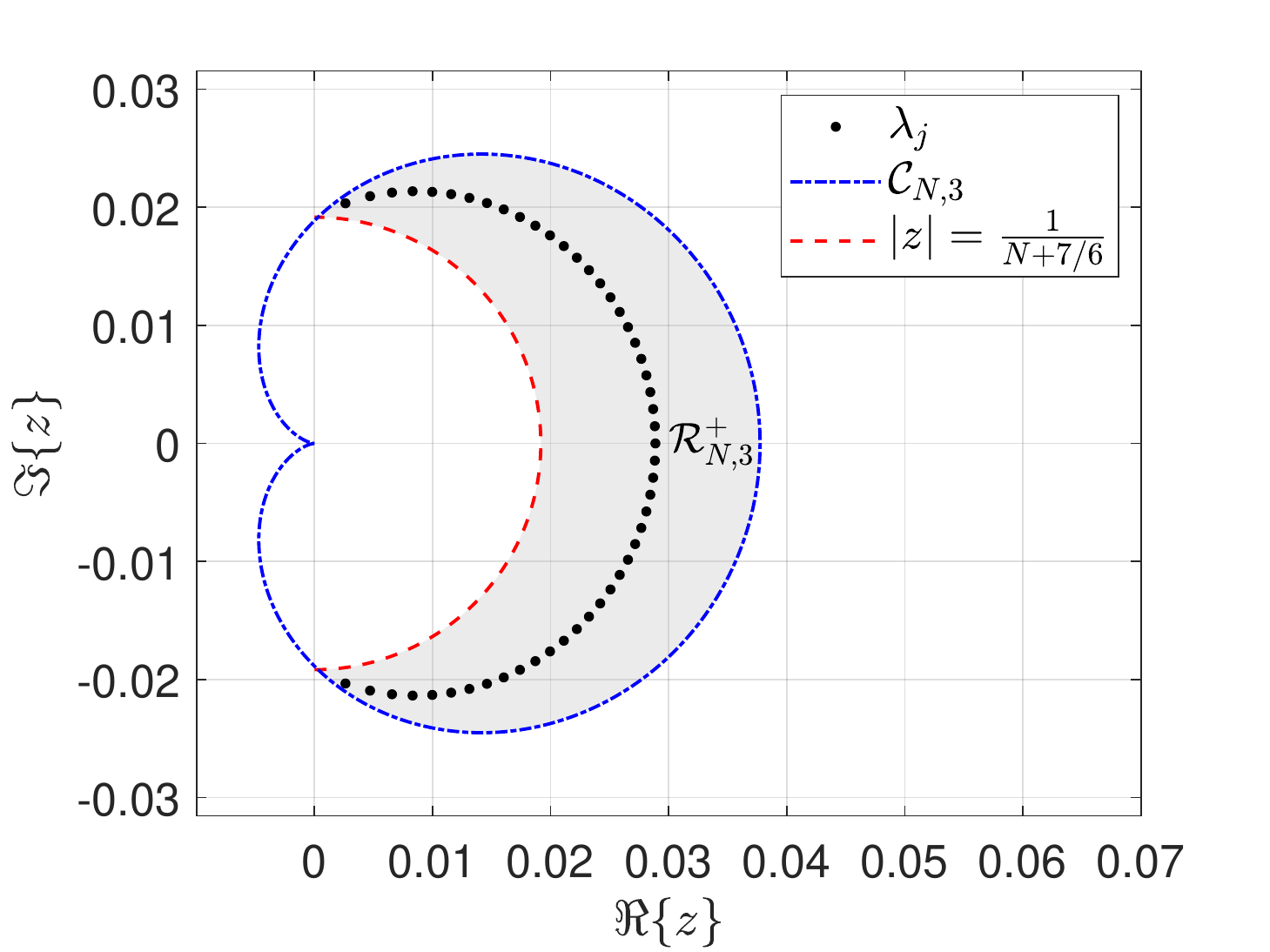}\hspace*{-10pt}
  \includegraphics[width=.34\textwidth]{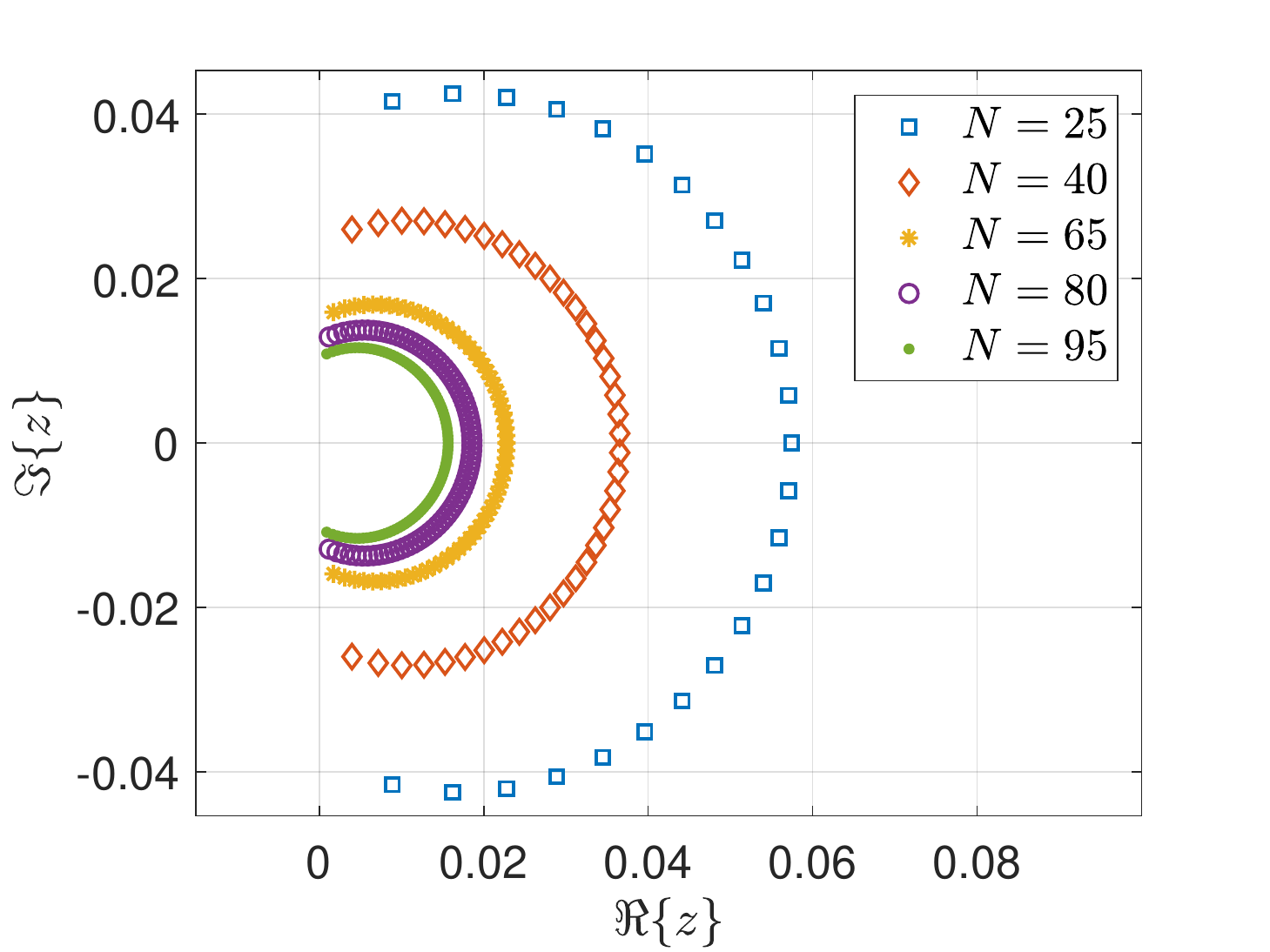} \hspace*{-10pt}
  \includegraphics[width=.34\textwidth]{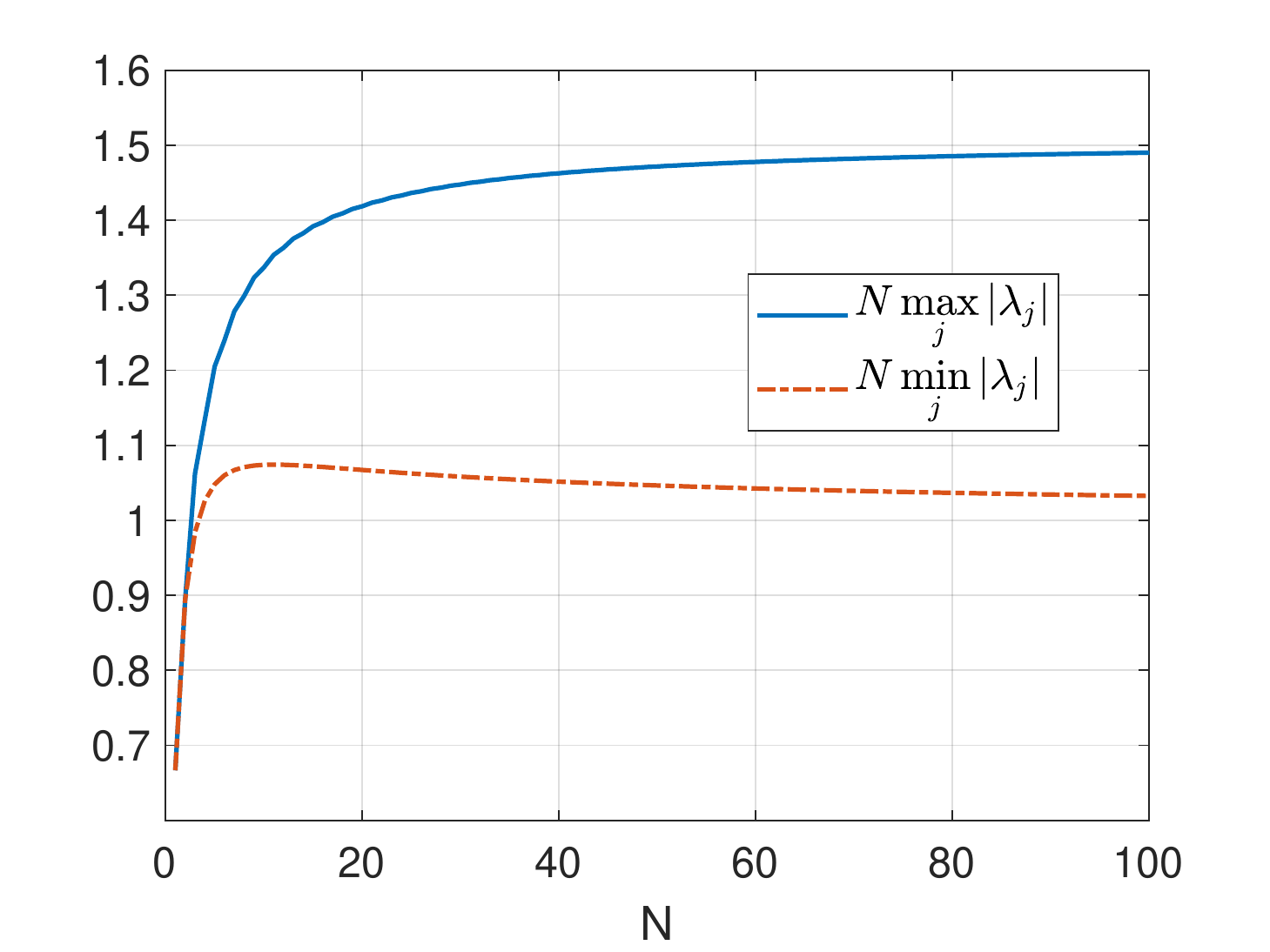}
  \vspace*{-10pt}
  \caption{\small Eigenvalues of $\bs M$ in  \eqref{eq:Mat1st}. Left:  Distribution of $\lambda_j=\lambda_{N,j}$ with $N=51$ in the crescent-shaped region (left).
  Middle:  Distributions of the eigenvalues with various $N$.  Right: Behaviour of the eigenvalues with the smallest and largest magnitudes for $1\le N\le 100.$}
  \end{center}
  \label{fig:lamj51}
\end{figure}

In the above derivations, we  chose the basis functions in \eqref{eq:dPetbas}-\eqref{eq:dPetbas2} with proper coefficients so that the derivative matrix in
the linear system \eqref{eq:dodesystem} became an identity matrix, but as a matter of fact, the eigenvalues are  independent of the choice of basis functions.  More precisely, consider any two sets of bases:
$$
 {}_0{\mathbb P_N}={\rm span}\big\{\psi_k\,:\,0\le k\le N-1\big\}, \quad {}^0{\mathbb P_N}={\rm span}\big\{\psi_k^*\,:\,0\le k\le N-1\big\}.
$$
Let $\bs P,\bs Q$ be the corresponding transformation matrices such that
$$
\psi_k(t)=\sum_{j=0}^{N-1} \bs P_{jk} \phi_j(t),\quad  \psi_k^*(t)=\sum_{j=0}^{N-1} \bs Q_{kj} \phi_j^*(t).
$$
Further, let $\widehat{\bs S}, \widehat {\bs M}$ be the related derivative and mass matrices, that is, $\widehat{\bs S}_{jk}=(\psi_k', \psi_j^*)$ and $\widehat {\bs M}_{jk}=(\psi_k, \psi_j^*).$
It can be shown by  direct manipulations the following results, so we leave its proof to the interested readers.
\begin{corollary}\label{cor:newbas}
The eigenvalues  $\{\lambda_j := \lambda_{N,j}\}_{j=1}^N$ of the matrix $\bs M$  in \eqref{eq:Mat1st} are also the eigenvalues of the generalised eigenvalue problem:
\begin{equation*}\label{geneign}
\widehat{\bs S} \hat{\bs v}=\lambda  \widehat{\bs M} \hat{\bs v},
\end{equation*}
where the   corresponding eigenvectors are
  \begin{equation*}\label{eq:eigpq}
    \hat{\bs v}_j =  \frac{(\bs Q \bs P)^{-1} \bs v_j}{| (\bs Q \bs P)^{-1} \bs v_j|}, \quad 1 \le j \le N,
  \end{equation*}
  with  $\{\bs v_j\}_{j=1}^N$ being the eigenvectors of  $\bs M$ given in \eqref{djeigen}.
\end{corollary}

\subsection{Collocation scheme at Legendre points}
Trefethen and Trummer \cite{Trefethen1987SINUM}  studied the collocation scheme for hyperbolic problems on the nodes consisting
of $N$  Legendre-Gauss points and an endpoint (to impose the underlying one-sided boundary condition).
In what follows, we employ this scheme to the model problem \eqref{eq:IVP1st}, and show that it can be reformulated into a Petrov-Galerkin scheme.
We can then characterise the eigenvalues of the spectral differentiation  matrix through  zeros of the BP.

Let $t_0=-1$ and  $t_1 < t_2 < \cdots < t_N$ be the Legendre-Gauss points, i.e., zeros of $P_N(t)$, and  let $\{h_j\}_{j=0}^N\subseteq {\mathbb P}_N$ be the corresponding Lagrange interpolating basis polynomials such that  $h_j(t_i)=\delta_{ij}.$
The collocation scheme in \cite{Trefethen1987SINUM}  adapted to \eqref{eq:IVP1st} is to find $u_N \in {\mathbb P_N}$ such that
\begin{equation}\label{Trencol}
u_N'(t_j)= \sigma u_N(t_j),\;\; 1 \le j \le N;\quad
   u_N(t_0)=u_{0}.
\end{equation}
Since  ${}_0{\mathbb P_N}={\rm span}\{h_j: 1\le j\le N\},$ we  write and denote
\begin{equation*}\label{uNxt}
u_N(t)=\sum_{j=0}^N u_N(t_j)h_j(t), \quad \bs u=(u_N(t_1),\cdots, u_N(t_N))^\top.
\end{equation*}
Then we obtain the linear system
\begin{equation}\label{odesystem0}
 \big( \bs D - \sigma \bs I_N \big)\bs u = - u_0 \bs h_0',
\end{equation}
where $\bs h_0' = (h_0'(t_1), \cdots, h_0'(t_N))^\top,$ and the spectral differentiation matrix $\bs D\in {\mathbb R}^{N\times N}$ with the entries ${\bs D}_{ij}=h_j'(t_i)$ for $1\le i,j\le N.$
One verifies readily that
$$
h_0(t) = (-1)^N P_N(t),\quad  h_j(t)=\frac{1+t}{1+t_j}l_j(t),\quad 1\le j\le N,
$$
where $\{l_j(t)\}_{j=1}^N\subseteq {\mathbb P}_{N-1}$ are the Lagrange interpolating basis polynomials at the Legendre-Gauss points $\{t_j\}_{j=1}^N.$ Thus, we can compute the entries of $\bs D$ via
\begin{equation}\label{Colentries}
 {\bs D}_{ij}=h_j'(t_i)=\frac 1 {1+t_j}\delta_{ij}+\frac{1+t_i}{1+t_j}l_j'(t_i),\quad 1\le i,j\le N,
\end{equation}
where $\{l_j'(t_i)\}$ can be computed by  the formulas in \cite[Ch.\! 3]{Shen2011Book}.

To facilitate the eigenvalue analysis of  $\bs D$, we reformulate  \eqref{Trencol} as a pseudospectral scheme.
Let $\{\omega_j\}_{j=1}^N$ be the Legendre-Gauss quadrature weights corresponding to $\{t_j\}_{j=1}^N,$ and
denote the induced discrete inner product by $\langle \cdot,\cdot \rangle_N.$ Then we obtain from  \eqref{Trencol} that
\begin{equation}\label{decaseA}
\langle u_N', \psi \rangle_N = \sigma \langle  u_N,\psi \rangle_N,\;\; u_N(-1) = u_0,\;\; \forall \psi\in {\mathbb P}_{N-1},
\end{equation}
As the Legendre-Gauss quadrature rule has a degree of precision $2N-1,$  the scheme \eqref{decaseA} is equivalent to  the Petrov-Galerkin scheme: find $u_N = u_0 + v_N \in {\mathbb P_N}$ with $v_N \in {_0\mathbb P_N}$ such that
\begin{equation}\label{decaseAB0}
(v_N',\psi) - \sigma(v_N,\psi) = \sigma (u_0, \psi),\;\;  \forall \psi\in {\mathbb P}_{N-1}.
\end{equation}
It is essential to choose the basis functions for ${}_0{\mathbb P_N}$ as
\begin{align}\label{eq:Petbas}
  \phi_0(t) = \frac{1}{\sqrt{2}}(t+1),\quad \phi_k(t) = \frac{1}{\sqrt{2}}(P_{k+1}(t)-P_{k-1}(t)),\quad 1\leq k\leq N-1,
\end{align}
and choose $\{\psi_j(t)=\frac{1}{\sqrt 2}P_j(t)\}_{j=0}^{N-1}$ as the basis polynomials for  the test function space.
One verifies from \eqref{eq:PnOrtho} and  \eqref{devfun}-\eqref{eq:PnBV} that  $(\phi_k',\psi_j)=\delta_{jk}$ and
\begin{equation}\label{diagmatrix}
\xoverline{\bs M}_{jk} =(\phi_k,\psi_j)=
  \begin{dcases}
    1, & k=j=0, \\
    -\frac{1}{2j+1}, & k = j+1, \\
    \frac{1}{2j+1}, & k = j-1,\\
    0, & {\rm otherwise}.
  \end{dcases}
\end{equation}
Then the  system corresponding to \eqref{decaseAB0} reads
\begin{equation}\label{odesystem10}
 \big(\bs I_N  - \sigma \xoverline{\bs M} \big) \bar{\bs v} = \sqrt{2} \sigma u_0 \bs e_1,
\end{equation}
where
\begin{equation*}\label{notationB}
v_N(t)  =\sum_{k=0}^{N-1} \bar v_k \phi_k(t), \quad \bar{\bs v} =(\bar v_0,\cdots, \bar v_{N-1})^\top.
\end{equation*}
\begin{lemma}\label{lem:RelsMC}
  The spectral differentiation matrix  $\bs D$ in \eqref{odesystem0} is similar to the inverse $\xoverline{\bs M}^{-1},$ i.e.,
  \begin{equation*}\label{DsDs}
  \bs D= \bs \Phi\,\xoverline {\bs M}^{-1}\,\bs \Phi^{-1},
  \end{equation*}
   where  $\bs \Phi\in {\mathbb R}^{N\times N}$ has the entries $\{\bs \Phi_{ik}=\phi_k(t_i)\}_{1\le i\le N}^{0\le k\le N-1}$ and
  the basis $\{\phi_k\}$ is given in \eqref{eq:Petbas}.
\end{lemma}
\begin{proof} We first show that
\begin{equation}\label{dphicase}
\bs D \bs \Phi=  \bs \Phi',\quad \bs \Phi'_{ik}= \phi_k'(t_i).
\end{equation}
For any $p\in {}_0{\mathbb P_N},$ we have the exact differentiation
$$
p'(t)=\sum_{j=1}^N p(t_j) h_j'(t), \;\;\; {\rm so} \;\;\; \bs p'=\bs D \bs p,
$$
where $\bs p=(p(t_1),\cdots, p(t_N))^\top$ and $\bs p'=(p'(t_1),\cdots, p'(t_N))^\top.$ Taking $p$ to be columns of $\bs \Phi,$ yields
$$
 \bs \Phi'=(\bs \phi_0',\cdots, \bs \phi_{N-1}')=\bs D (\bs \phi_0,\cdots, \bs \phi_{N-1})= \bs D \bs \Phi.
$$
We now introduce the matrix related to the dual basis functions: $\bs \Psi \in \mathbb R^{N \times N}$ with $\bs \Psi_{jk} = \psi_{j}(t_k) \omega_k$.
One verified readily from \eqref{decaseA} that
$$(\bs \Psi \bs \Phi)_{jk} = \sum_{l=0}^{N-1} \bs \Psi_{jl} \bs \Phi_{lk}  = \sum_{l=0}^{N-1} \psi_{j}(t_l)
\phi_k(t_l) \omega_l =\langle \phi_k,\psi_j   \rangle_N=(\phi_k,\psi_j)
=\xoverline {\bs M}_{jk},$$
and
 $$(\bs \Psi \bs \Phi')_{jk} =\langle \phi_k',\psi_j   \rangle_N=(\phi_k',\psi_j)=\delta_{jk},$$
so we have
\begin{equation}\label{PsiPhiA}
\bs \Psi \bs \Phi=\xoverline {\bs M},  \quad \bs \Psi \bs \Phi' = \bs I_N.
\end{equation}
Then  from \eqref{dphicase}-\eqref{PsiPhiA}, we obtain immediately that
\begin{equation*}
  \bs D = \bs \Phi'\, \bs \Phi^{-1} = \bs \Psi^{-1}\, \bs \Phi^{-1} = \bs \Phi\, \xoverline {\bs M}^{-1}\, \bs \Phi^{-1}.
\end{equation*}
This ends the proof.
\end{proof}

Remarkably, we can show that the eigenvalues of  the matrix $\xoverline{\bs M}$ are zeros of the classical BP, which together with
Lemma \ref{lem:RelsMC} allows  us to characterise the distribution of eigenvalues of $\bs D$.
\begin{thm}\label{thm:eigMp} Let $\{\bar \lambda_j:=\bar \lambda_{N,j}\}_{j=1}^N$ be the eigenvalues of the matrix $\xoverline{\bs M}$ with $N\ge 2.$ 
\begin{itemize}
\item [(i)] The eigenvalues $\{\bar \lambda_{j}=-z_{j}\}_{j=1}^N$ with $\big\{z_{j}:=z_{N,j}\big\}$ being zeros of the Bessel polynomial $B_N(x)$ defined in \eqref{3termBPdefn}, so the eigenvalues are all simple and  conjugate of each other. For odd $N,$ its unique real eigenvalue behaves like $-Z_{N,2}$ in \eqref{realZnalpha}.
\medskip
\item [(ii)] All the eigenvalues are in the open right half-plane and located in crescent-shaped region:
    \begin{equation*}\label{annula}
     \bar \lambda_j \in  \mathcal{R}_{N}^+ := \left\{ z = \rho \re^{\ri \theta} \in \mathbb{C}: \frac{1}{N +  2/ 3} < \rho \le \frac{1+\cos\theta}{N+1},\;\; |\theta| < \pi - \Theta_{N, 2} \right\},
    \end{equation*}
    for $1 \le j \le N$, where $\Theta_{N, 2}$ is given in \eqref{thetanalpha} and the equal sign can be only attainable at $\theta = 0.$
\end{itemize}
The corresponding eigenvector are
\begin{equation}\label{bjeigen}
\bs {\bar v}_j = \frac{\bs {\bar b}_j}{|\bs {\bar b}_j|},\quad
\bs {\bar b}_j=\big(B_0(-\bar\lambda_j), B_1(-\bar\lambda_j), \cdots, B_{N-1}(-\bar\lambda_j)\big)^\top,\quad 1\le j\le N.
\end{equation}
\end{thm}
\begin{proof}   We rewrite  the three-term recurrence relation
 \eqref{eq:recur0} as
  \begin{equation}\label{3termBP20}
    \begin{dcases}
          -zB_0(z) = B_0(z) - B_1(z),\\
      -zB_j(z) = \frac{1}{2j+1} B_{j-1}(z) - \frac{1}{2j+1} B_{j+1}(z),\;\;\; 1\le j\le N-1,
    \end{dcases}
  \end{equation}
and denote
\begin{equation*}\label{bmatrix}
 \bar {\bs b}(z):= (B_0(z), B_1(z), \cdots, B_{N-1}(z))^\top,\quad \bs e_N:=(0, \cdots, 0, 1)^\top\in {\mathbb R}^N.
 \end{equation*}
 In view of  \eqref{diagmatrix},   we can reformulate \eqref{3termBP20} as  the matrix form
  \begin{equation}\label{notationBe2}
    -z \bs {\bar b}(z) = \xoverline{\bs M} \bs {\bar b}(z) - \frac{1}{2N-1} B_{N}^{(2)}(z) \bs e_N.
  \end{equation}
  Taking $z = z_j = -\bar\lambda_j$ in \eqref{notationBe2}, we obtain immediately that $\bar\lambda_j \bs v_j = \xoverline{\bs M} \bs v_j$ for $1 \le j \le N$ with  the corresponding eigenvectors $\{\bs {\bar b}_j:=\bs {\bar b}(-\bar\lambda_j)\}$ given in \eqref{bjeigen}.
  With this, we claim these properties from  Theorem \ref{lem:Bnzero} and Remark \ref{nmBnalpha} with $\alpha = 2$.
\end{proof}

In Figure \ref{fig:lamjMD51A} (left), we plot the  region $\mathcal{R}_{N}^+$ with $N=51$ confined by the cardioid curve:
 \begin{equation*}\label{eq:cardBnaneg}
    \mathcal{C}_{N}^+= \Big\{ z = \rho \re^{\ri\theta} \in \mathbb{C}: \rho = \frac{1+\cos\theta}{N+1},\;\;  -\pi <  \theta \le\pi \Big\},
  \end{equation*}
  and the semi-circle: $|z|={1}/(N +  2/ 3)$ with $\theta\in [-\pi/2,\pi/2].$ Like Figure \ref{fig:lamj51}, we also illustrate the distributions of the eigenvalues for various $N,$ and examine the behaviour of the eigenvalues with the smallest and largest magnitudes, from which we observe similar asymptotic properties.
\begin{figure}[htp]
  \centering
  \includegraphics[width=.34\textwidth]{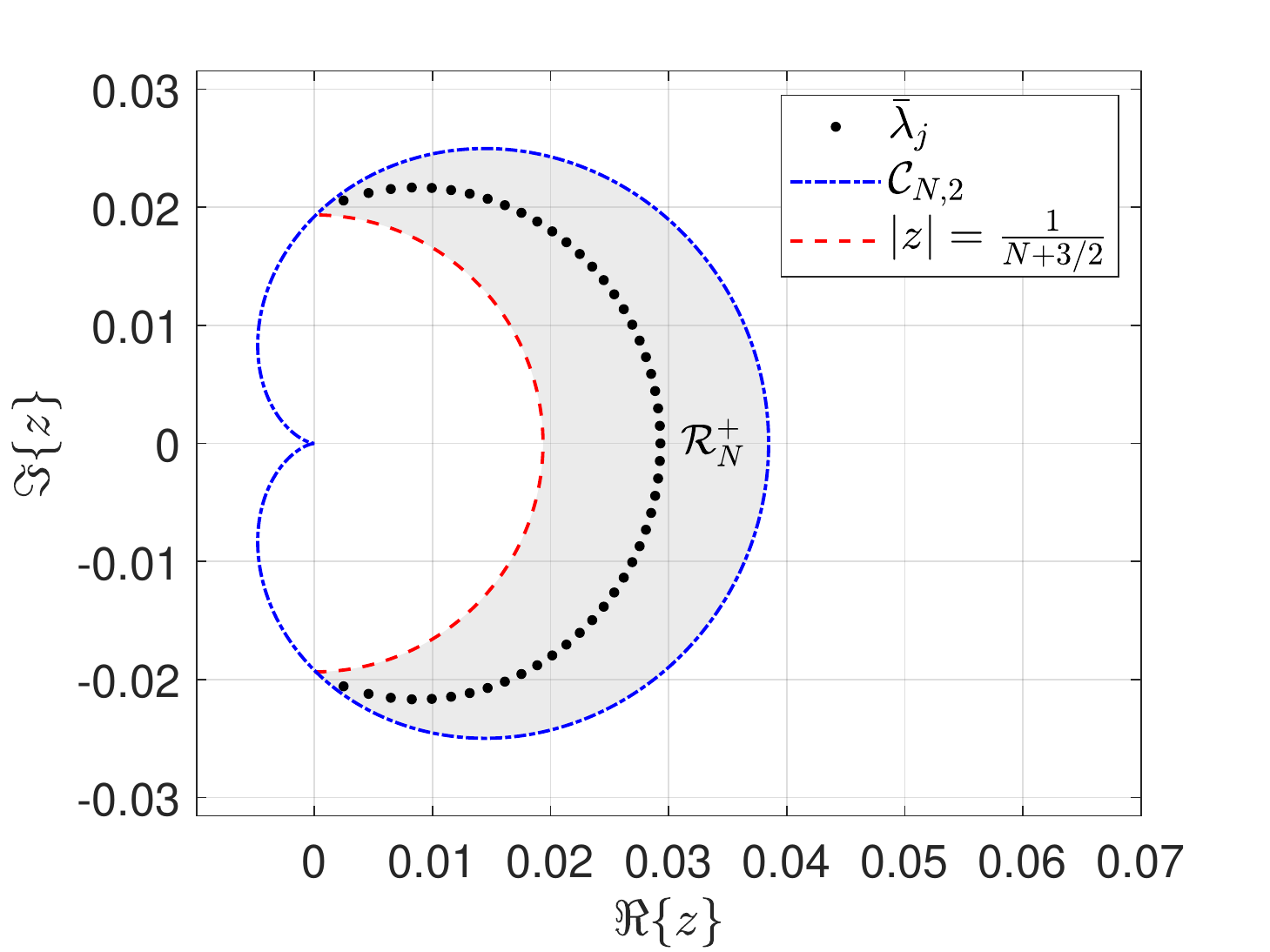}\hspace*{-10pt}
  \includegraphics[width=.34\textwidth]{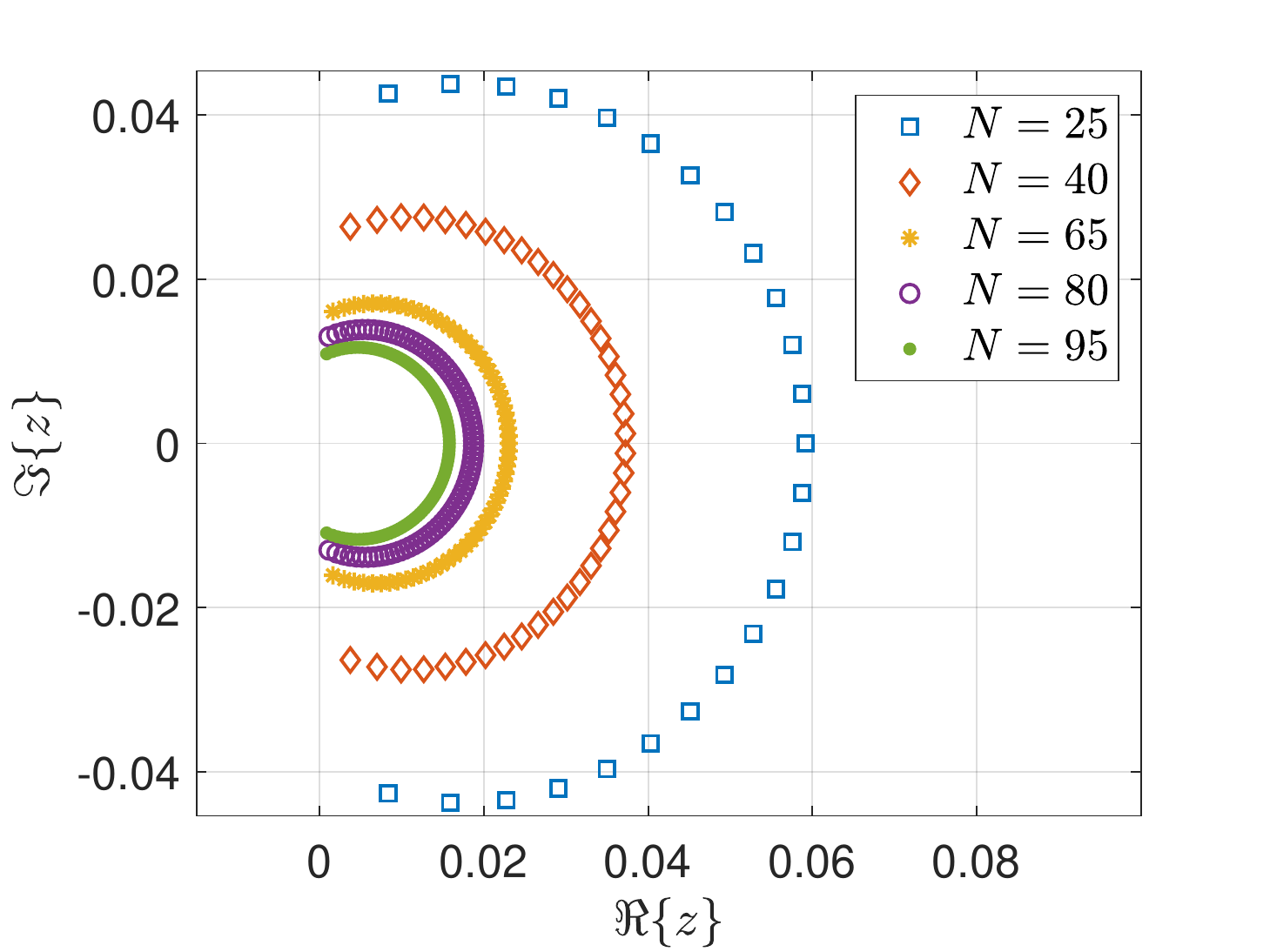} \hspace*{-10pt}
  \includegraphics[width=.34\textwidth]{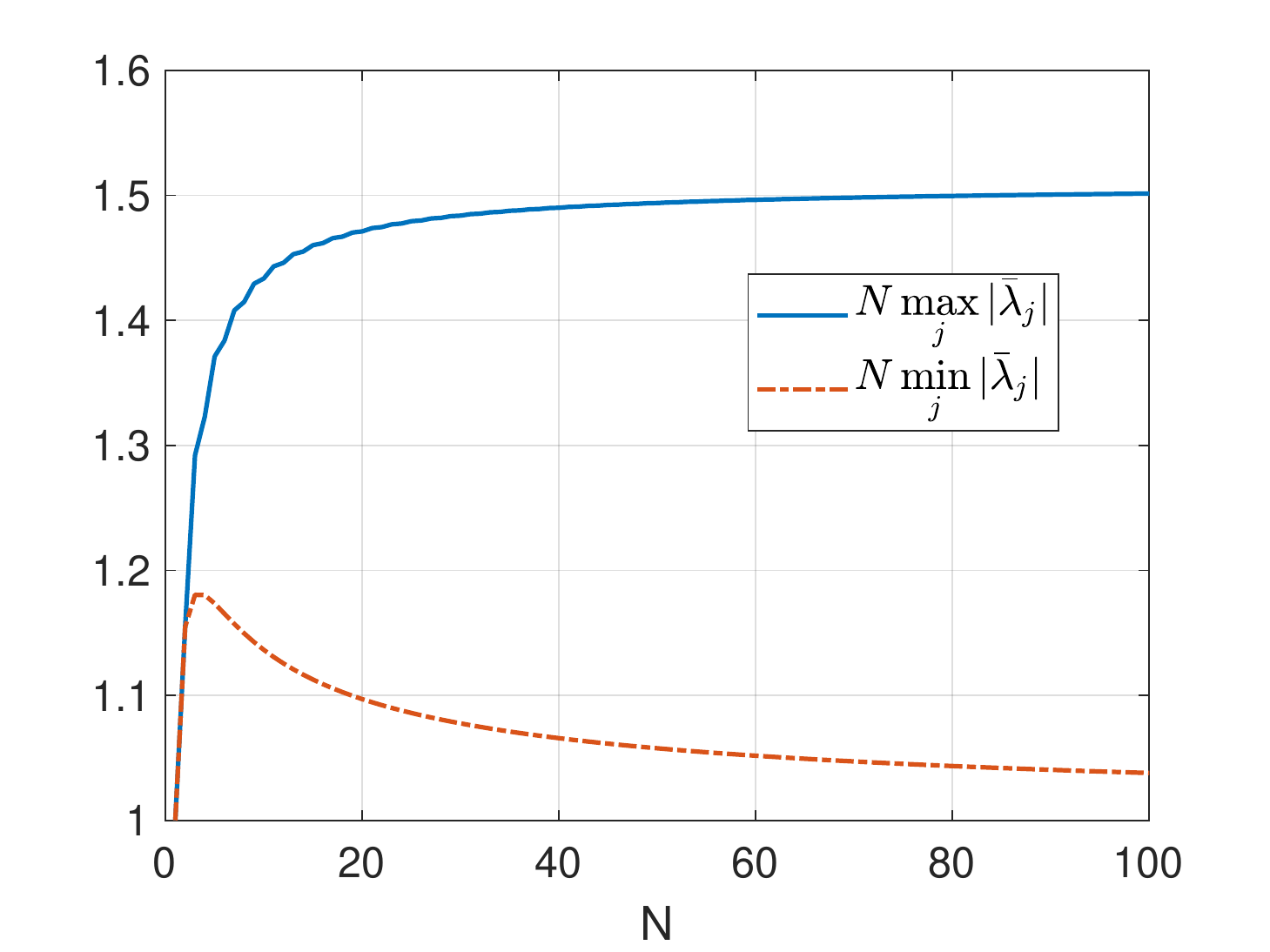}
  \vspace*{-10pt}
  \caption{\small Eigenvalues of $\xoverline{\bs M}$ in \eqref{diagmatrix}. Left:  Distribution of $\bar \lambda_j=\bar \lambda_{N,j}$ with $N=51$ in the crescent-shaped region.
  Middle:  Distributions of the eigenvalues for different $N$. Right: Behaviour of the eigenvalues with the smallest and largest magnitudes for $1\le N\le 100$.}
  \label{fig:lamjMD51A}
\end{figure}

As a direct consequence of Lemma  \ref{lem:RelsMC} and Theorem \ref{thm:eigMp}, we have the following important result on the  eigenvalues and eigenvectors of $\bs D.$
\begin{cor}\label{Deigen}  The eigenvalues $\{\lambda_j^D\}_{j=1}^N$ of $\bs D$ in \eqref{odesystem0}-\eqref{Colentries} are reciprocal of the eigenvalues $\{\bar \lambda_j\}_{j=1}^N$ of $\xoverline{\bs M}$ in \eqref{diagmatrix}-\eqref{odesystem10}, so they are distinct, and conjugate of each other, except for  one real eigenvalue when $N$ is odd.  Moreover, all the eigenvalues satisfy  $\Re\{\lambda_j^D\}>0,$ and
lie in the region
\begin{equation}\label{annulaD}
  \widehat{{\mathcal R}}_N^+:= \Big\{z = \rho e^{\ri \theta} \in \mathbb{C} : \frac{N+1}{1+\cos\theta} \le \rho < N+\frac{2}{3},\,\, |\theta| < \pi - \Theta_{N, 2} \Big\},\quad 1\le j\le N.
\end{equation}
The corresponding eigenvector of $\lambda_j^D$ is
\begin{equation*}\label{pjeigen}
\bs v_j^D = \frac{\bs b_j^D}{|\bs b_j^D|}, \quad \bs b_j^D= \bs\Phi \bs {\bar b}_j,\quad 1\le j\le N,
\end{equation*}
where the matrix $\bs \Phi$ is given in Lemma \ref{lem:RelsMC} and $\{\bs {\bar b}_j\}$ is given by \eqref{bjeigen}.
\end{cor}
\begin{figure}[t]
  \centering
  \includegraphics[width=.34\textwidth]{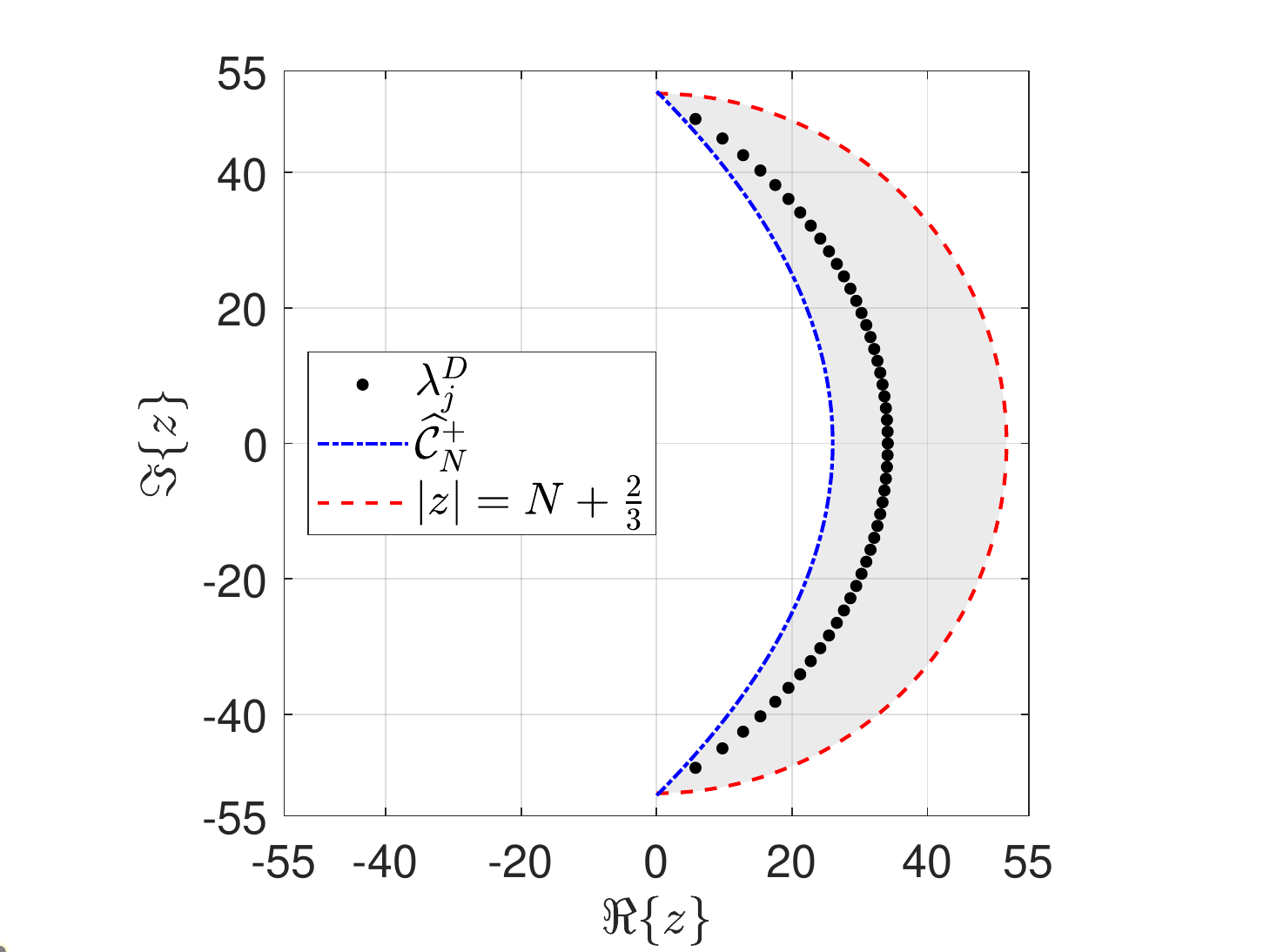}\hspace*{-15pt}
  \includegraphics[width=.34\textwidth]{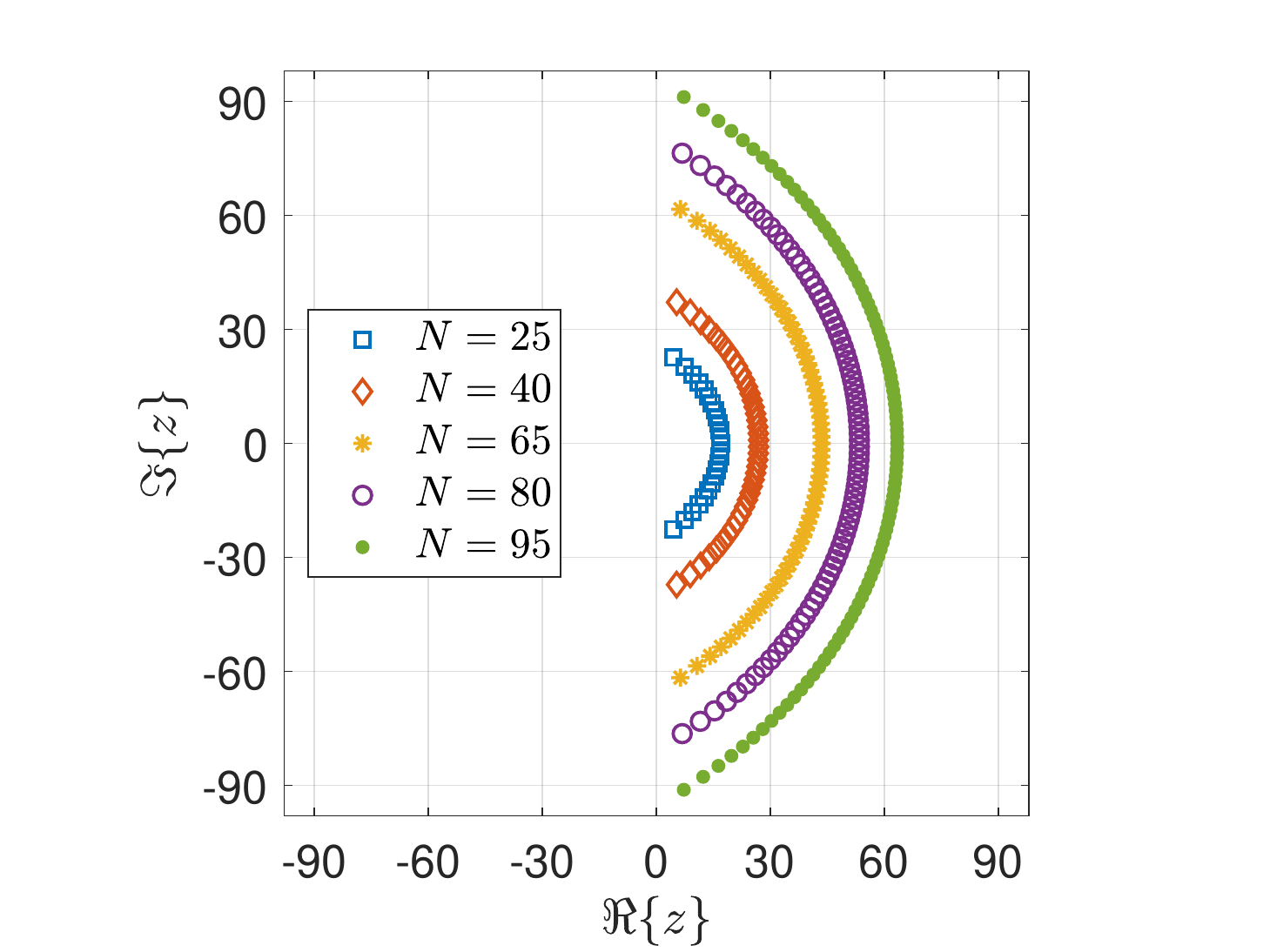} \hspace*{-15pt}
  \includegraphics[width=.34\textwidth]{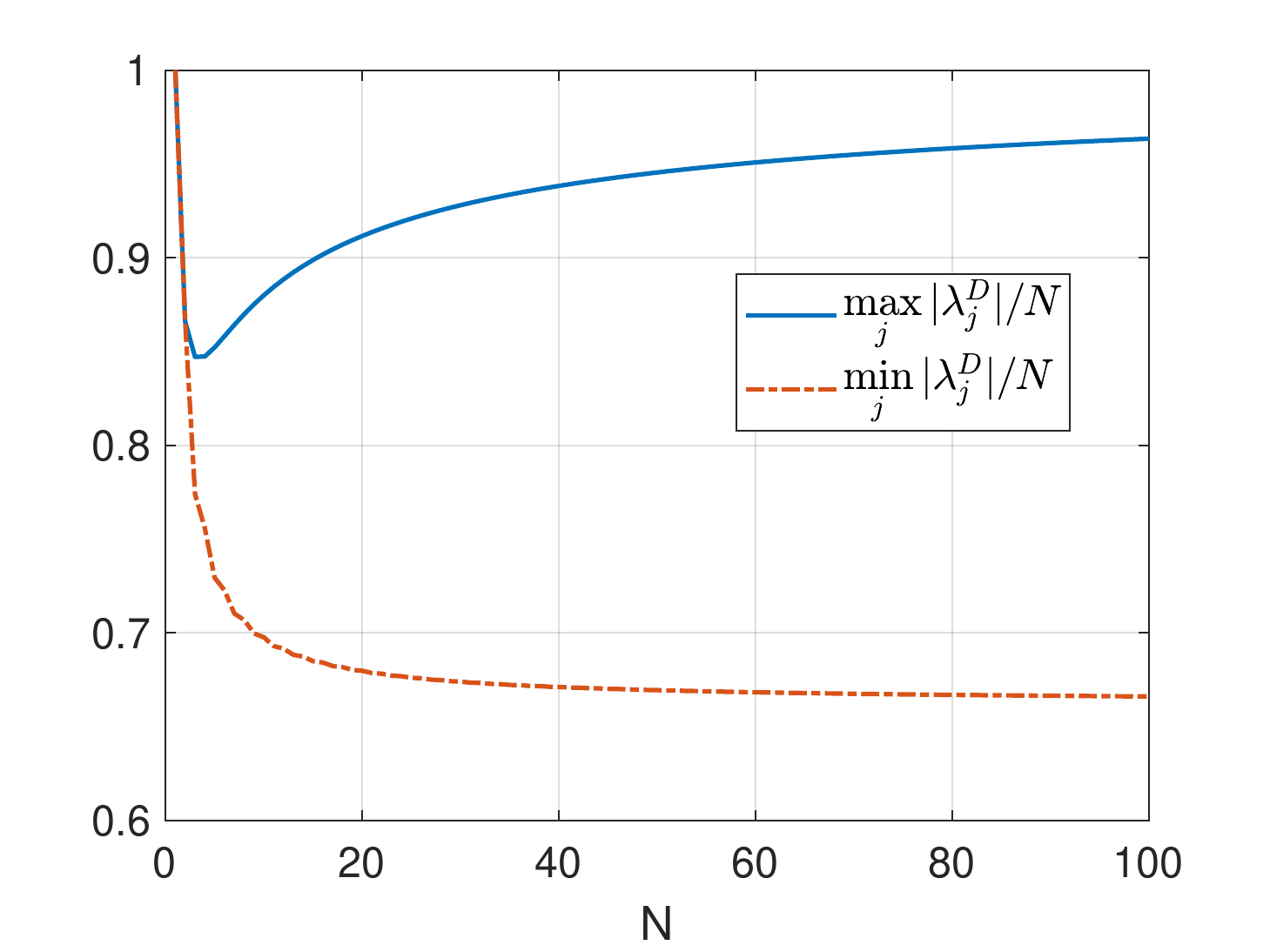}
  \caption{\small Eigenvalues of $\bs D$  in \eqref{odesystem0} with the same setting as Figure  \ref{fig:lamjMD51A}. Here $\widehat{\mathcal C}_N^+$ denotes the left boundary curve in \eqref{annulaD}.}
  \label{fig:lamjMD51}
\end{figure}

\begin{rem}\label{relaRmk}\emph{The Bessel polynomial directly  bears on the modified  Bessel function {\rm (cf. \cite{Grosswald1978}):}
\begin{equation*}
  B_N(z) = \sqrt{\frac{2}{\pi z}} \re^{1/z} K_{N+\frac12}(1/z),
\end{equation*}
which agrees with the identification  in  \cite{WangWaleffe2014a}. However, we can also identify the eigenvectors through our approach.}
\end{rem}

We illustrate in Figure \ref{fig:lamjMD51} the eigenvalues of $\bs D$ in the same setting as in Figure  \ref{fig:lamjMD51A}.
It is seen from \eqref{annulaD} that  $|\lambda_j^D|=O(N).$ Indeed, by Remark \ref{nmBnalpha},  the unique real eigenvalue for odd $N$ behaves like
\begin{equation*}\label{realeignD}
 \lambda_{(N+1)/2}^{D}=\frac  N \nu + O(1),\;\;\; \nu\approx 1.50888.
\end{equation*}

\section{Eigenvalue analysis of  LDPG methods for second-order IVPs}\label{sect4:2nd}
\setcounter{equation}{0}
\setcounter{lmm}{0}
\setcounter{thm}{0}

In this section, we extend the eigenvalue analysis to the second-order IVP:
\begin{equation}\label{hyperModel2}
  u''(t) = \sigma u(t),\;\;\; t \in I; \quad  u(-1) = u_0,\quad  u'(-1) = u_1,\quad
\end{equation}
where the constants $\sigma \neq 0,$ and  $u_0, u_1$ are given. 
Introduce the dual approximation spaces
\begin{equation*}\label{eq:Space2nd}
V_N= \big\{ \phi \in \mathbb P_{N+1} : \phi(-1) = \phi'(-1) = 0 \big\}, \quad
V_N^* =\big \{ \psi \in \mathbb P_{N+1} : \psi(1) = \psi'(1) = 0\big \}.
\end{equation*}
Here, we set the highest degree to be $N+1$ so that  ${\rm dim}(V_N)={\rm dim}(V_N^*)=N.$
Let  $\{t_i, \omega_i\}_{i=0}^{N+1}$ be the  Legendre-Gauss-Lobatto  quadrature nodes and weights with the discrete inner product and the exactness of quadrature
\begin{equation}\label{exactA}
 \langle f, g \rangle_N := \sum_{i=0}^{N+1} f(t_i)g(t_i) \omega_i; \quad \langle f, g \rangle_N=(f,g),\quad \forall\, f\cdot g\in {\mathbb P}_{2N+1}.
\end{equation}
The Legendre pseudospectral dual-Petrov-Galerkin scheme for \eqref{hyperModel2}  is to find
 $u_N = u_0 + (1+t)u_1 + v_N \in \mathbb P_{N+1}$ with $v_N \in V_N$ such that
\begin{equation}\label{eq:dPGalM2nd2}
  \langle v_N', \psi' \rangle_N  + \sigma \langle v_N, \psi \rangle_N = -\sigma \langle u_0 + (1+t)u_1, \psi \rangle_N,\quad \forall\, \psi \in V_N^*.
\end{equation}
Choose the basis functions for $V_N$ and $V_N^*$ as
\begin{equation*}\label{eq:wdPetbas}
\begin{split}
  \phi_k(t) &= c_k \big(P_k(t) + a_k P_{k+1}(t) + b_k P_{k+2}(t)\big) \in V_N, \\
  \phi_k^*(t) &= d_k \big(P_k(t) - a_k P_{k+1}(t) + b_k P_{k+2}(t)\big) \in V_N^*,
\end{split}
\end{equation*}
for $0 \le k \le N-1,$ where
$$
a_k = \frac{2k+3}{k+2},\quad b_k = \frac{k+1}{k+2},\quad
c_k = \frac{k+2}{\sqrt{2}},\quad d_k= \frac{1}{\sqrt{2}(k+1)(2k+3)}. 
$$
Write and denote
\begin{equation*}
  v_N(t) = \sum_{k=0}^{N-1} \tilde{v}_k \phi_k(t),\quad
  \tilde{\bs{v}} = (\tilde{v}_0, \cdots, \tilde{v}_{N-1})^\top.
\end{equation*}
The matrix form of \eqref{eq:dPGalM2nd2} reads
\begin{equation}\label{linearsys}
  \big( \bs I_{\!N} - \sigma \bs{{M}}^{(2)} \big) \tilde{\bs{v}} = \sigma \bs g, \;\; {\rm where} \;\;
  \bs g := \frac{2u_0+u_1}{3\sqrt{2}}\bs e_1 + \frac{u_1}{15\sqrt{2}}\bs e_2.
\end{equation}
Thanks to \eqref{exactA}, we verify from the properties   \eqref{eq:PnOrtho}-\eqref{eq:deriv12} readily that
$$  \langle \phi_k'', \phi_j^* \rangle_N = (\phi_k'', \phi_j^*) = \delta_{jk},\quad  0 \le j,k \le N-1,$$
and  the mass matrix $\bs{{M}}^{(2)} \in \mathbb R^{N \times N}$ is a penta-diagonal (non-symmetric) matrix with nonzero entries  given by
\begin{equation}\label{eq:wdmatrix}
\bs M_{jk}^{(2)} = \langle \phi_k, \phi_j^* \rangle_N =
\begin{dcases}
  \frac{2(j-1)}{j(2j+1)} d_j c_{j-2}, & k=j-2,\;\; 2 \le j \le N-1, \\
  \frac{4}{(j+1)(j+2)} d_j c_{j-1}, & k=j-1,\;\; 1 \le j \le N-1, \\
  \Big( \frac{2}{2j+1} - \frac{2(2j+3)}{(j+2)^2} + \Big(\frac{j+1}{j+2} \Big)^2 \frac{2}{2j+5} \Big) d_j c_j, & k = j,\;\; 0 \le j \le N-2, \\
  \frac{-N^3-2N^2+4N+2}{N(N+1)^2(2N-1)(2N+1)}, & k=j=N-1, \\
  -\frac{4}{(j+2)(j+3)} d_j c_{j+1}, & k=j+1,\;\; 0 \le j \le N-2, \\
  \frac{2(j+1)}{(j+2)(2j+5)} d_j c_{j+2}, & k=j+2,\;\; 0 \le j \le N-3.
\end{dcases}
\end{equation}
Note that for $j=k=N-1$, the LGL quadrature is not exact, so we obtain this entry via
\begin{equation}\label{phi2}
\begin{split}
  \langle \phi_{N-1}, & \phi_{N-1}^* \rangle_N
  = c_{N-1} d_{N-1} \big\{ \| P_{N-1}\|^2 -a_{N-1}^2 \| P_N\|^2 + b_{N-1}^2 \langle P_{N+1}, P_{N+1}\rangle_N \big\} \\
  &= \frac{N+1}{N(2N-1)(2N+1)} - \frac{1}{N(N+1)} + \frac{N}{2(N+1)(2N+1)} \langle P_{N+1}, P_{N+1}\rangle_N \\
  &= \frac{-N^3-2N^2+4N+2}{N(N+1)^2(2N-1)(2N+1)},
\end{split}
\end{equation}
where we used the property (see \cite[p.\! 101]{Shen2011Book}):
\begin{equation*}
  \langle P_{N+1}, P_{N+1}\rangle_N = \frac{2}{N+1}.
\end{equation*}
Remarkably,  we can exactly characterise the eigenvalues of $\bs M^{(2)}$ as follows.

\begin{thm}\label{thm:eigMw} There holds the relation
\begin{equation}\label{squareA}
\bs M^{(2)}=\widetilde{\bs M} \times \widetilde{\bs M},
\end{equation}
where $\widetilde{\bs M}$ is the Jacobi matrix of the three-term recurrence relation of the GBPs $\{B_n^{(4)}(z)\},$ i.e.,
\eqref{eq:recur0} with $\alpha = 4, \beta = 2.$ Consequently, the  eigenvalues 
 of  $\bs{{M}}^{(2)}$ are
\begin{equation*}\label{mujform}
\mu_j = \big(z_j^{(4)}\big)^2,\quad 1 \le j \le N,
\end{equation*}
where $\{z_j^{(4)}\}$ are zeros of the GBP $B_{N}^{(4)}(z).$ They are  simple, conjugate of each other and  lie in the region
\begin{equation}\label{eq:RegionMw}
  \mu_j \in \mathcal{R}_{N, 4}^{(2)} := \Big\{ z = \rho \re^{\ri\theta} \in \mathbb{C}: \frac{1}{(N+\frac53)^2} < \rho \le \Big(\frac{1+\cos(\theta/2)}{N+3}\Big)^2,\;\; |\theta| \le 2(\pi-\Theta_{N,4}) \Big\},
  \end{equation}
  for $1 \le j \le N,$ where $\Theta_{N,4}$ is given in \eqref{thetanalpha} and the equal sign can be only attained  at $\theta = 0.$
Furthermore, the corresponding eigenvectors are
\begin{equation}\label{EigvectorW}
  \tilde{\bs v}_j = \frac{\tilde {\bs b}_j}{|\tilde{\bs b}_j|},\quad \tilde{\bs b}_j:= \big(B_0^{(4)}(z_j^{(4)}), B_1^{(4)}(z_j^{(4)}), \cdots, B_{N-1}^{(4)}(z_j^{(4)})\big)^\top,\quad 1 \le j \le N.
\end{equation}
\end{thm}
\begin{proof}  We  write the recurrence relation \eqref{eq:recur0} with $\alpha = 4, \beta = 2$ and $j=0, \cdots, N-2$ as the matrix form
\begin{equation}\label{eq:th41-2}
  -z \tilde{\bs b}(z) = \widetilde{\bs M} \tilde{\bs b}(z) - \frac{N+2}{(N+1)(2N+1)} B_N^{(4)}(z) \bs e_N,
\end{equation}
where $\tilde{\bs b}(z) = \big(B_0^{(4)}(z), B_1^{(4)}(z), \cdots, B_{N-1}^{(4)}(z)\big)^\top$ and $\widetilde{\bs M}$ is the tri-diagonal Jacobi matrix with nonzero entries given by
\begin{equation}\label{eq:Mat2nd}
 \widetilde {\bs M}_{jk} =
  \begin{dcases}
    \frac{j}{(j+1)(2j+3)}, & k=j-1,\;\;\; 1 \le j \le N-1, \\
    \frac{1}{j+1} - \frac{1}{j+2}, & k=j,\;\; 0 \le j \le N-1, \\
    -\frac{j+3}{(j+2)(2j+3)}, & k=j+1,\;\; 0 \le j \le N-2.
  \end{dcases}
\end{equation}
We conclude from \eqref{eq:th41-2} that the eigenvalues of $\widetilde{\bs M}$ are $N$ zeros of $ B_{N}^{(4)}(z)$ with the corresponding unit eigenvectors given by  \eqref{EigvectorW}.

Notably,  we can directly  verify from \eqref{eq:wdmatrix} and \eqref{eq:Mat2nd} that $\bs M^{(2)}=\widetilde{\bs M}\times \widetilde{\bs M}.$
In view of  Theorem \ref{lem:Bnzero} with $\alpha = 4$,   we complete the proof.
\end{proof}

In Figure \ref{fig:muj51} (left), we depict the distribution of  eigenvalues of $\bs M^{(2)}$ and the region $\mathcal{R}_{N, 4}^{(2)}$ given in \eqref{eq:RegionMw} with $N=51$. Apparently, they are all distributed within  the shaded region  as shown in Theorem \ref{thm:eigMw}. We also demonstrate the  eigenvalues for various $N,$ and illustrate behaviours of the maximum and the minimum magnitudes in the other two sub-figures in Figure \ref{fig:muj51}.
As predicted by Theorem \ref{thm:eigMw}, we have $|\mu_j|=O(N^{-2}).$ 

\begin{figure}[htp]
  \centering
  \includegraphics[width=.34\textwidth]{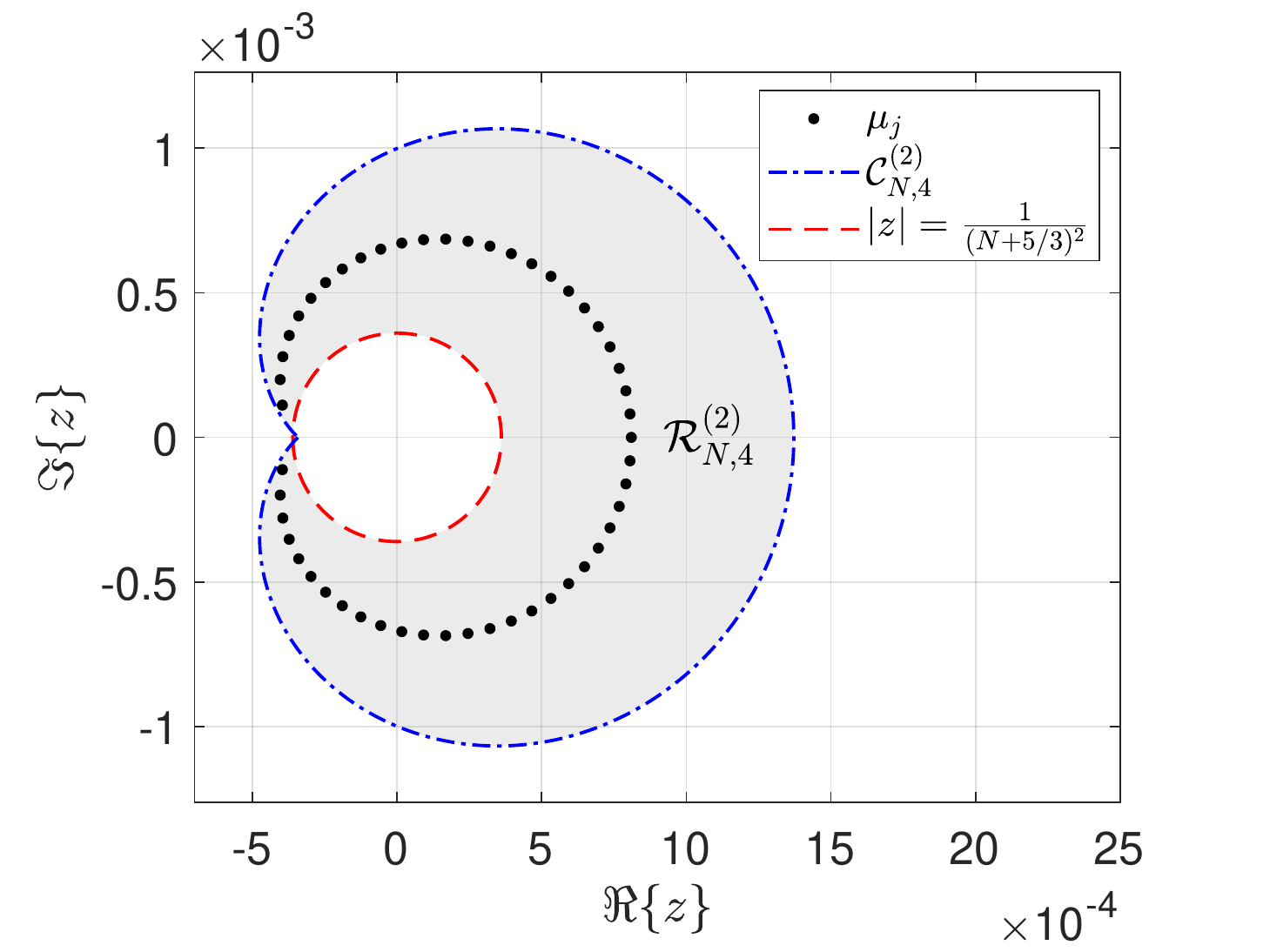}\hspace*{-10pt}
  \includegraphics[width=.34\textwidth]{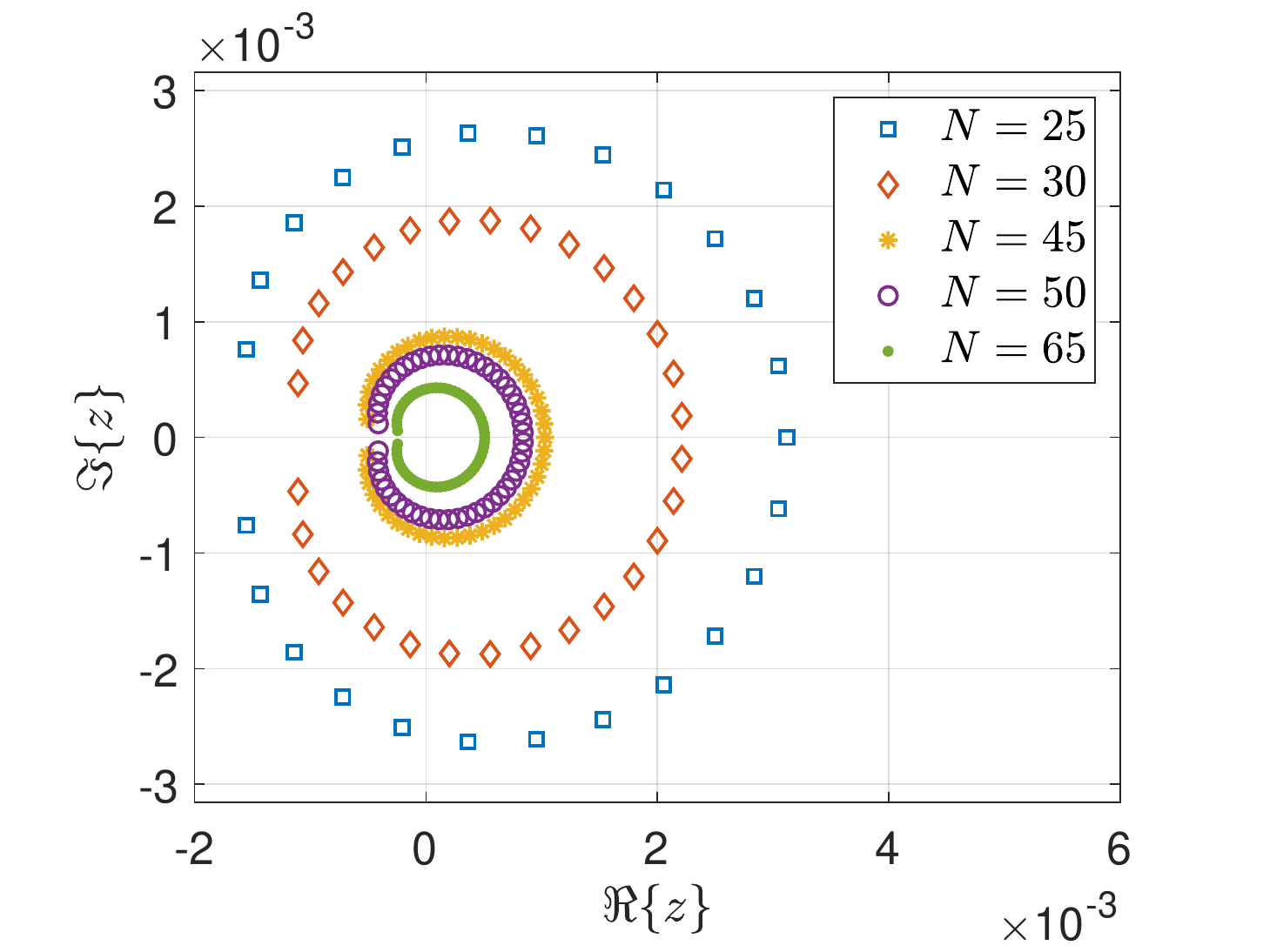}\hspace*{-10pt}
  \includegraphics[width=.34\textwidth]{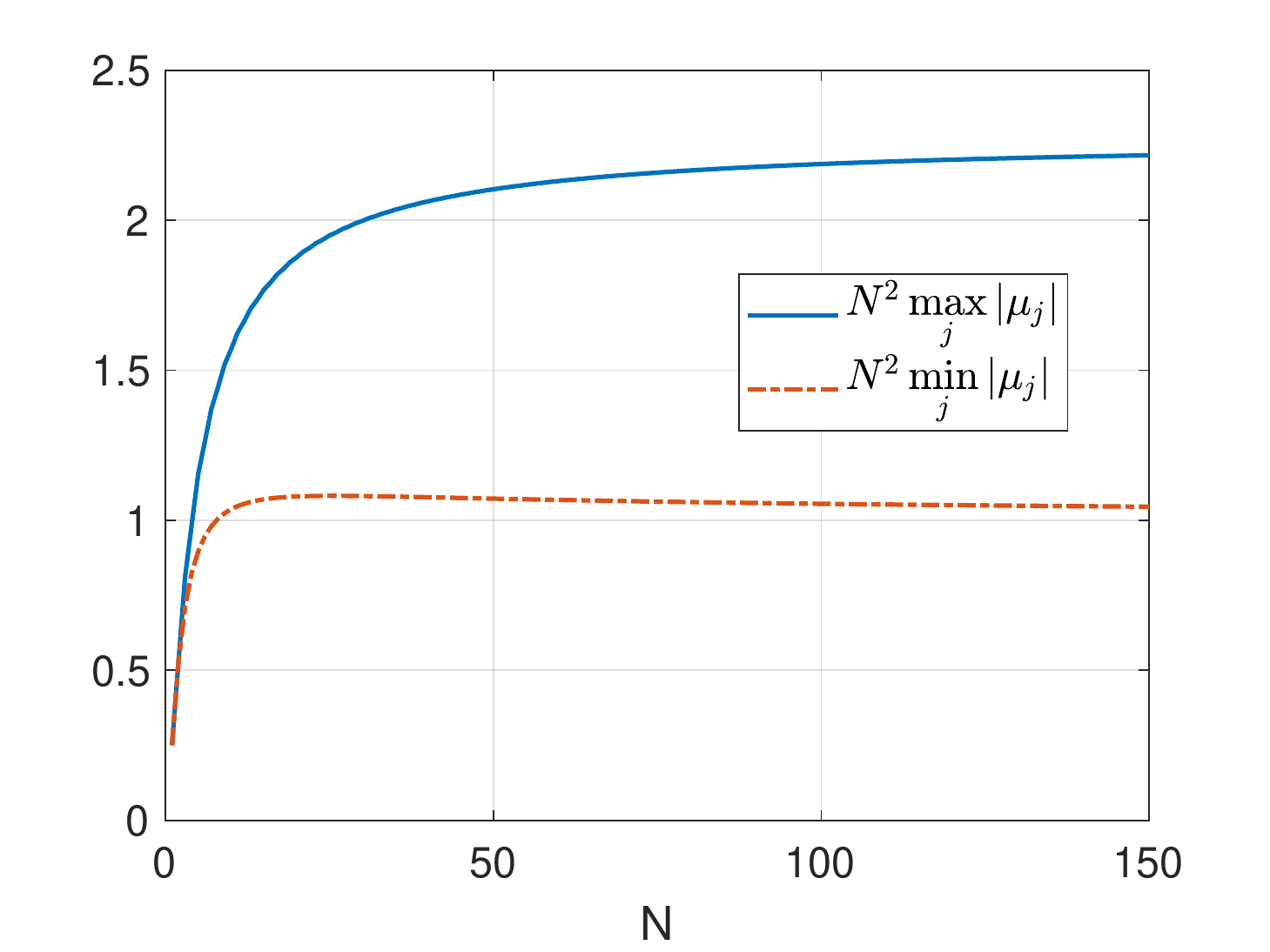}
  \vspace*{-10pt}
  \caption{\small Eigenvalues of $\bs M^{(2)}$  in \eqref{eq:wdmatrix} with the same setting as Figure  \ref{fig:lamjMD51A}. }
  \label{fig:muj51}
\end{figure}

It is seen from the above discussions that with the appropriate choice of basis functions, we were able to associate the mass matrix of  the pseudospectral scheme \eqref{eq:dPGalM2nd2}  with the GBP through  \eqref{squareA}. However, for the spectral Petrov-Galerkin scheme, the eigenvalues in magnitude appear an $O(N^{-3})$ perturbation of those for the pseudospectral ones.
To demonstrate  this,  we   replace the discrete inner product in  \eqref{eq:dPGalM2nd2} by the continuous inner product, leading to the spectral dual-Petrov-Galerkin scheme for \eqref{hyperModel2}{\rm:}  \emph{find $u_N = u_0 + (1+t)u_1 + v_N \in \mathbb P_{N+1}$ with $v_N \in V_N$ such that}
\begin{equation}\label{eq:dPGalM2nd}
  (v_N', \psi')  + \sigma (v_N, \psi) = -\sigma (u_0 + (1+t)u_1, \psi),\quad \forall \psi \in V_N^*.
\end{equation}
The only difference is to modify the last entry of $\bs M^{(2)}$ in  the linear system \eqref{linearsys}, i.e.,
\eqref{phi2} by
$$
(\phi_{N-1},  \phi_{N-1}^*)= -\frac{2(N^2+N-3)}{N(N+1)(2N-1)(2N+3)}= \langle\phi_{N-1},  \phi_{N-1}^*\rangle_N+O(N^{-2}).
$$
Numerically, we find from Figure \ref{fig:MJ-diff-2nd} that the eigenvalues $\{\tilde \mu_j\}$ of the modified matrix are a  $O(N^{-3})$ perturbation of the eigenvalues $\{\mu_j\}$ of $\bs M^{(2)}$.

\begin{figure}[!ht]
  \centering
  \includegraphics[width=.45\textwidth]{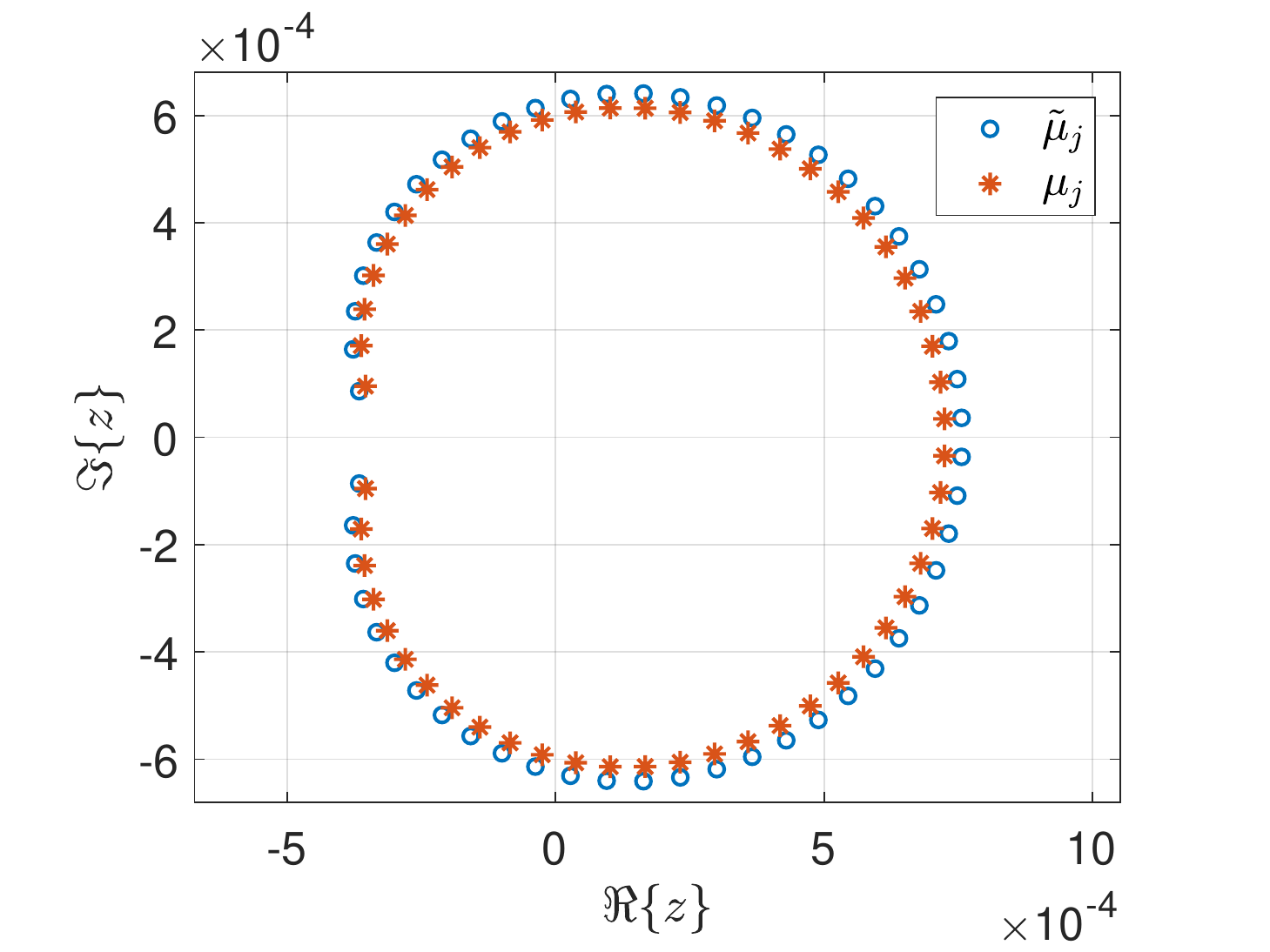} \quad
  \includegraphics[width=.45\textwidth]{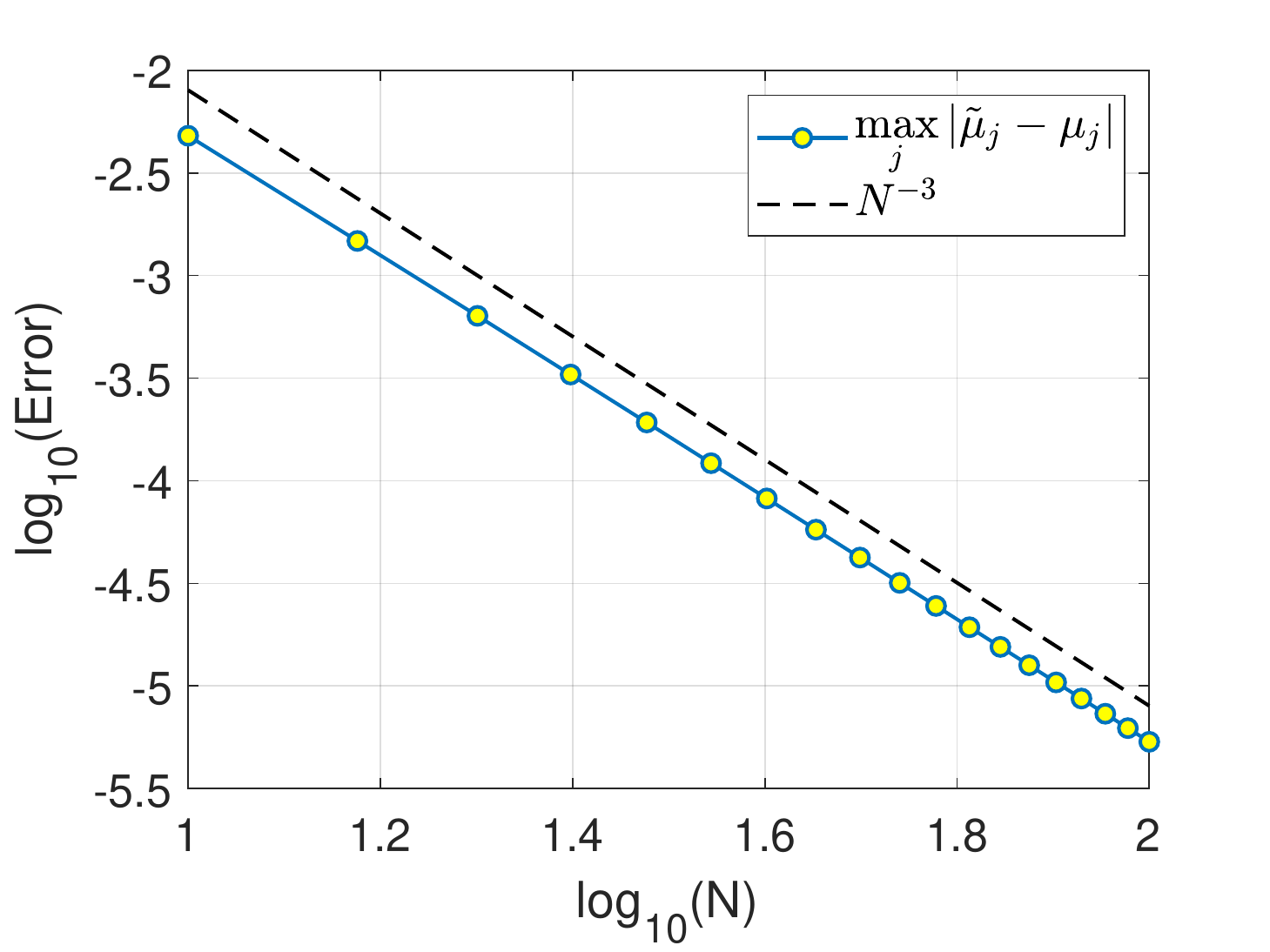}
  \vspace*{-6pt}
  \caption{\small Eigenvalues of $\bs M^{(2)}$ in the pseudospectral scheme \eqref{eq:dPGalM2nd2}   and the spectral scheme
  \eqref{eq:dPGalM2nd}. Left: Distributions
   with $N=54.$ Right:  Asymptotic order $O(N^{-3})$.}
  \label{fig:MJ-diff-2nd}
\end{figure}

\section{Eigenvalue analysis of LDPG  methods for higher-order IVPs}\label{sect5:higher}
\setcounter{equation}{0}
\setcounter{lmm}{0}
\setcounter{thm}{0}

In this section, we extend the eigenvalue analysis to higher-order IVPs with a focus on the third-order IVPs. We also consider the reformulation of the high-order IVPs as a system of first-order equations for which we can precisely characterise the eigenvalues.
\subsection{Third-order IVPs}
 To fix the idea, we consider the third-order IVP:
\begin{equation}\label{eq:IVP3rd}
  u'''(t) = \sigma u(t),\quad t \in I; \;\quad
  u(-1) = u_0,\;\; u'(-1) = u_1, \;\; u''(-1) = u_2, 
\end{equation}
where $\sigma\not=0, u_0, u_1$ and $u_2$ are given constants. 

Define the dual approximation spaces
\begin{equation*}
\begin{split}
  & V_N = \{ \phi \in \mathbb P_{N+2}: \phi(-1) = \phi'(-1) = \phi''(-1) = 0 \},\\
&  V_N^* = \{ \phi \in \mathbb P_{N+2}: \phi(1) = \phi'(1) = \phi''(1) = 0 \},
  \end{split}
\end{equation*}
with ${\rm dim}(V_N)={\rm dim}(V_N^*)=N.$
Then the LDPG scheme for \eqref{eq:IVP3rd} is
\begin{equation}\label{eq:DualSpec3rd}
\begin{dcases}
  \text{Find} \; u_N = u_0 + (1+t)u_1 + \frac12(1+t)^2 u_2 + v_N \in \mathbb P_{N+2}\;\; \text{with}\;\; v_N\in V_N\;\; \text{s.t.}\\
  (v_N', \psi'') - \sigma (v_N, \psi) = \sigma (u_0 + u_1 (1+t) +  \frac{u_2}2(1+t)^2, \psi),\quad \forall \psi \in V_N^*.
 \end{dcases}
\end{equation}
Choose the basis functions
\begin{equation*}
\begin{split}
  \phi_k(t) &= d_k \big( P_k(t) + a_k P_{k+1}(t) + b_k P_{k+2}(t) + c_k P_{k+3}(t) \big) \in V_N, \\
  \phi_k^*(t) &= e_k \big( P_k(t) - a_k P_{k+1}(t) + b_k P_{k+2}(t) - c_k P_{k+3}(t) \big) \in V_N^*,
\end{split}
\end{equation*}
for $0\le  k \le N-1$, where
\begin{gather*}\label{eq:abck3rd}
  a_k = \frac{3(2k+3)}{2k+5},\quad b_k = \frac{3(k+1)}{k+3},\quad
  c_k = \frac{(k+1)(2k+3)}{(k+3)(2k+5)},\\
  d_k = \frac{(k+2)(k+3)}{2(2k+3)},\quad e_k = \frac{1}{(k+1)(k+2)(2k+3)}.
\end{gather*}
Then we can write and denote
\begin{equation*}
  v_N(t) = \sum_{k=0}^{N-1} \tilde{v}_k \phi_k(t), \quad \tilde{\bs v} = (\tilde{v}_0, \cdots, \tilde{v}_{N-1})^\top.
\end{equation*}
Substituting $v_N$ into \eqref{eq:DualSpec3rd} and taking $\psi = \phi_j^*$ for $0 \le j \le N-1$ lead to
\begin{equation}\label{3rdLDPGs}
  (\bs I_N - \sigma \bs M^{(3)}) \tilde{\bs v} = \sigma \bs g,
\end{equation}
where the elements of the column-$N$ vector $\bs g$ is given by
$$\bs g_j= (u_0 + u_1(1+t) + \frac{u_2}2 (1+t)^2, \phi_j^*),\quad 0 \le j \le N-1.$$
 Using the properties of Legendre polynomials, we find that 
 the nonzero entries of the seven-diagonal (non-symmetric) matrix $\bs M^{(3)}$ are given by
\begin{equation*}
\begin{split}
  \bs M_{jk}^{(3)} = (\phi_k, \phi_j^*) =
  \begin{dcases}
    e_j d_{j-3} c_{j-3} \gamma_j, & k = j-3, \\
    e_j d_{j-2} \big( b_{j-2} \gamma_j - c_{j-2}a_j \gamma_{j+1} \big), & k = j-2, \\
    e_j d_{j-1} \big( a_{j-1} \gamma_j - b_{j-1}a_j \gamma_{j+1} + c_{j-1}b_j \gamma_{j+2} \big) & k = j-1, \\
    e_j d_j \big( \gamma_j - a_j^2 \gamma_{j+1} + b_j^2 \gamma_{j+2} - c_j^2 \gamma_{j+3} \big), & k = j, \\
    e_j d_{j+1} \big( -a_j \gamma_{j+1} + a_{j+1}b_j \gamma_{j+2} - b_{j+1}c_j \gamma_{j+3} \big), & k = j+1, \\
    e_j d_{j+2} \big( b_j \gamma_{j+2} - a_{j+2}c_j \gamma_{j+3} \big), & k = j+2, \\
    -e_j d_{j+3} c_j \gamma_{j+3}, & k = j+3,
  \end{dcases}
\end{split}
\end{equation*}
where $\gamma_j = 2/(2j+1)$ as in \eqref{eq:PnOrtho}, and the matrix of the third derivative is an identity matrix as
\begin{equation*}
  (\phi_k''', \phi_j^*) = - (\phi_k'', (\phi_j^*)' ) = (\phi_k', (\phi_j^*)'') =  \delta_{jk}.
\end{equation*}
We can directly verify the following properties and omit the details.
\begin{prop}\label{Bn5case} The matrix $\bs M^{(3)}$ is only different from $(\breve{\bs M})^3$ in three entries:
\begin{equation}\label{diffind}
(j,k)= (N-2, N-1),\; (N-1, N-2),\; (N-1, N-1),
\end{equation}
where $\breve{\bs M}$ is  the tri-diagonal Jacobi matrix of the three-term recurrence relation for the GBPs $\{B_n^{(5)}(z)\}$ with nonzero entries given by
\begin{equation*}\label{eq:Mat3rd}
 \breve {\bs M}_{jk} =
  \begin{dcases}
    \frac{j}{(j+2)(2j+3)}, & k=j-1,\;\;\; 1 \le j \le N-1, \\
    \frac{6}{(2j+3)(2j+5)}, & k=j,\;\; 0 \le j \le N-1, \\
    -\frac{j+4}{(j+2)(2j+5)}, & k=j+1,\;\; 0 \le j \le N-2.
  \end{dcases}
\end{equation*}
Moreover, for these $(j,k)$ in \eqref{diffind},
\begin{equation*}\label{Morder3}
|{\bs M}^{(3)}_{jk}-  \big(\breve {\bs M}^3\big)_{jk}|=O(N^{-3}).
\end{equation*}
\end{prop}

\begin{figure}[!th]
  \centering
  \includegraphics[width=.45\textwidth]{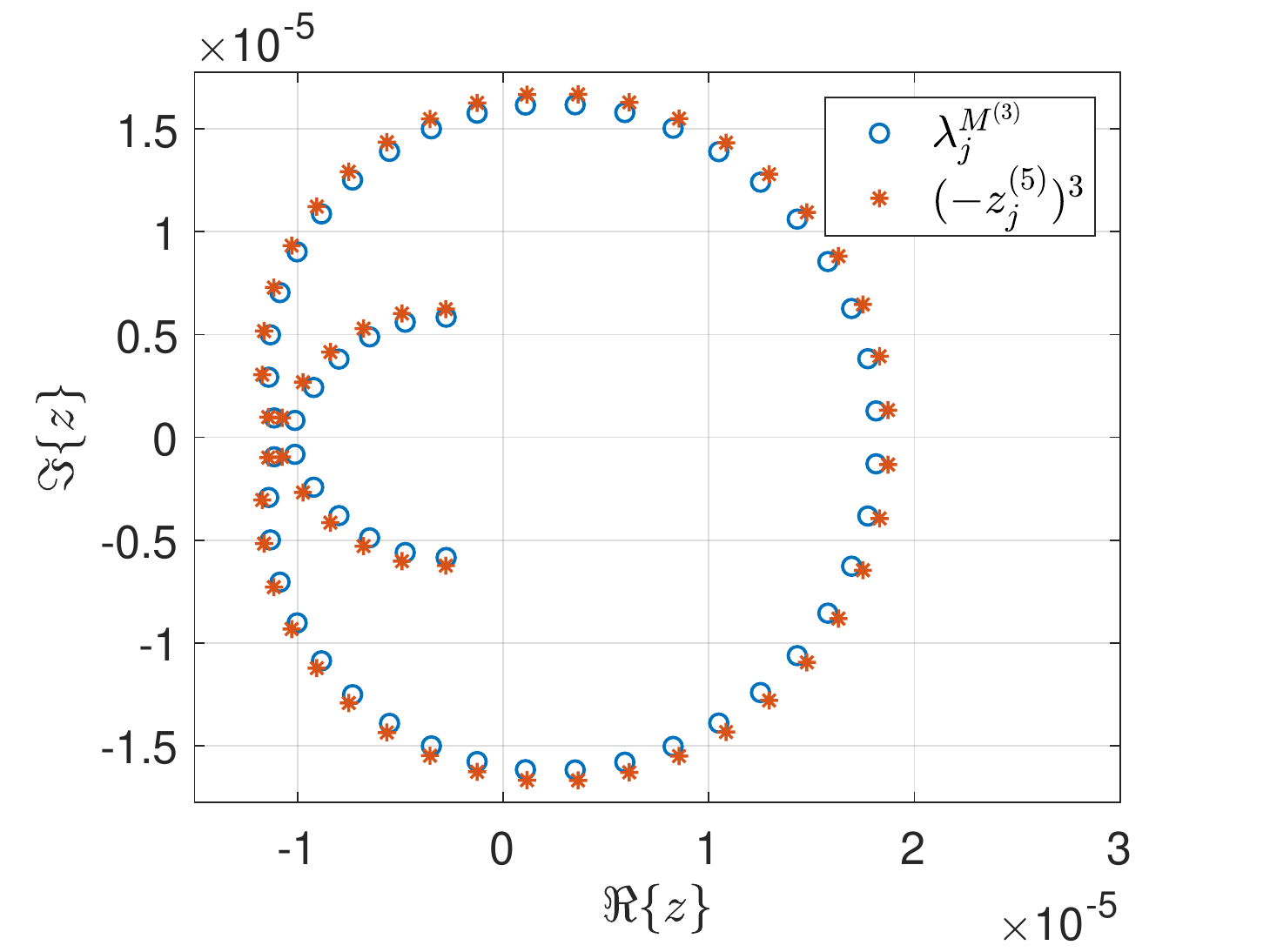} \quad
  \includegraphics[width=.45\textwidth]{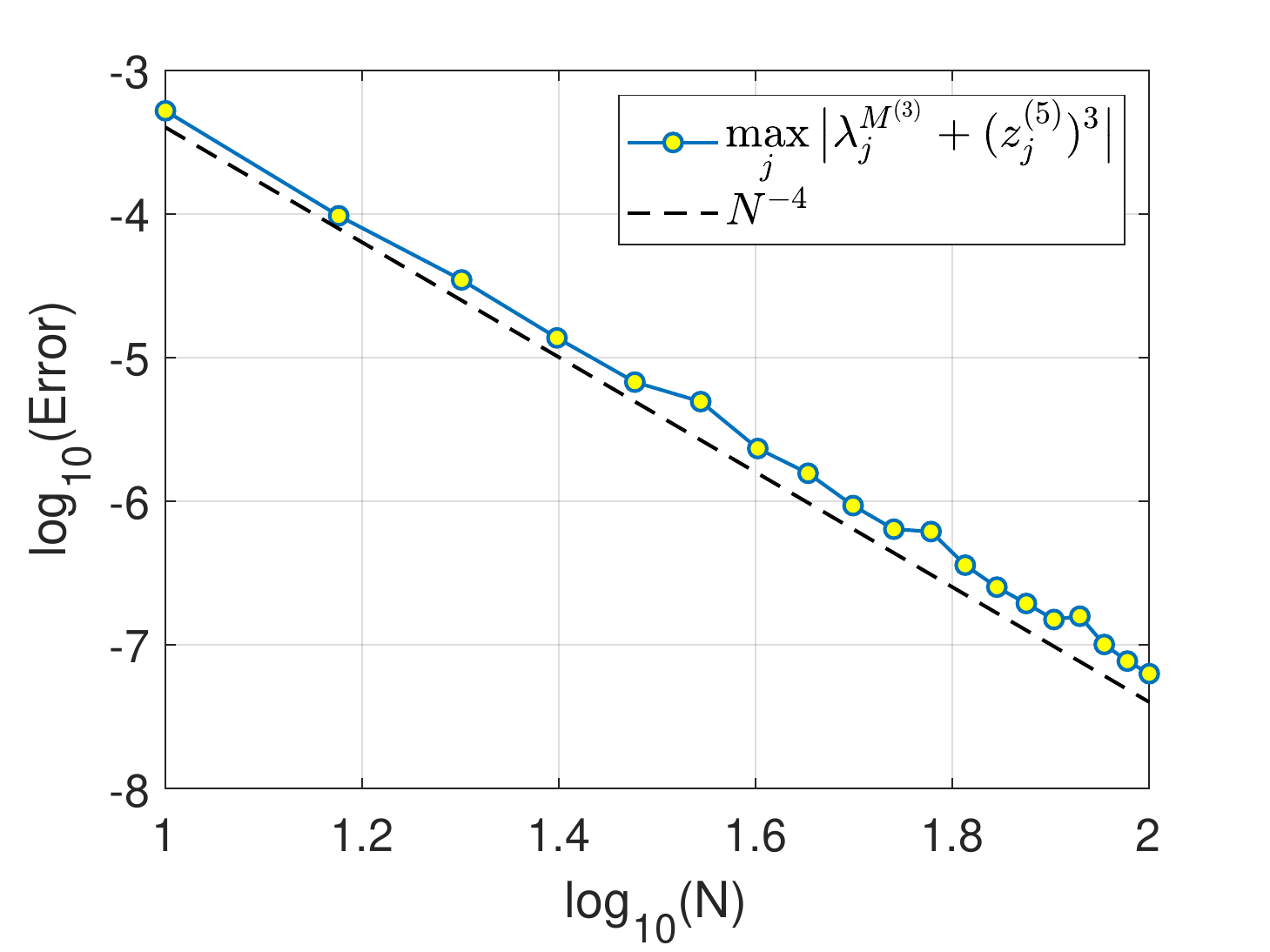}
  \caption{\small Left: Distributions of the eigenvalues of $\bs M^{(3)}$ and the cubic power of the Jacobi matrix of the GBPs
  $\{B_n^{(5)}\}$
  with $N=54.$ Right:  Numerical verification of \eqref{diffmaxj}.}
  \label{fig:MJ-diff-3rd}
\end{figure}

Note that  in Theorem \ref{thm:eigMw}, the matrix of the pseudospectral scheme   \eqref{eq:dPGalM2nd2}  corresponds exactly  to the Jacobi matrix of the GBPs $\{B_n^{(4)}(z)\}$. However,  in the third-order case, it seems not possible to construct a  pseudospectral scheme built upon a suitable quadrature rule so that
$\bs M^{(3)}=(\breve{\bs M})^3.$
The numerical evidence in Figure \ref{fig:MJ-diff-3rd} indicates
\begin{equation}\label{diffmaxj}
\max_{j}\big|\lambda_j^{\bs M^{(3)}}-(-z_j^{(5)})^3\big|=O(N^{-4}),
\end{equation}
where $\{z_j^{(5)}\}_{j=1}^N$ are  zeros of the GBP: $B_N^{(5)}(z).$

\subsection{General $m$th-order IVPs} In general, we consider the $m$th-order IVP:
\begin{equation}\label{eq:IVPmth}
  u^{(m)}(t) = \sigma u(t),\quad t \in I; \;\quad
  u^{(l)}(-1) = u_l, \;\;\; l=0,\cdots, m-1,
\end{equation}
where $\sigma\not=0$ and $\{u_l\}$ are given constants.  Interestingly, we observe   from the previous analysis of
the LDPG method for \eqref{eq:IVPmth} with $m=1,2,3$ the following pattern:
\begin{itemize}
\item[(i)] As shown in Theorem \ref{thm:eigMd} for the first-order IVP,  the eigenvalues of  the mass matrix can be exactly characterised by  zeros of the GBP: $B_N^{(3)}(z).$
\medskip
\item[(ii)] For the second-order IVP,  the eigen-pairs of the matrix for the  pseudospectral scheme \eqref{eq:dPGalM2nd2}  can be exactly characterised by (square of)  zeros of the GBP: $B_N^{(4)}(z)$ (see Theorem \ref{thm:eigMw}).  However, the corresponding matrix of the spectral scheme differs from that of the pseudospectral version in the last entry (i.e., $j=k=N-1$), which leads to a perturbation
in the eigenvalues of order $O(N^{-3})$ (from numerical evidences).
\medskip
\item[(iii)]  For the third-order IVP, the mass matrix of the spectral dual-Petrov-Galerkin scheme differs from the cube of the Jacobi matrix corresponding to the GBPs: $B_j^{(5)}(z)$ in the last three entries,  where the perturbation is of order $O(N^{-3}),$ see Proposition \ref{Bn5case}. However, a dual pseudospectral version with the exact correspondence to zeros of  $B_N^{(5)}(z)$ might not exist.
\end{itemize}
\smallskip

For the general $m$-th order IVP  \eqref{eq:IVPmth}, it is natural to speculate that the eigenvalues of the mass matrix  be    associated with the GBP: $B_N^{(m+2)}(z).$  More specifically, define the dual approximation spaces
\begin{equation*}
\begin{split}
  & V_N = \big\{ \phi \in \mathbb P_{N+m-1}: \phi^{(l)}(-1) = 0,\;\; l=0,\cdots, m-1\big\},\\
&  V_N^* =\big\{ \psi \in \mathbb P_{N+m-1}: \psi^{(l)}(1) = 0,\;\; l=0,\cdots, m-1 \big\},
  \end{split}
\end{equation*}
with ${\rm dim}(V_N)={\rm dim}(V_N^*)=N.$ Then, as in  the previous cases, we choose suitable compact combinations of Legendre polynomials
as basis functions denoted by $\{\phi_k,\phi_k^*\}_{k=0}^{N-1},$ respectively, such that $(\phi_k^{(m)},\phi_j^*)=\delta_{jk}.$ Denote the $N$-by-$N$ mass matrix by $\bs M^{(m)}$ with entries $\bs M^{(m)}_{jk}=(\phi_k,\phi_j^*).$
We conjecture that for $m<N,$  the matrix $\bs M^{(m)}$ differs from  $(\breve{\bs M})^m$ in $2m-3$ entries, where $\breve{\bs M}$ is
 the tri-diagonal  Jacobi matrix corresponding to the GBPs: $B_j^{(m+2)}(z).$ Moreover,  we speculate that
 \begin{equation*}\label{Mmconjecture}
 \|\bs M^{(m)}-(\breve{\bs M})^m\|_\infty=O(N^{-m}),\quad \max_{1\le j\le N}\big|\lambda_j^{\bs M^{(m)}}-\big(\!\!-z_j^{(m+2)}\big)^m\big|=
 O(N^{-(m+1)}),
 \end{equation*}
where $\{z_j^{(m+2)}\}$ are zeros of $B_N^{(m+2)}(z).$

\subsection{LDPG method for $m$th-order  IVP based on   first-order system} Remarkably, we can show that if we rewrite
the underlying IVP as a system of first-order equations, then we are able to exactly characterise the eigenvalues through  zeros of the GBP:
$B_N^{(3)}(z).$

To fix the idea, we still consider the third-order IVP \eqref{eq:IVP3rd}, and reformulate it as the system
\begin{equation*}\label{eq:IVPhigh1}
\begin{dcases}
  v_0'(t) = v_{1}(t),\;\; v_{1}'(t) = v_{2}(t),\;\; v_{2}'(t) = \sigma v_0(t),\quad t\in I=(-1,1),\\
  v_0(-1) = u_{0},\;\; v_{1}(-1) = u_{1},\;\; v_{2}(-1) = u_{2},
  \end{dcases}
\end{equation*}
where we introduced the unknowns:  $v_{i} = u^{(i)}$ for $i=0,1,2$.
For each first-order equation, we apply the Legendre dual-Petrov-Galerkin scheme  \eqref{eq:dPGalM}:  find $v_{i, N} = u_i + w_{i, N} \in \mathbb P_N$ with $w_{i, N} \in {_0 \mathbb P}_N$ such that for all $\psi \in {^0 \mathbb P_N}$,
\begin{equation}\label{eq:IVPhigh2}
\begin{dcases}
  (w_{0, N}', \psi) - ( w_{1, N}, \psi) = (u_1, \psi), \\
(w_{1, N}', \psi) - (w_{2, N}, \psi) = (u_2, \psi), \\
  ( w_{2, N}', \psi) - \sigma (w_{0, N}, \psi) = \sigma (u_0, \psi).
\end{dcases}
\end{equation}
Choosing the basis functions for ${^0 \mathbb P_N}$  and ${_0 \mathbb P_N}$ as in \eqref{eq:dPetbas}-\eqref{eq:dPetbas2},  we  write and denote
\begin{equation*}
w_{i, N}= \sum_{k=0}^{N-1} \hat w_{i, k} \phi_k(t),\quad
  \hat{\bs{w}}_{i} = (\hat w_{i, 0}, \cdots, \hat w_{i, N-1})^\top,\quad i = 1,2,3.
\end{equation*}
Taking $\psi = \phi_j^*$ in \eqref{eq:IVPhigh2}, we immediately obtain the corresponding linear systems:
\begin{equation}\label{eq:IVPhighSys}
\begin{split}
  \hat{\bs w}_{0} - \bs M \hat{\bs w}_{1}  = \hat{\bs g}_1, \;\;
  \hat{\bs w}_{1} - \bs M \hat{\bs w}_{2} = \hat{\bs g}_{2}, \;\;
  \hat{\bs w}_{2} - \sigma \bs M \hat{\bs w}_{0} = \sigma \hat{\bs g}_0,
\end{split}
\end{equation}
where the tri-diagnal matrix $\bs M$ is given in \eqref{eq:Mat1st} and the column-$N$ vectors $\hat{\bs g}_i$ have the entries
$\hat{\bs g}_{i, j} = (u_i, \phi_j^*)$ for  $0\le j\le N-1$ and $i=0, 1,2$.
We  eliminate $  \hat{\bs w}_{1}$ and $  \hat{\bs w}_{2}$ from \eqref{eq:IVPhighSys}, leading to
\begin{equation}\label{eq:IVP3rd2-mat}
  \big(\bs I_N- \sigma \bs M^3\big) \hat{\bs w}_0 =  \sigma \bs M^{2} \hat{\bs g}_0 + \bs M \hat{\bs g}_2 + \hat{\bs g}_1.
\end{equation}
In practice, we directly solve \eqref{eq:IVP3rd2-mat} and then obtain the numerical solution of the IVP \eqref{eq:IVP3rd} by
\begin{equation}\label{eq:IVP3rd2-uN}
u(t)\approx u_N(t):=u_0+w_{0,N}(t)=u_0+\sum_{k=0}^{N-1} {\hat w}_{0, k} \phi_k(t).
\end{equation}
\begin{rem}\emph{This approach should be in contrast to  the (direct) LDPG scheme \eqref{eq:DualSpec3rd}-\eqref{3rdLDPGs}. In fact, the scheme
\eqref{eq:IVP3rd2-mat}-\eqref{eq:IVP3rd2-uN} is easier to implement and  extend to higher-order IVPs, as it only involves  the matrix of the first-order IVP. However
the  initial conditions of derivatives are not exactly imposed, while the solution of the scheme \eqref{eq:DualSpec3rd}  exactly meets all initial conditions.}
\end{rem}

As a direct consequence of Theorem \ref{thm:eigMd}, we have the following exact   characterisation.
\begin{cor}\label{M3eign}   The eigenvalues of  $\bs M^3$ are   $\{\lambda_{N,j}^3 \}_{j=1}^N,$ where  $\{\lambda_{N,j} \}_{j=1}^N$ are the eigenvalues of the tri-diagonal matrix $\bs M$ with $N \ge 2$ as in
Theorem \ref{thm:eigMd}. The corresponding eigenvectors are given by \eqref{djeigen}.
\end{cor}

It is straightforward to extend the above approach to  the general  $m$th-order IVP \eqref{eq:IVPmth}.  Following the same lines as for the third-order case, we look for the numerical solution \eqref{eq:IVP3rd2-uN}  that satisfies
\begin{equation*}\label{uMlAu}
  \big(\bs I_N - \sigma \bs M^m\big) \hat{\bs w}_0 =  \sigma \bs M^{m-1} \hat{\bs g}_0 + \sum_{l=1}^{m-1} \bs M^{l-1} \hat {\bs g}_l,
\end{equation*}
where the vectors $\hat{\bs g}_l$ are defined similarly  with the entries $\hat{\bs g}_{l, j} = (u_l, \phi_j^*)$ for  $0\le j\le N-1$ and $l=0, 1,\cdots, m-1$. Then the eigenvalue distributions in Corollary \ref{M3eign} for $m=3$ can be directly extended to general $m\ge 4.$ We omit the details.

\section{Computing zeros of  GBPs}\label{sect6:zero}
It is seen from the previous sections that the eigenvalue distributions could be precisely characterized or approximately described  by zeros of  GBPs.   Thus accurate computation of  zeros of  GBPs becomes crucial as a naive evaluation of the eigenvalues of the aforementioned  non-symmetric, even tri-diagonal, matrices with double precision may suffer from severe instability even for small $N.$

Observe from \eqref{gBPexp} that for $\beta\not=0,$ we can obtain zeros of $B_n^{(\alpha,\beta)}(z)$ from those of $B_n^{(\alpha)}(z)$ (i.e.,
$\beta=2$) through a simple scaling.   Here, we restrict our attention to $B_n^{(\alpha)}(z)$  with the parameter $\alpha$ satisfying the conditions in Theorem \ref{lem:Bnzero}. We denote its $n$ distinct zeros by $\{z_j^{(\alpha)}\}_{j=1}^n,$ which are conjugate pairs, so it suffices to compute half of them.

\subsection{Algorithm in Pasquini \cite{Pasquini1994,Pasquini2000}}
The algorithm introduced by  Pasquini \cite{Pasquini1994,Pasquini2000} is for finding zeros of polynomials that satisfy a class of second-order differential equations, like  \eqref{gBP2ndODE2} for the GBPs.  The basic idea is to convert the polynomial root-finding problem into
solving a suitable system of nonlinear equations so that their zeros are identical.  Such a reformulation
appears fairly effective for the GBPs  as the three-term recurrence  \eqref{eq:recur0} is  sensitive to the range of
$z$ in the complex plane  and the normalization.  Here, we outline the algorithm  in \cite{Pasquini2000} below.

Introduce the complex coordinates $\bs z=(z_1,\cdots, z_n)\in {\mathbb C}^n$ and define the nonlinear vector-valued function
\begin{equation*}\label{eq:BnnonitB}
		\bs F(\bs z) := (F_1(\bs z),\cdots, F_n(\bs z))^\top,\quad \bs z\in {\mathbb A}^{\!n}\subset {\mathbb C}^n,
\end{equation*}
where the nonlinear complex-valued  functions and admissible set are respectively given by
\begin{equation*}\label{eq:Bnnonit}
\begin{split}
		& F_i(\bs z) := \frac {\alpha} {2z_i} + \frac 1{z_i^2}+ \sum_{j=1,\ j \neq i}^n \frac 1{z_i-z_j},\quad 1\le i\le n,\\
		&{\mathbb A}^{\!n}:=\big\{\bs z\in {\mathbb C}^n\, :\, z_i\not=z_j, \; \text{if}\;\; i\not=j; \; {\rm and}\;\; z_i\not=0,\;i,j=1,\cdots,n\big\}.
		\end{split}
\end{equation*}
According to Pasquini \cite[Theorem 4.1]{Pasquini2000}, the vector $\bs z^{(\alpha)}_*=(z_1^{(\alpha)},\cdots, z_n^{(\alpha)})^\top$ contains $n$ zeros of $B_n^{(\alpha)}(z)$ if and only if  $\bs F(\bs z^{(\alpha)}_*)=\bs 0.$ Moreover, the Jacobian matrix of $\bs F$ is nonsingular at $\bs z= \bs z^{(\alpha)}_*.$ Then a suitable iterative solver, e.g. the Newton's method, can be employed to solve the
nonlinear system. In practice, one can choose  the initial guess within the region given in \eqref{eq:annulBna0}.

\subsection{Algorithm in Segura \cite{Segura2013}}  Segura \cite{Segura2013} proposed  accurate  algorithms for computing   complex zeros of solutions of second-order linear ODEs: $w''(z)+A(z)w(z)=0$ with $A(z)$ meromorphic, based on qualitative study of the solution structures. Through suitable substitutions and connections with Laguerre functions, the ODE of the GBPs can be   transformed into this form. The  Maple worksheets in symbolic computation are available on the author's website.

\subsection{Numerical tests}
We  implement the above algorithms in {\tt Matlab} (with double precision computation), but generate the reference zeros (i.e., the eigenvalues of the Jacobi matrix for the three-term recurrence relation of the GBPs) using the Multiprecision Computing Toolbox\footnote{Multiprecision Computing Toolbox: \url{https://www.advanpix.com/}} with a sufficiently
large number of digits.  Here, we denote the reference zeros by  $\{z_j^{\texttt{mp}}\}$ and name the above two algorithms  in order as
Method 1 and Method 2, respectively.

In Figure \ref{fig:comp-eig} we depict the relative errors between  the eigenvalues computed by the two methods and the reference values
for different $\alpha$. We observe that both Pasquini's algorithm and Segura's algorithm are stable with large $N$, and the former performs slightly better than the latter.
\begin{figure}[!th]
	\centering
	\includegraphics[width=0.4\textwidth]{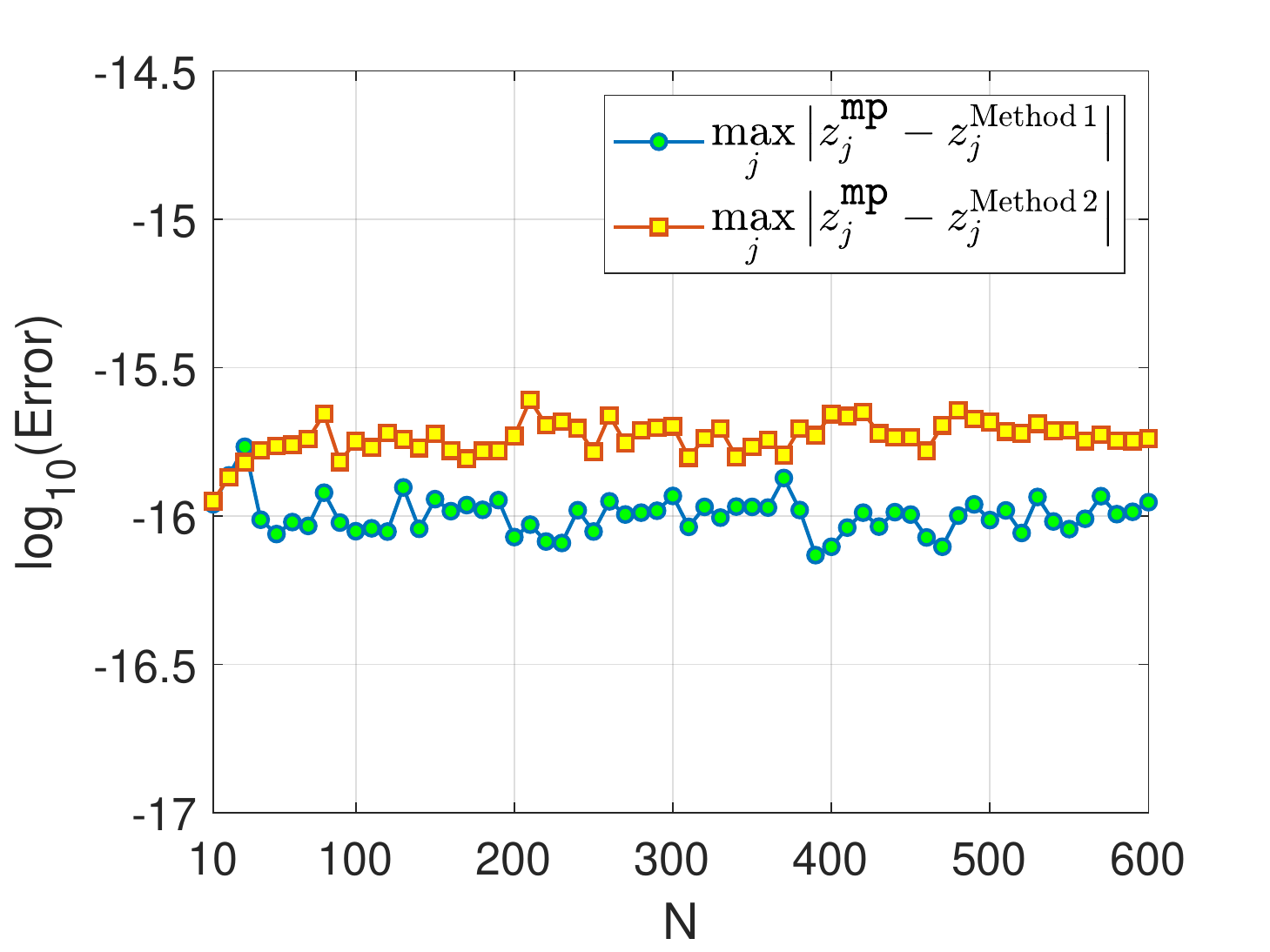}  \quad
	\includegraphics[width=0.4\textwidth]{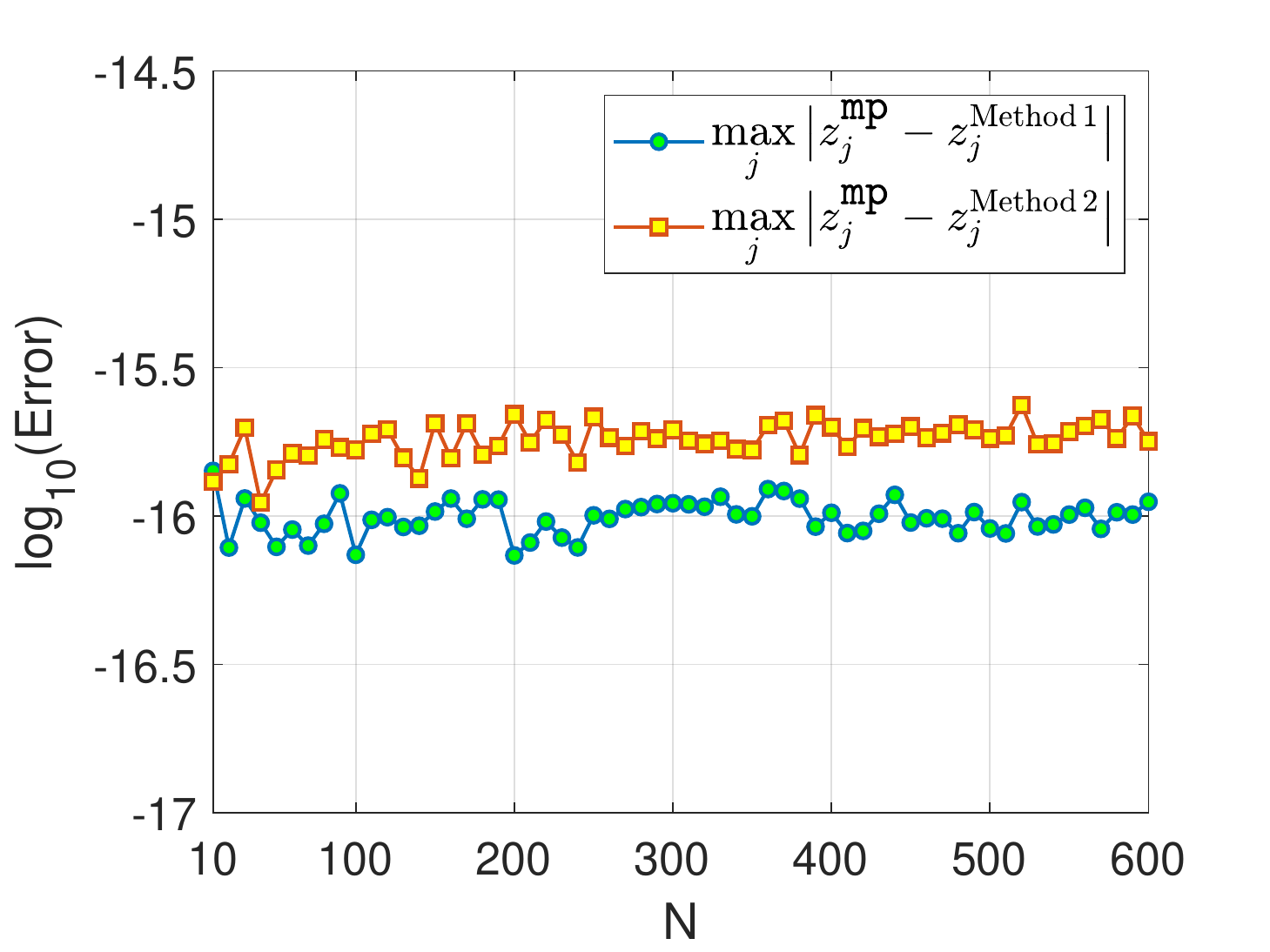}
	\vspace*{-6pt}
	\caption{\small Errors between the eigenvalues computed  by two algorithms  (with double precision) and reference values (with
	multiple precision) for $\alpha = 2$ (left) and $\alpha = 4$ (right).}
	\label{fig:comp-eig}
\end{figure}

\section{Space-time spectral methods}\label{Sect7:numerical}

The findings from  previous eigenvalue analysis have deep implications in spectral methods in time.  In this section, we first present a general framework
for the LDPG time discretisation of evolutionary PDEs with a focus on  two techniques: (i)  matrix diagonalisation and (ii) QZ decomposition (or generalised Schur decomposition) to  effectively deal with the resulted matrices.
We then apply the space-time spectral methods to some interesting linear and nonlinear wave problems that require high accuracy in both space and time.

\subsection{A general framework for the LDPG spectral method in time}\label{Subsect:gen}
We demonstrate the algorithm via the system of linear ODEs:
\begin{equation}\label{ODEsystemB}
\bs u'(t)+\bs A \bs u(t)=\bs f(t), \;\;  t\in (-1,1);\quad \bs u(-1)=\bs u_0,
\end{equation}
where  $\bs u, \bs f, \bs u_0 \in {\mathbb R}^{N_x}$ and $\bs A\in {\mathbb R}^{N_x\times N_x}$ (which might be resulted from spatial discretisations).

Using the LDPG scheme described in Subsection \ref{subsec:d1st} to \eqref{ODEsystemB}, we seek the LDPG spectral approximation in terms of the basis
\eqref{eq:dPetbas},
\begin{equation*}\label{vectexp}
\bs u(t)\approx \bs u_*(t)=\bs u_0+\sum_{k=0}^{N_t-1} \hat {\bs  u}_k \phi_k(t),
\end{equation*}
and denote the matrix of unknowns by $ \widehat {\bs  U}=(\hat {\bs  u}_0,\cdots, \hat {\bs  u}_{N_t-1})^\top \in {\mathbb R}^{N_t\times N_x}.$ Then
 the corresponding linear system reads
\begin{equation}\label{linearVec}
\widehat {\bs  U}+\bs M_t\, \widehat {\bs  U} \bs A^\top =\bs F,
\end{equation}
where $\bs M_t\in  {\mathbb R}^{N_t\times N_t}$  is the same as in \eqref{eq:dodesystem} and
$$\bs F= \bs U_0 \bs A^\top + (\bs { \hat {f}}_0,\cdots, \bs { \hat {f}}_{N_t-1})^\top \in {\mathbb R}^{N_t\times N_x},\quad
 \bs U_0=\sqrt{2} \bs e_1 \bs u_0^\top,\;\;  \bs { \hat {f}}_j = (\bs f, \phi_j^*),$$
with the basis functions $\{\phi_j^*\}$ given in \eqref{eq:dPetbas2}.

The linear system  \eqref{linearVec} in a moderate scale can be solved directly by rewriting it in the Kronecker product form, but  it is costly  in particular for multiple spatial dimensions in space, as the time discretisation is globally implicit.
  In the following, we introduce  two techniques  to alleviate this burden.
\subsubsection{Matrix diagonalisation}
Let
$$\bs E = (\bs v_1, \cdots, \bs v_{N_t}),\quad \bs \Lambda = {\rm diag}(\lambda_1, \cdots, \lambda_{N_t}), \;\;\; \text{so} \;\;\;     \bs M_t \, \bs E = \bs E \bs \Lambda,$$
 where $\{\lambda_j, \bs v_j\}_{j=1}^{N_t}$ are  eigen-pairs of $\bs M_{t}$ given in Theorem \ref{thm:eigMd}. Introducing the substitution $\widehat{\bs U} = \bs E \bs W,$
  we obtain from \eqref{linearVec} that
 \begin{equation}\label{eq:diag-algo0}
    \bs W - \bs \Lambda \bs W \bs A^\top = \bs E^{-1} \bs F =: {\bs G},
  \end{equation}
which can be decoupled and solved in parallel
\begin{equation}\label{eq:diag-algo}
    \big( \bs I_{\! N_x} - \lambda_j \bs A \big) \bs w_j = \bs g_j, \quad 1 \le j \le N_t,
  \end{equation}
 where $\bs w_j, \bs g_j$ are the $j$th column of $\bs W^\top, \bs G^\top,$ respectively.

 \begin{rem}\label{Diagrem}\emph{Different from a normal matrix,
we see from \eqref{eq:diag-algo0} that  the diagonalisation of a non-normal matrix involves the
inverse of the eigenvector matrix $\bs E$. In practice, we can evaluate $\bs G$ by solving the linear system
$ \bs E {\bs G}=\bs F$ via a suitable iterative solver. It is known that the stability and round-off errors essentially rely on the conditioning of $\bs E.$ Unfortunately, this eigenvector matrix is extremely ill-conditioned with an exponential growth condition number, see Figure  \ref{fig:oneway3} below. As a result, this technique needs to be implemented in  a multi-domain manner with relatively small $N_t$ on each sub-domain.}
 \end{rem}

 \subsubsection{QZ  {\rm(}or generalised Schur{\rm)} decomposition}
 Recall that given square matrices $\bs A$ and $\bs B$, the QZ decomposition factorizes both matrices as $\bs Q \bs A \bs Z=\bs S$ and
 $\bs Q \bs B \bs Z=\bs T,$ where $\bs Q$ and $\bs Z$ are unitary (i.e., $\bs Q^{-1}=\bs Q^{\rm H}, \bs Z^{-1}=\bs Z^{\rm H},$ where $\bs Q^{\rm H}$ stands for the  conjugate  transpose  or Hermitian transpose of $\bs Q$), and $\bs S$ and $\bs T$ are upper triangular.

We now  apply the QZ decomposition to $\bs A=\bs I_{\! N_t}$ and $\bs B=\bs M_t,$ and obtain
  \begin{equation}\label{eq:qz-deco}
    \bs Q \bs Z = \bs S,\quad  \bs Q \bs M_{\! t} \bs Z = \bs T.
  \end{equation}
  Introducing  $\widehat{\bs U} = \bs Z \bs W,$ and multiplying both sides of   \eqref{linearVec} by $\bs Q,$ we obtain from
  \eqref{eq:qz-deco} that
  \begin{equation}\label{zwequation}
    \bs S \bs W - \bs T \bs W \bs A^\top = \bs Q \bs F =: \bs G.
  \end{equation}
 As $\bs S,\bs T$ are upper triangular matrices, direct backward substitution reduces   \eqref{zwequation} to
\begin{equation}\label{QZsolver}
  \bs w_j \big( S_{jj} \bs I_{\! N_x} -  T_{jj} \bs A^\top \big) = \bs g_j - \bs r_j,
  \quad j = N_t, N_t-1, \cdots, 1,
\end{equation}
where  $\bs w_j, \bs g_j$ are the $j$th row of $\bs W, \bs G$, and
\begin{equation*}\label{QZrj}
\bs r_{N_t}=\bs 0,\quad   \bs r_j = \sum_{k=j+1}^{N_t} \bs w_k \big(S_{jk} \bs I_{\! N_x} -  T_{jk} \bs A^\top \big), \;\;\; j=N_t-1,\cdots, 1.
\end{equation*}
Here, $S_{jk},T_{jk}$ with $j\le k$ are the entries of  $\bs S, \bs T,$  respectively.

\begin{rem}\label{sequAlg} \emph{The QZ decomposition  is stable for large $N_t$ and  different from \eqref{eq:diag-algo0},
the step \eqref{zwequation} does not involve any matrix inversion.  Moreover,
 this technique can be assembled into a multi-domain approach  and  march sequentially in time. However,
as $\{\bs r_j\}$ involves $\bs w_k$ with $k>j,$    the linear systems \eqref{QZsolver} can only be solved sequentially.}
\end{rem}

\subsection{Space-time spectral methods for linear and nonlinear wave propagations}

\subsubsection{A linear wave-type equation}  As the first example of applications,
we consider the linear wave-type equation
  \begin{equation}\label{eq:oneway1}
    \begin{dcases}
      \partial_{x t}^2 u(x, t) + \sigma u(x, t) = 0,\quad  x\in \Omega:=(x_L, x_R), \;\; t\in (0,T], \\
      u(x, 0) = u_0(x),\;\; x\in \Omega;\quad u(x_L, t)=0,\;\; t\in [0,T],
    \end{dcases}
  \end{equation}
for a given constant $\sigma > 0.$
It  is of hyperbolic type since
 under the transformation,
 we have
\begin{equation*}\label{uxthyper}
\partial_{xt}^2 u(x,t)= \partial_{\xi}^2 v(\xi,\eta)- \partial_{\eta}^2 v(\xi,\eta),\quad u(x,t):=v(x+t, x-t).
\end{equation*}
As shown in   \cite{WKW22}, it  is related to the linear
  Kadomtsev-Petviashvili (KP) model, and the wave propagates along one side if the initial data is compactly supported.
  More precisely, if $\sigma>0$ and the initial value $u_0(x)$ is  compactly supported in $(x_0, \infty),$ then
$u(x, t)= 0$ for all $x< x_0$ and $t\ge 0.$
Thus, we can impose the left-sided boundary condition in \eqref{eq:oneway1}, and then simulate the waves in a finite domain. We transform \eqref{eq:oneway1} in $\Omega\times (0,T)$ to the reference square $\Lambda:=(-1,1)^2$ by simple linear transformations that lead to the same equation but with $\hat \sigma={4\sigma}/((x_R-x_L)T)$ in place of $\sigma.$

Applying the LDPG method in Subsection \ref{subsec:d1st} for spatial discretisation, we can obtain the system \eqref{ODEsystemB} of the form
\begin{equation}\label{ODEsystemB00}
\bs u'(t)+ \hat\sigma\bs M_x \bs u(t)=\bs 0, \;\;  t\in (-1,1);\quad \bs u(-1)=\bs u_0,
\end{equation}
where $\bs M_x\in  {\mathbb R}^{N_x\times N_x}$  is given in \eqref{eq:dodesystem}. Then the general setup detailed in Subsection
\ref{Subsect:gen} directly carries over to  \eqref{ODEsystemB00}.
We now test the proposed LDPG space-time spectral method  and choose
  \begin{equation}\label{eq:u0-gauss}
    u_0(x) = {\rm sech}^2 \big( \sqrt{3}(x+35)/6 \big) - {\rm sech}^2( 5\sqrt{3}/2), \quad x \in \Omega=(-50, 50).
  \end{equation}

\begin{figure}[!ht]
  \centering
  \subfigure[Diagonalisation (multi-domain in $t$)]{\includegraphics[width=.32\textwidth]{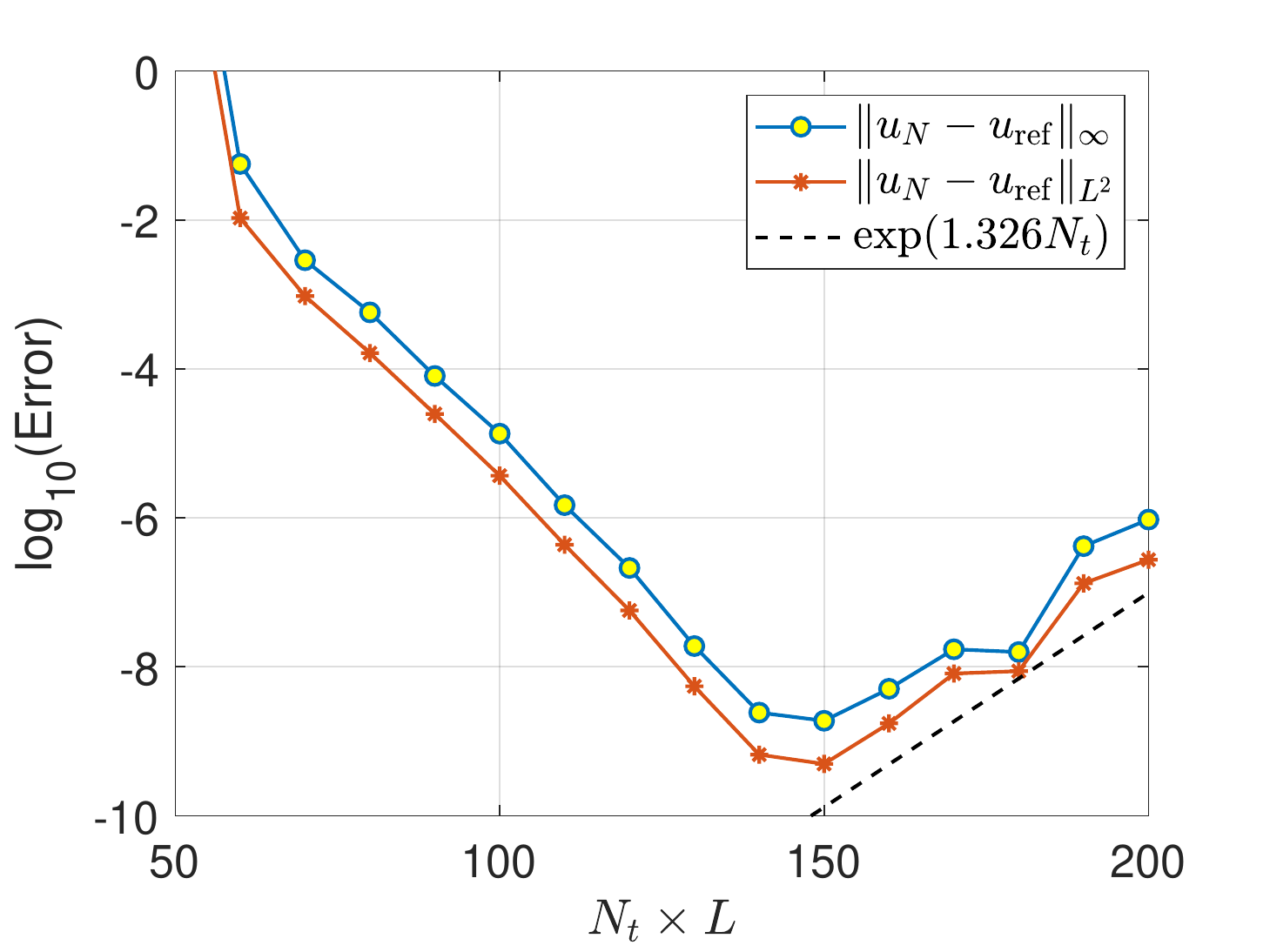}} \quad
  \subfigure[Conditioning of $\bs E$]{\includegraphics[width=.32\textwidth]{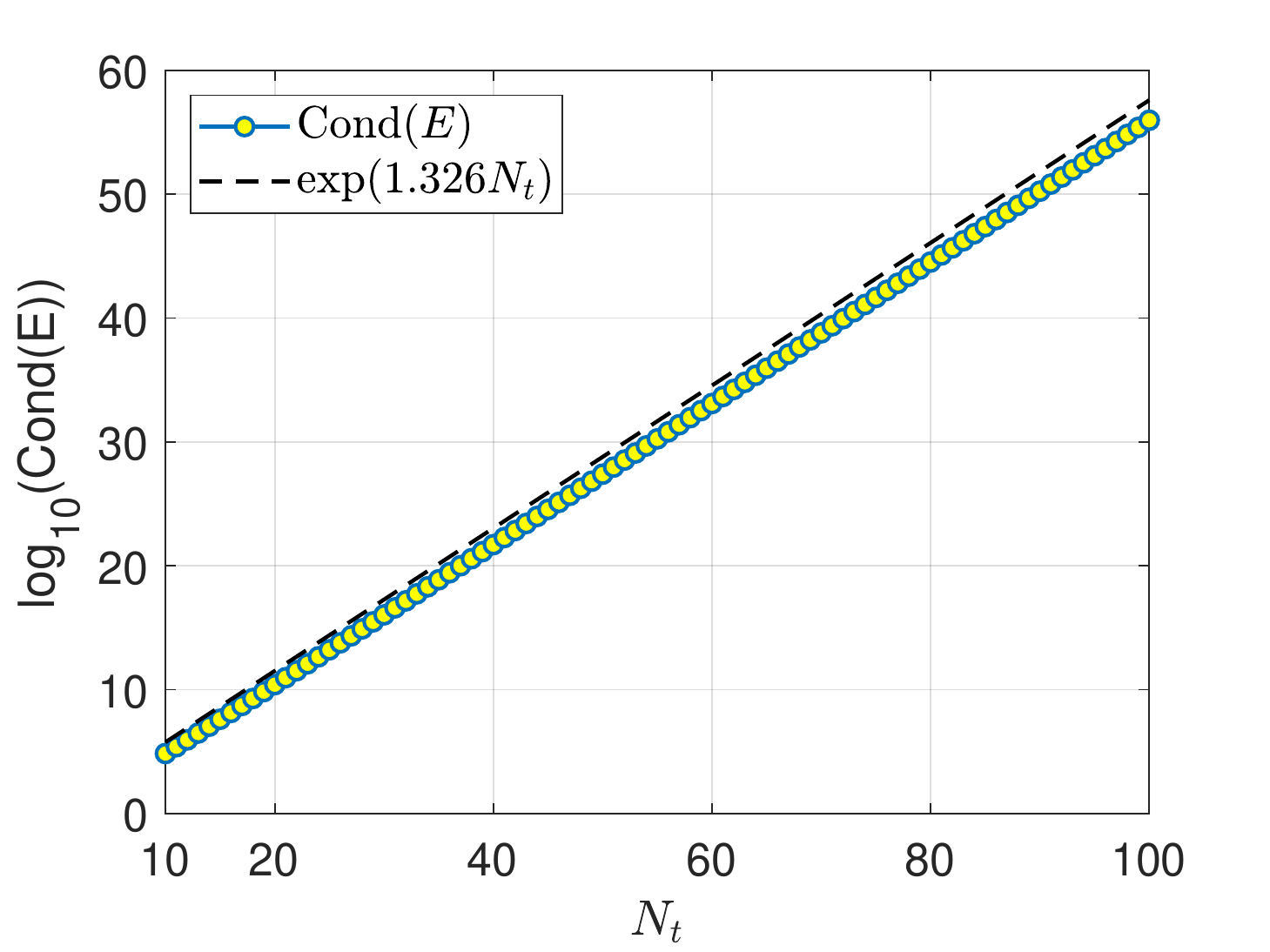}}
   \vspace*{-10pt}

\subfigure[QZ  (multi-domain in $t$)]{\includegraphics[width=.32\textwidth]{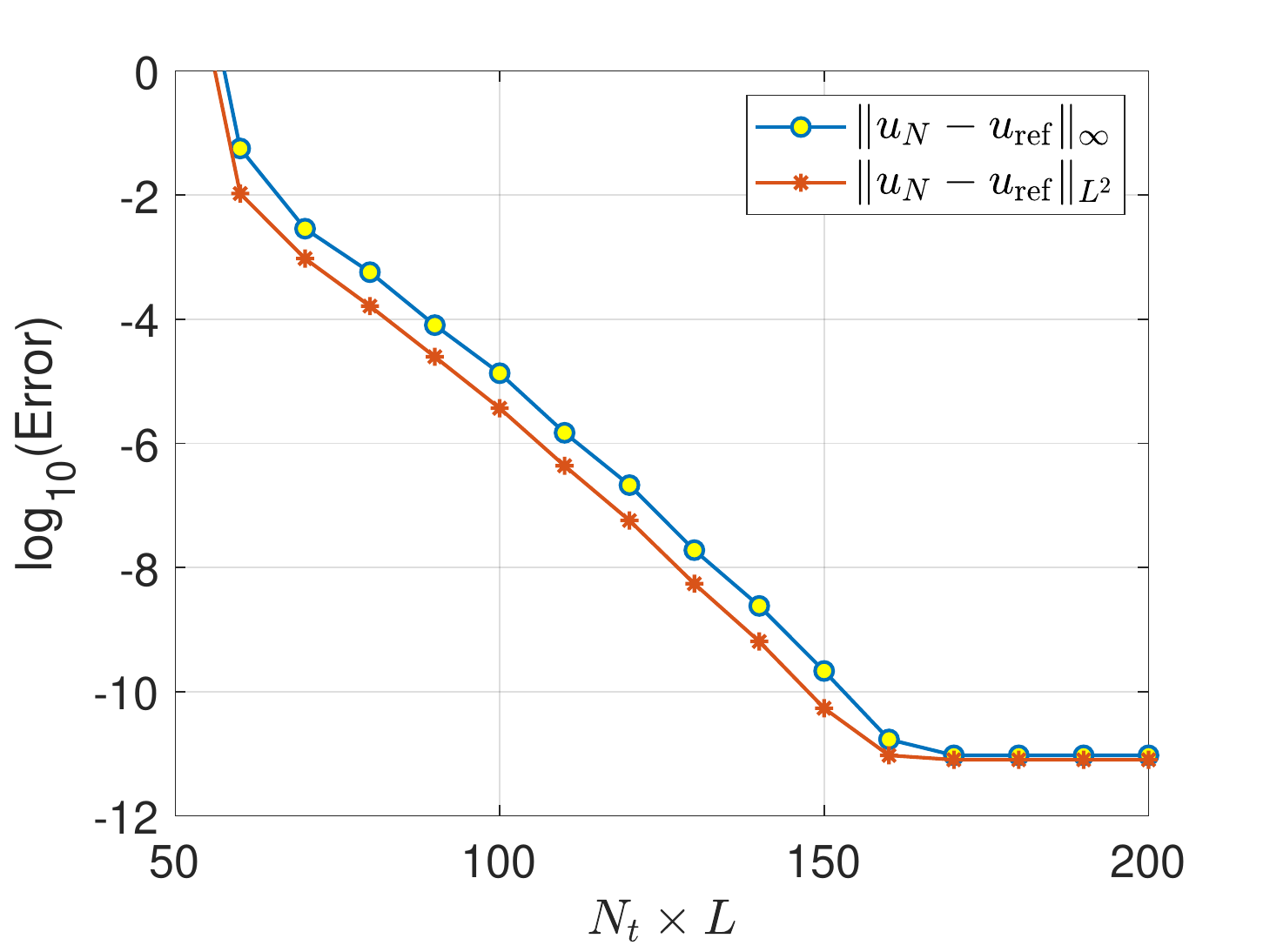}} \quad
  \subfigure[QZ  (single domain in $t$)]{\includegraphics[width=.32\textwidth]{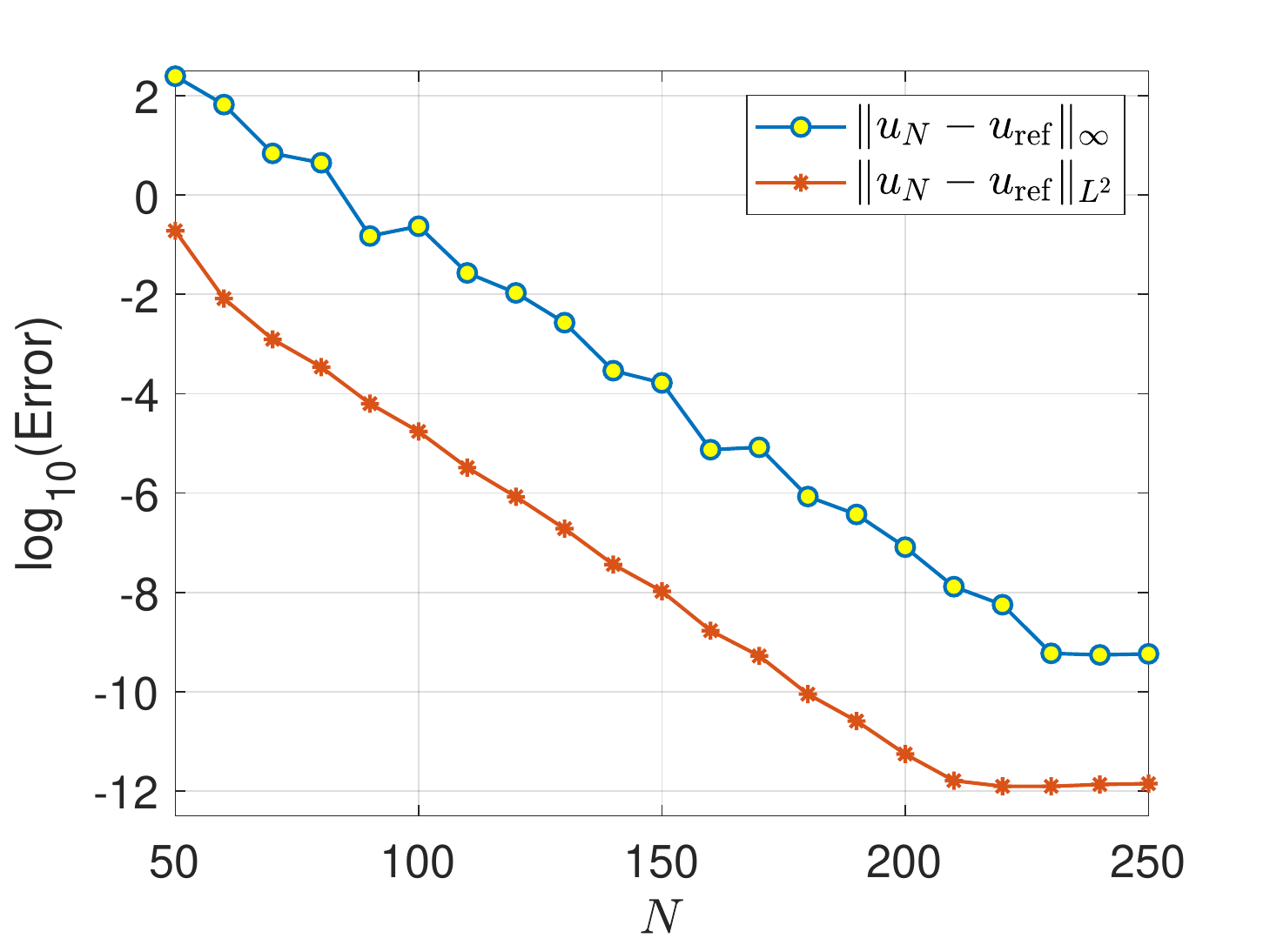}}
  \vspace*{-6pt}
  \caption{\small Convergence of space-time spectral methods and  conditioning of the eigenvector matrix (in (b)).\!
   (a) Diagonalisation  technique;  (c) QZ decomposition with a multi-domain implementation: $L=10$ subintervals of $t\in [0,40]$ with various
   $N_t\in [5,20]$ and $N_x=400.$  (d) Single domain space-time spectral method with various $N=N_t=N_x.$}
  \label{fig:oneway2}
\end{figure}

  We first check the accuracy of the solvers with the reference solution (denoted  by $u_{\rm ref}$) obtained from the scheme with sufficiently large $N_t = N_x=400$
  and up to $T=40$ in time.
As the diagonalisation technique works for small $N_t$, we partition the time interval $[0,T]$ equally into $L$ sub-intervals (but with the same  $N_t$ on each subinterval).   Now, we vary  $N_t$ from $5$ to $20$ and fix $L=10$ and $N_x=400.$  We plot in
Figure \ref{fig:oneway2} (a)  (resp.  Figure \ref{fig:oneway2} (c)) the discrete $L^\infty$- and $L^2$-errors for the
spectral solver with diagonalisation (resp. QZ decomposition)  in the semi-log scale. We observe from Figure \ref{fig:oneway2} (a) that
  the errors of the diagonalisation technique  for $N_t$ up to $15$ decay exponentially,   but grow exponentially  for $16\le N_t\le 20.$ Surprisingly, the growth rate agrees well with that of the condition number of the eigenvector matrix $\bs E,$ illustrated in Figure \ref{fig:oneway2} (b). However, the QZ decomposition in the same setting is stable and accurate for all  samples of $N_t,$
  as shown in Figure \ref{fig:oneway2} (c).  We also depict in Figure \ref{fig:oneway2} (d) the errors of the QZ decomposition with a single domain in $t$ and $N=N_t=N_x$ for various $N,$ from which we observe an exponential convergence rate.

We now use the space-time spectral solver with QZ decomposition in time  to simulate the wave propagations.
In the computation, we take $N_t=N_x=400.$  Observe from Figure \ref{fig:oneway3} (a) that the initial input \eqref{eq:u0-gauss} immediately disperses and spreads  in both space and time that leads to oscillations and the waves decay in space at certain
algebraic rate (refer to \cite{WKW22} for detailed analysis).

  \vspace*{-10pt}
\begin{figure}[!ht]
  \centering
  \subfigure[$u_N(x,t)$ with $\sigma=1$ ]{\includegraphics[width=.27\textwidth]{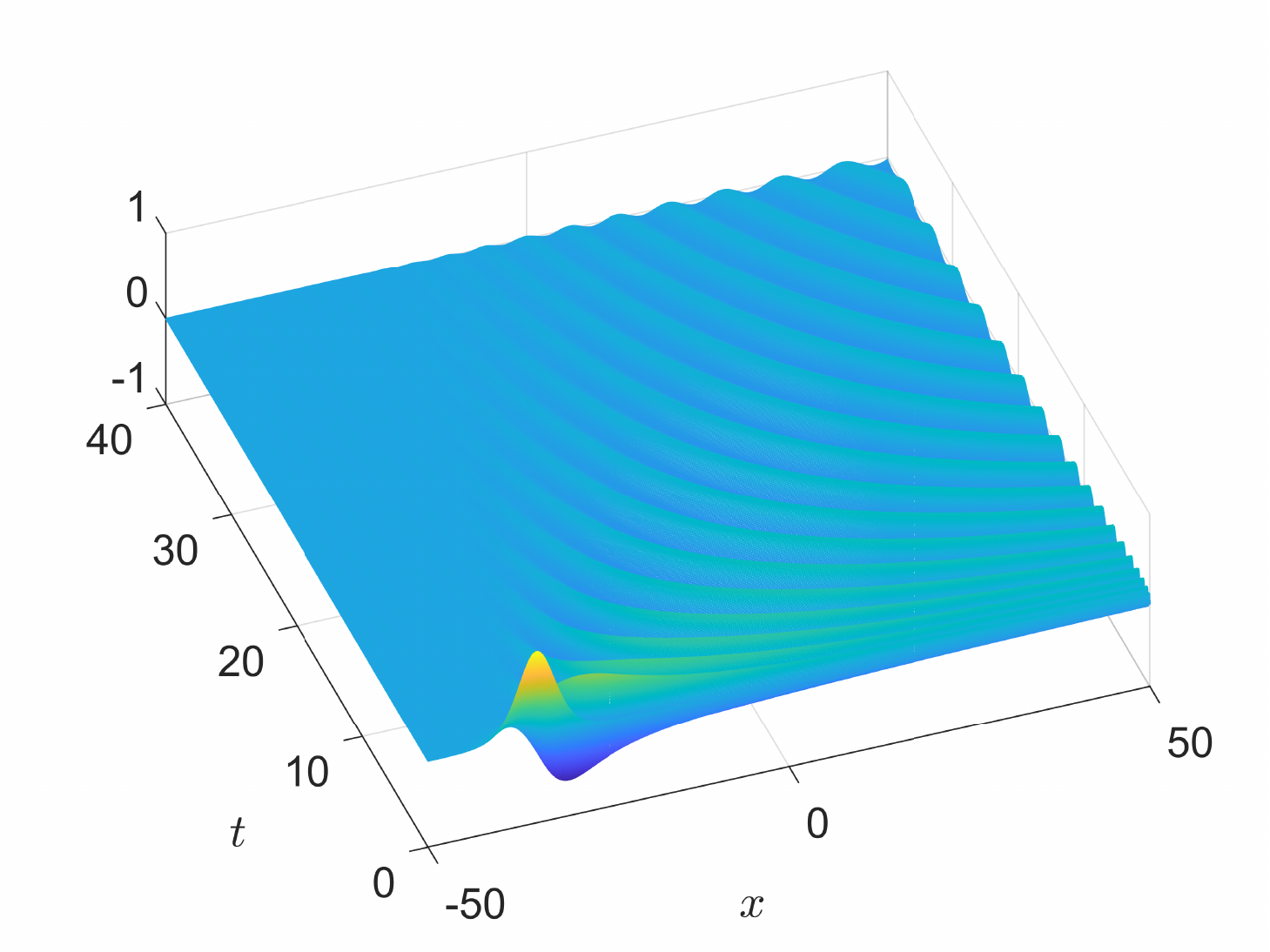}} \quad
  \subfigure[$u_N(x,40)$ with different $\sigma$]{\includegraphics[width=.27\textwidth]{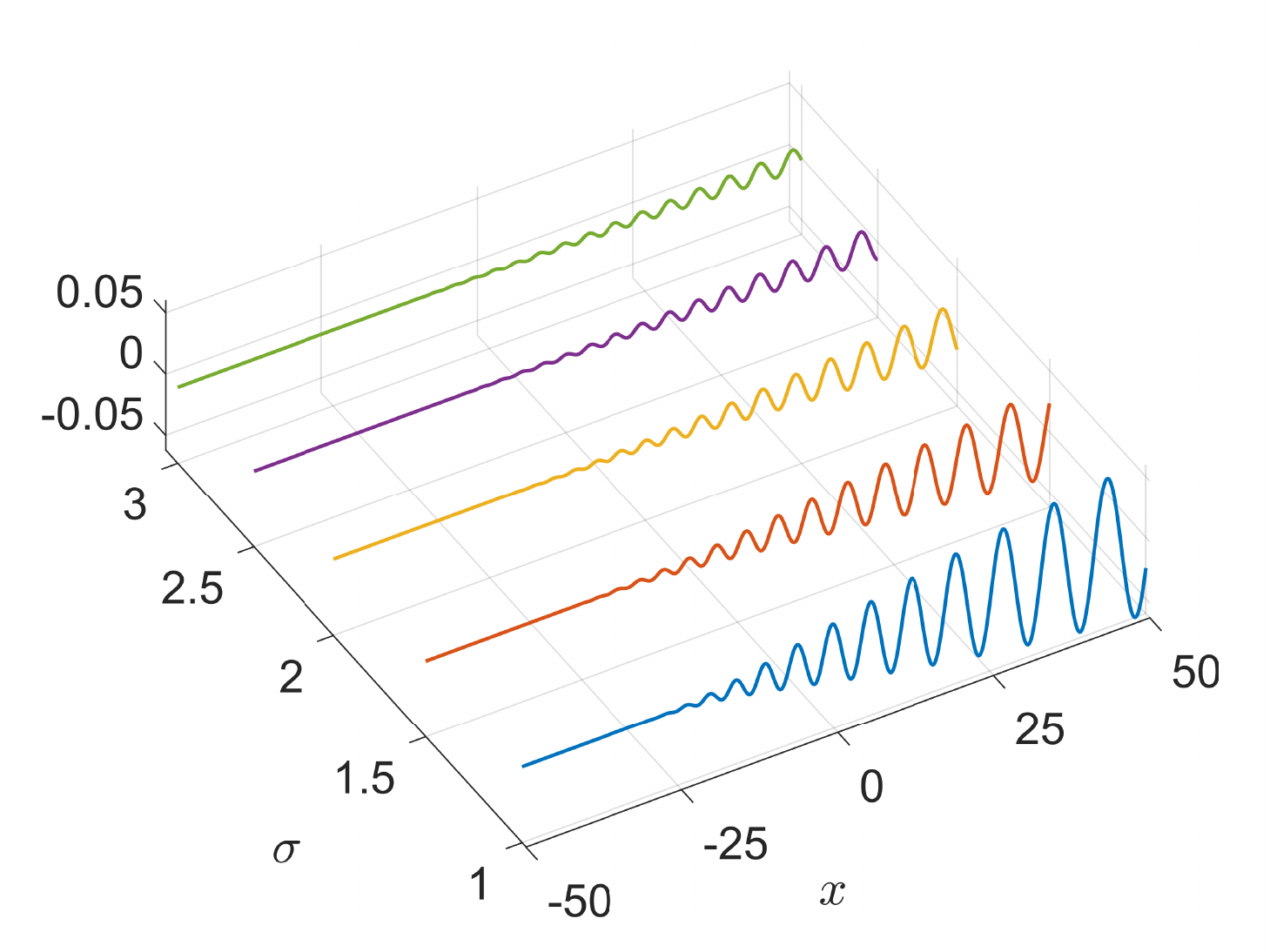}} \quad
  \subfigure[$u_N(50,t)$ with different $\sigma$]{\includegraphics[width=.27\textwidth]{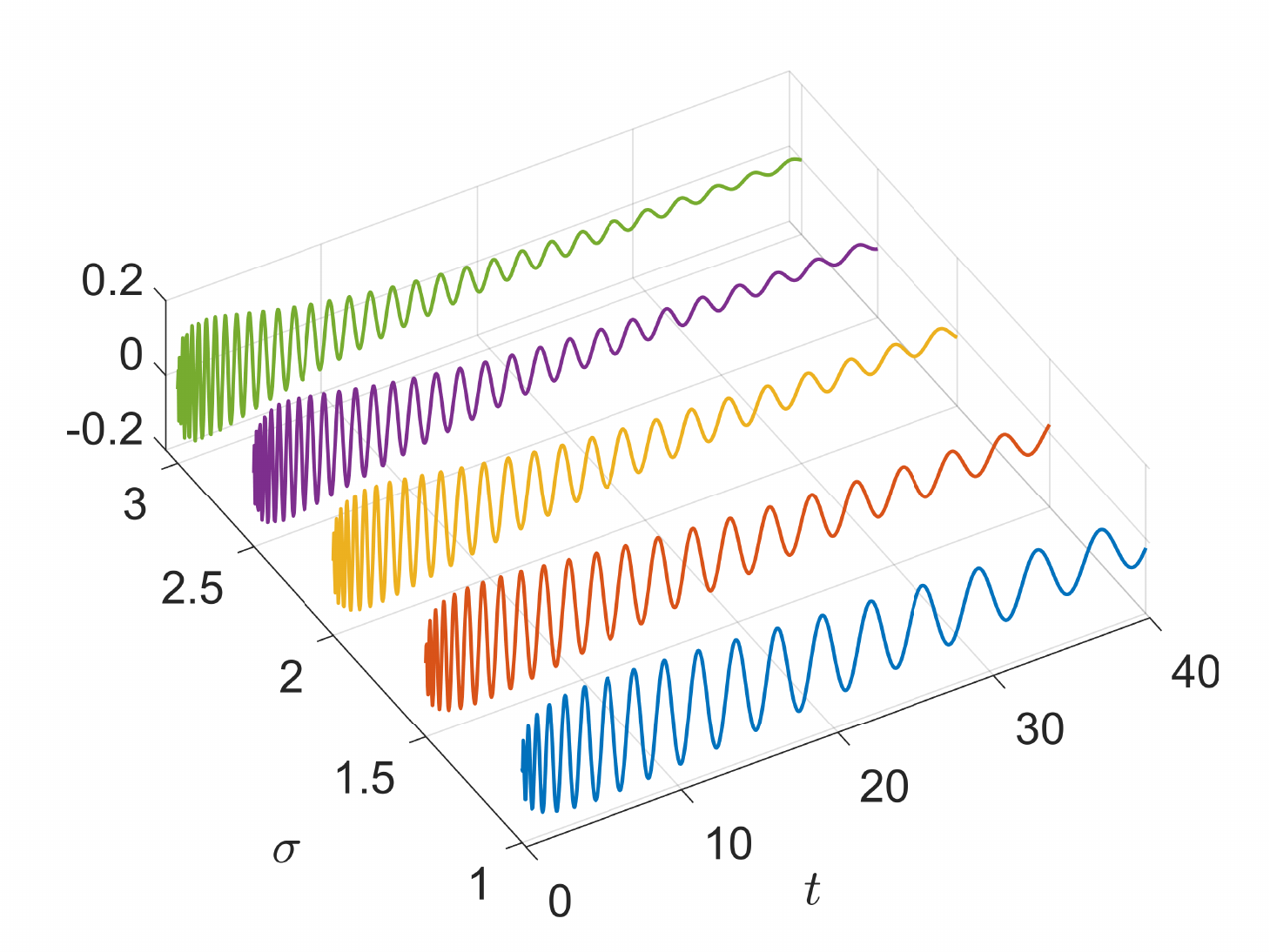}}
  \vspace*{-6pt}
  \caption{\small Numerical solutions $u_N(x,t)$  obtained by  \eqref{eq:oneway1} given the initial profile \eqref{eq:u0-gauss}
   with $N_x=N_t=400$ and different values of the parameter $\sigma.$
   (a)  Numerical solution with $\sigma=1$;  (b)-(c) Profiles of the numerical solution at fixed $t$ or $x$ but with different $\sigma.$}
  \label{fig:oneway3}
\end{figure}

It is important to point out that the coefficient matrices of linear systems  \eqref{eq:diag-algo} and \eqref{QZsolver} in space are well-conditioned.  In Table \ref{tab:oneway4}, we tabulate  the condition numbers and eigenvalues with the minimum and maximum moduli of
the matrix $\bs A:=\bs I_{\! N_x} + \bs M_{x}.$
  \begin{table}[!ht]
\centering
\caption{\small Conditioning and eigenvalues  of $\bs A=\bs I_{\! N_x} + \bs M_{x}.$}\label{tab:oneway4}\small
\vspace*{-6pt}
\begin{tabular}{|c|cccccc|}
\hline
$N_x$ & $20$ & $30$ & $40$ & $50$ & $100$ & $200$ \\
\hline
${\rm Cond} (\bs A)$ & 1.8730 & 1.8730 & 1.8730 & 1.8730 & 1.8730 & 1.8730 \\
$\min_j |\lambda_j^{\bs A}|$ & 1.0143 & 1.0071 & 1.0043 & 1.0029 & 1.0009 & 1.0003 \\
$\max_j |\lambda_j^{\bs A}|$ & 1.0708 & 1.0513 & 1.0539 & 1.0533 & 1.0513 & 1.0514  \\
\hline
\end{tabular}
\end{table}

\vspace*{-10pt}

\subsubsection{A nonlinear KdV-type equation}  As a  second example, we apply the space-time LDPG method to the KdV-type equation:
\begin{equation}\label{KdVmodelE}
\begin{dcases}
\partial_t U+\alpha U \partial_x U+\epsilon^2 \partial_{x}^3 U+\sigma \partial_x^{-1} U=0,\quad  & x\in \Omega:=(x_L, x_R),\;\; t \in (0,T]\\
U(x_L, t)=U(x_R, t)=\partial_x U(x_R, t)=0,\quad & t\in[0,T],\\
U(x,0)=U_0(x),\quad & x\in \bar \Omega,
\end{dcases}
\end{equation}
where $\alpha\ge 0, \epsilon>0$ and $\sigma\not=0$ are constants,  and
 \begin{equation*}\label{pinvU}
 \partial_x^{-1} U(x,t)=\frac 1 2 \Big(\int_{x_L}^x U(y,t) {\rm d} y- \int_{x}^{x_R} U(y,t) {\rm d} y\Big).
\end{equation*}
It also relates to the reduced KP equations \cite{WKW22}. Like before, we covert \eqref{KdVmodelE} to the reference square $(-1,1)^2,$
 and then follow Shen \cite{Shen2003SINUM} to formulate the LDPG in space involving the pair of dual spaces
\begin{equation*}\label{mVnVnstar}
\begin{split}
{\mathbb  V}_{\!N_x} &:= \big\{ \phi \in \mathbb P_{N_x+2}: \phi(-1) = \phi(1) = \phi'(1) = 0\big \}, \\
{\mathbb  V}_{\!N_x}^* &:= \big\{ \psi \in \mathbb P_{N_x+2}: \psi(-1) = \psi(1) = \psi'(-1) = 0\big\}.
\end{split}
\end{equation*}
We omit the details on the expressions of the basis functions and formulation of the matrix form. In the numerical tests, we use the Newton's iteration to solve the resulting nonlinear system when $\alpha\not=0.$
Consider \eqref{KdVmodelE} with the initial condition  given by
\begin{equation}\label{eq:u0-sech}
  U_0(x) = {\rm sech}^2 \big( \sqrt{3}(x+5)/6 \big),\quad x\in \Omega=(-50, 50),
\end{equation}
 taken from the soliton solution of the KdV equation (i.e., \eqref{KdVmodelE} with $\alpha=\epsilon=1$ and $\sigma=0$):
\begin{equation*}\label{KdVsolu}
U(x, t) = {\rm sech}^2 \big( \sqrt{3} (x- t/3 + 5)/6 \big).
\end{equation*}
We  first show the accuracy by solving the KdV equation
in the domain $\Omega\times (0,20).$
In Figure \ref{fig:kdv-sol00}, we plot the numerical solution and the discrete $L^\infty$- and $L^2$-errors with $N_t=80$ and various $N_x=50:10:200,$ which shows a spectral accuracy.
\begin{figure}[!h]
  \centering
  \includegraphics[width=.32\textwidth]{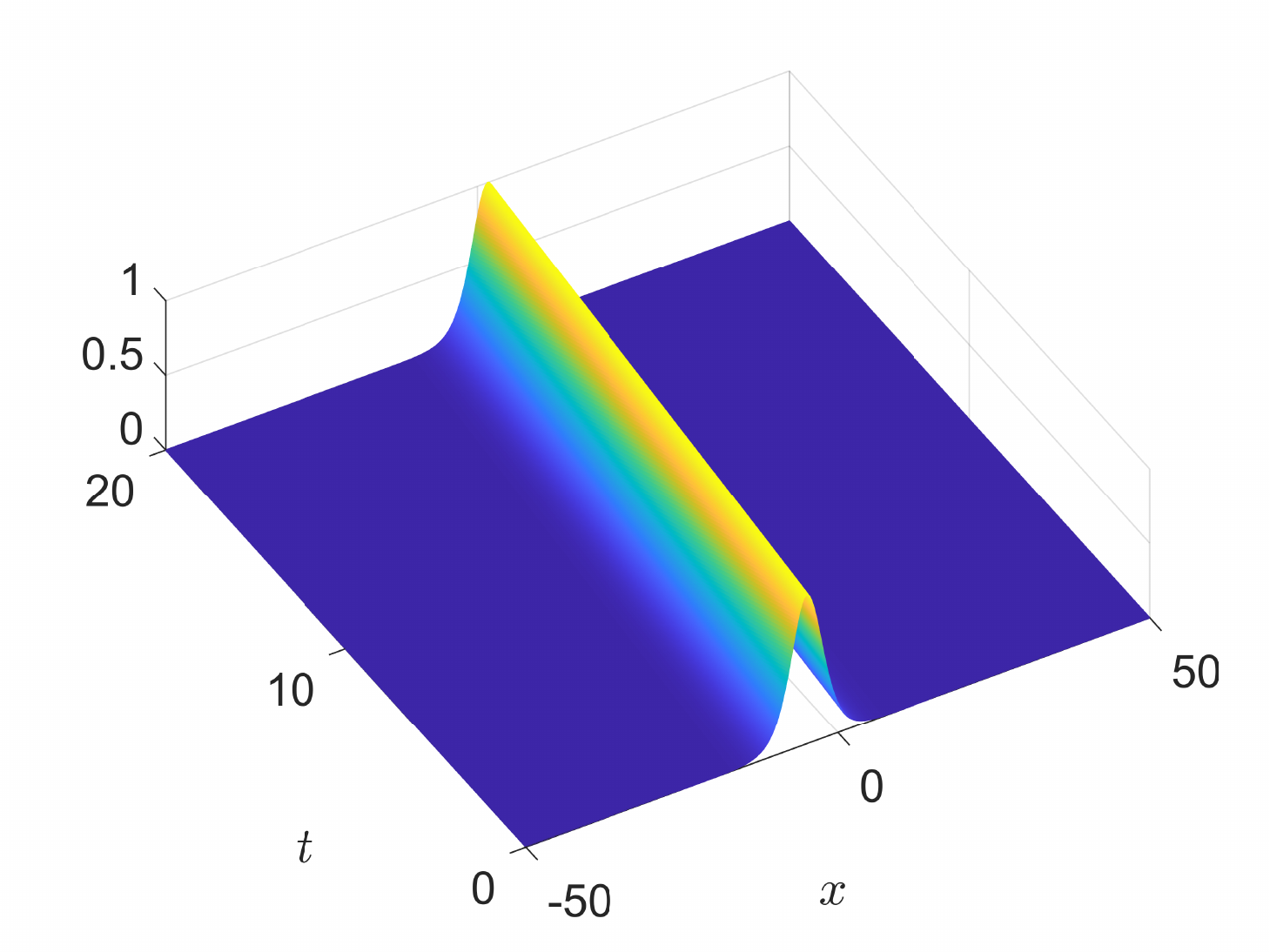} \quad \includegraphics[width=.32\textwidth]{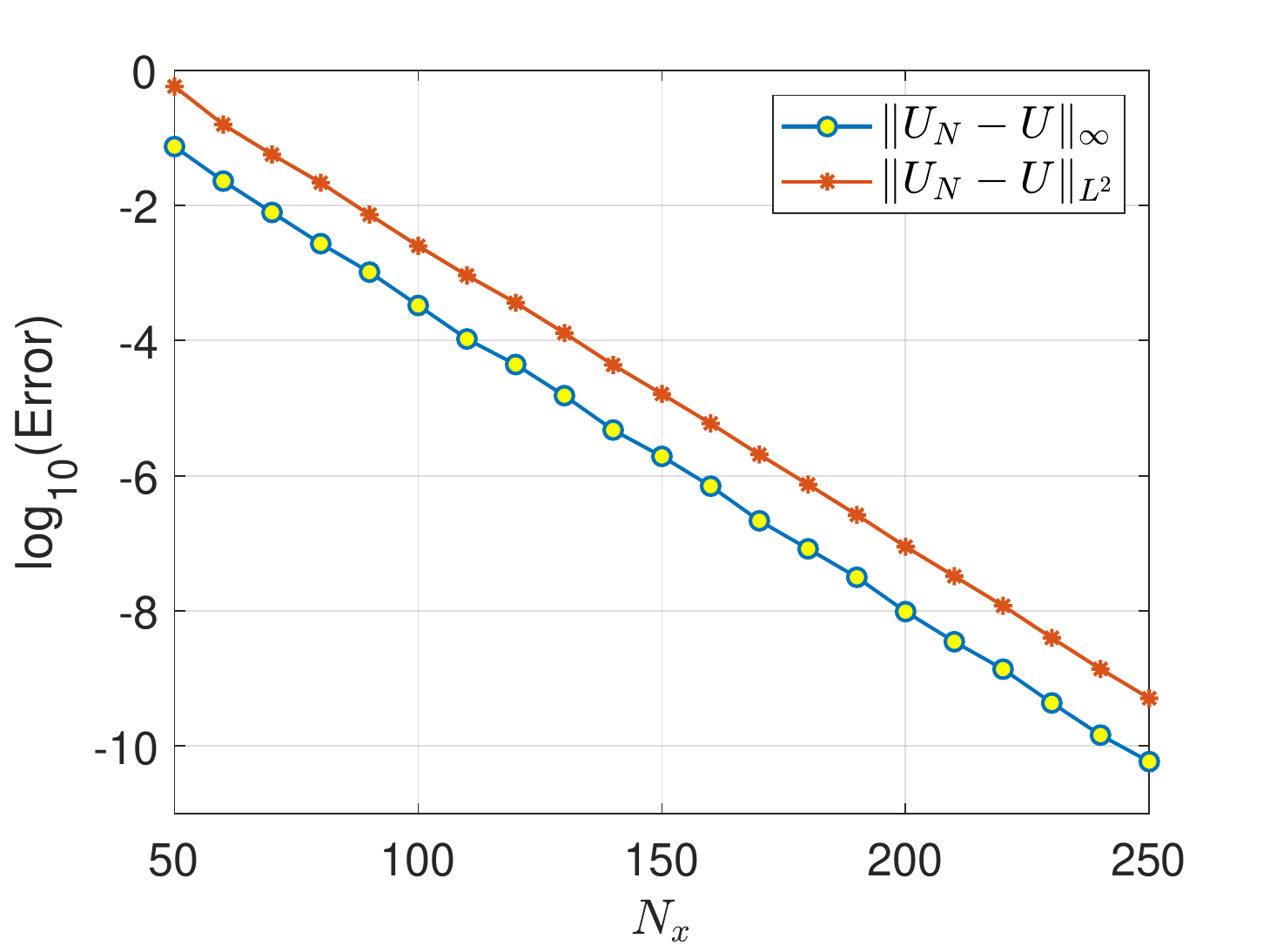}
  \caption{\small Left: numerical solution $u_N(x,t)$ (with $N_x=250, N_t=80$) of the KdV equation (i.e.,  \eqref{KdVmodelE} with $\alpha=\epsilon=1$ and $\sigma=0$) for given initial input in \eqref{eq:u0-sech}. Right:
 Errors against $N_x=50:10:250$ and fixed  $N_t=80$ in semi-log scale.}
  \label{fig:kdv-sol00}
\end{figure}

\begin{figure}[!ht]
  \centering
  \subfigure[Linear KdV]{\includegraphics[width=.32\textwidth]{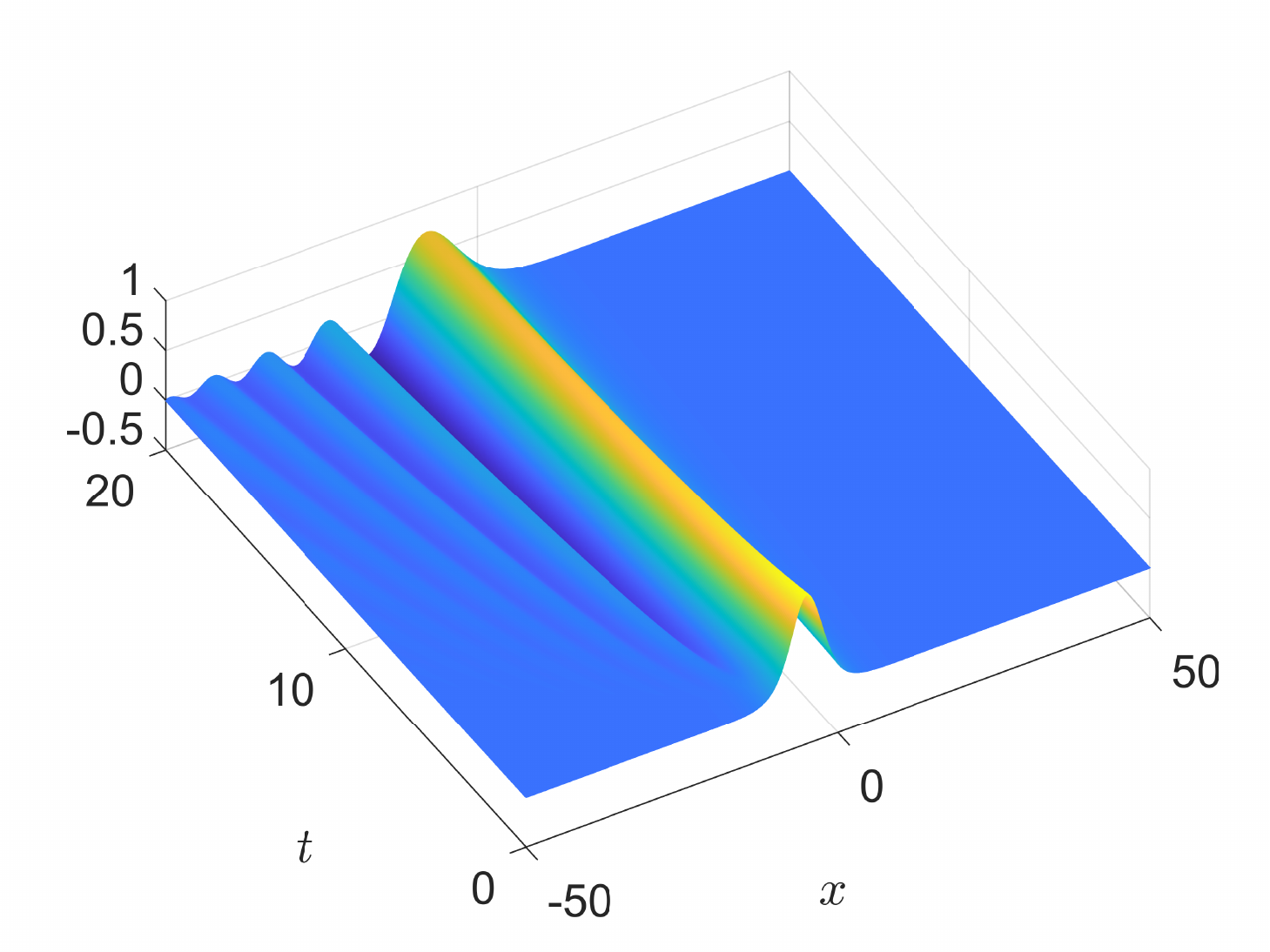}} \quad
  \subfigure[KdV with a nonlocal term]{\includegraphics[width=.32\textwidth]{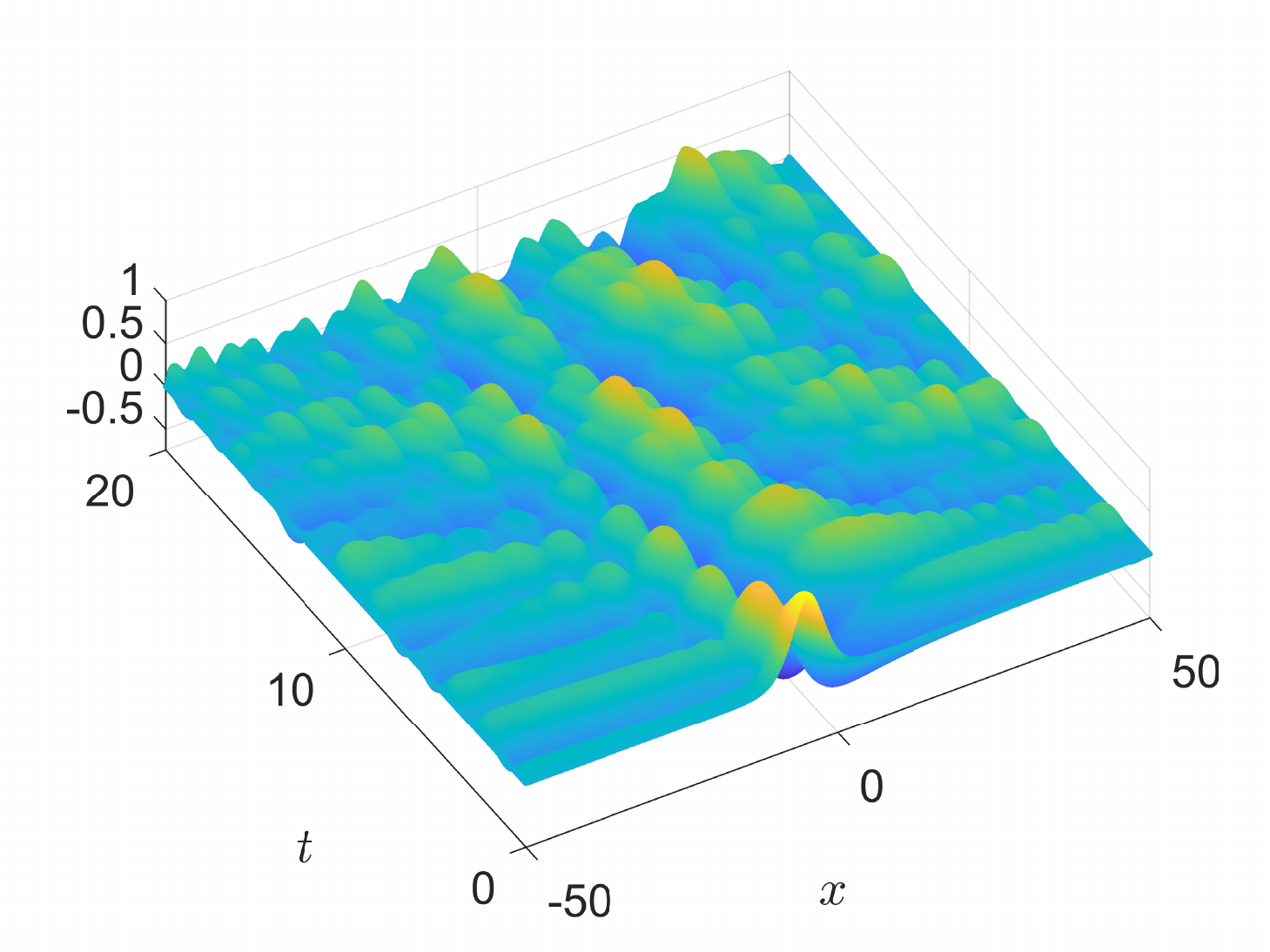}}
  \vspace*{-6pt}
  \caption{\small  Numerical solutions with $N_x=250, N_t=80$ and the initial input $U_0(x)$ given in \eqref{eq:u0-sech}.  (a) Linear KdV equation
  (i.e.,  \eqref{KdVmodelE} with  $\epsilon=1$ and $\alpha=\sigma=0$). (b) KdV-type equation with a nonlocal term
  (i.e.,  \eqref{KdVmodelE} with   $\alpha=0,\epsilon=1$ and $\sigma=0.1$).}
  \label{fig:kdv-sol-sig100}
\end{figure}

Using the accurate space-time spectral solver, we simulate the wave propagations for the linear KdV equation and \eqref{KdVmodelE} with a nonlocal term under the same setting.  We plot the numerical solutions in Figure \ref{fig:kdv-sol-sig100} and find that for the linear KdV equation, the initial wave spreads out only on the left side of the soliton wave of the nonlinear KdV equation (see Figure \ref{fig:kdv-sol00} (left))  and induces some  wiggles.  However, when there is a nonlocal integral term, the initial wave immediately propagates chaotically without a clear pattern.

\section{Concluding remarks}
In this paper, we showed that the eigen-pairs of spectral  discretisation matrices resulted from the LDPG methods for prototype $m$th-order IVPs are associated with the GBPs. More precisely, we were able to exactly characterize the eigenvalues and eigenvectors for $m=1,2$, and approximately characterize the eigenvalues and eigenvectors for $m=3$. We also showed that  by reformulating of the $m$th-order ($m>3$)  IVP as a first-order system, the eigenvalues and eigenvectors of the matrix  can be exactly characterized by   zeros of the GBP: $B_N^{(3)}(z).$

As a by-product, we identified the eigen-pairs of the spectral-collocation differentiation matrix at the Legendre points by reformulating the collocation scheme into a Petrov-Galerkin formulation, and provided answers to some open questions in literature related to this special collocation method.

The findings in this paper have direct implication on developing LDPG spectral methods in time.  As an application, we proposed a general framework to construct space-time spectral methods for a class of time dependent PDEs, and presented two alternatives, matrix diagonalisation which is fully parallel but is limited to a small number of unknowns in time  due to the ill conditioning and QZ decomposition which is stable at large size but involves sequential computations, for their efficient implementation.


\end{document}